\documentclass[a4paper,oneside, 12pt]{article}

\usepackage{amsmath, amsthm, amsfonts, amssymb}
\usepackage[all]{xy}
\usepackage{graphicx}
\usepackage{MnSymbol}
\usepackage{tikz}
\usepackage{hyperref}
\usepackage[nottoc]{tocbibind}
\usepackage{hyperref}
\usetikzlibrary{matrix, arrows,positioning}
\usepackage{float}

\numberwithin{equation}{section}

\newtheorem{theorem}{Theorem}[section]

\newtheorem{lemma}[theorem]{Lemma}
\newtheorem{proposition}[theorem]{Proposition}

\theoremstyle{definition}
\newtheorem{definition}[theorem]{Definition}
\newtheorem{example}[theorem]{Example}
\newtheorem{remark}[theorem]{Remark}

\newcommand{\A}{\alpha}

\newcommand{\C}{\mathbb{C}}
\newcommand{\R}{\mathbb R}
\newcommand{\F}{\mathcal F}
\newcommand{\Or}{\mathfrak{O}}
\newcommand{\Q}{\mathbb Q}
\newcommand{\Z}{\mathbb Z}
\newcommand{\N}{\mathbb N}
\newcommand{\p}{\mathbb P}

\newcommand{\xto}{\xrightarrow}

\newcommand{\rank}{\text{rank}}

\def\url#1{\noindent \sf{#1}}

\begin{document}

\title{\bf The Monodromy problem for hyperelliptic curves}
\author{Daniel L\'opez Garcia}
\date{}
\maketitle
\begin{abstract} 
	We study the Dynkin diagram  associated to the monodromy of direct sums of polynomials. The \textit{monodromy problem} asks under which   conditions on a polynomial,  the monodromy of a vanishing cycle generates the whole homology of a regular fiber. We consider the case $y^4+g(x)$, which is a generalization of the results of Christopher and Marde$\check{s}$i\'c about the monodromy problem for hyperelliptic curves. Moreover, We solve the monodromy problem for direct sums of  fourth degree   polynomials.
\end{abstract}



\section{Introduction}
The interest on  polynomial foliations in $\C^2$ arises as an approach to the Hilbert Sixteen Problem \cite{hosseinabelian, Ro}. These foliations are given by 1-forms $\omega=P(x,y)dy-Q(x,y)dx$, where $P$ and $Q$ are polynomials. The classical notation for the foliation associated to the form $\omega$ is  $\F(\omega)$. In this context, a point $p\in \C^2$  is a singularity of $\F(\omega)$ if $P(p)=Q(p)=0$.   We say that the singularity $p$ is a \textbf{center singularity} if there is a local chart  such that $p$ is mapped to $0\in \C^2$,  and a  non-degenerated function $f:(\C^2,0)\to \C$ with fibers tangent to  the leaves of $\F(\omega)$. The degree of a foliation $\F(\omega)$ is  the greatest degree of the polynomials $P$ and $Q$, and the space of foliations  of degree  $d$ is  denoted by $\F(d)$. The closure of the set of foliations in $\F(d)$   with at least one center  is denote by $\mathcal{Mu}(d)$.

It is known that   $\mathcal{Mu}(d)$ is an algebraic subset of $\F(d)$ (e.g. \cite[\S 6.1]{Netocomponentesbook} \cite{hosseinabelian}). The problem of describing its irreducible components is formulated by Lins Neto \cite{netoFwithcenter}. In \cite{Il},  Y. Ilyashenko  proves that the space of Hamiltonian foliation $\F(df)$, where $f$ is a  polynomial of degree $d+1$, is an irreducible component of $\mathcal{Mu}(d)$. In   \cite{hosseincenterlog}, H. Movasati considers the logarithmic foliations $\F\left(\sum_{i=1}^s\lambda_i\frac{df_i}{f_i}\right)$, with $f_i\in \C[x,y]_{\leq d_i}$ and $\lambda_i\in \C^*$. He  proves that the set of logarithmic foliations  form an irreducible component of $\mathcal{Mu}(d)$, where  $d=\sum_{i=1}^s d_i-1$. Moreover, in  \cite{yadollahcenter}, Y. Zare works with pullback foliations $\F(P^*\omega)$, where $P=(R,S):\C^2\to \C^2$ is a generic morphism with $R,S\in \C[x,y]_{\leq d_1}$, and $\omega$ is a 1-form of degree $d_2$. Zare shows that they form an  irreducible component of $\mathcal{Mu}(d_1( d_2+1)-1)$.

The main idea in the proofs of these assertions is to choose a particular polynomial  $f$ and   consider deformations  $df+\varepsilon \omega_1$ in $\mathcal{Mu}(d)$.  Then, it is necessary to study the vanishing of the abelian integrals $\int_{\delta} \omega_1$, where $\delta$ is a homological 1-cycle in a regular fiber of $f$. This integral is zero on the vanishing cycle associated to the center singularity. If the monodromy action on this cycle generates the whole vector space $H_1(f^{-1}(b),\Q)$ for a regular value $b$, then the deformation is relatively exact to $df$. 

The condition that the vanishing of the integral $\int_{\delta} \omega_1$  implies that $\omega_1$ is relatively exact to $df$, is known as the (*)-\textit{property} (It was introduced  by J.P. Fran\c{c}oise in \cite{Fr}). The results of L. Gavrilov in \cite{gavrilovpetrov}, show that if we provided that the integral is zero over any cycle in a regular fiber, then $\omega_1$ is relatively exact. Therefore, if the subspace generated by monodromy action on the vanishing cycle $\delta$ is the whole space $H_1(f^{-1}(b), \Q)$, then the $(*)-property$ is satisfied.  This gives rise to the next natural problem, which is summarized by C. Christopher and P. Marde$\check{s}$i\'c   in \cite{ChMonodromy}  as follows.\\

\textbf{Monodromy problem}. Under which conditions on $f$ is the $\Q$-subspace of $H_1(f^{-1}(b), \Q)$ generated by the images of a vanishing cycle of a Morse point under monodromy equal to the whole of $H_1(f^{-1}(b), \Q)$?\\

Furthermore, they  show a characterization of the   vanishing cycles associated to a Morse point  in hyperelliptic curves given by $y^2+g(x)$, depending on whether  $g$ is decomposable (Theorem \ref{ChThm}). This case is closely related with the 0-dimensional monodromy problem; by using the definition of Abelian integrals of dimension zero in \cite{GavrilovMovasati}. For example, if we think in the Dynkin diagram associated with $y^2+g(x)$ and the one associated with $g(x)$, then  we  see that they coincide. However, the Dynkin diagram for $ y ^ 3 + g (x) $ is a bit more complicated. Moreover, in the case $y^4+g(x)$ there is always a pullback associated to $y\to y^2$, thus the Dynkin is expected to reflect this fact. For these two cases, we prove the following two theorems.
\begin{theorem}
	\label{Main2S3Intro}
	Let $g$ be a polynomial with real critical points,  and degree $d$ such that,  $4\nmid d$ and $d\leq 100$.  Consider the polynomial  $f(x,y)=y^4+g(x)$ , and let $\delta(t)$ be an associated vanishing cycle at a Morse point; then one of the following assertions holds.
	\begin{enumerate}
		\item The monodromy of $\delta(t)$ generates the homology $H_1(f^{-1}(t), \Q)$.
		\item The polynomial $g$ is decomposable (i.e., $g=g_2\circ g_1$), and $\pi_*\delta(t)$ is homotopic to zero in $\{y^4+g_2(z)=t\}$, where $\pi(x,y)=(g_1(x),y)=(z,y)$. Or,
		the cycle $\pi_*\delta(t)$ is homotopic to zero in $\{z^2+g(x)=t\}$, where $\pi(x,y)=(x,y^2)=(x,z)$.
	\end{enumerate}
\end{theorem} 
\begin{theorem}
	\label{Main1S3Intro}
	
	Let $g$ be a polynomial with real critical points,  and degree $d$ such that,  $3\nmid d$ and $d\leq 100$.  Consider the polynomial  $f(x,y)=y^3+g(x)$ , and let $\delta(t)$ be an associated vanishing cycle at a Morse point; then one of the following assertions holds.
	\begin{enumerate}
		\item The monodromy of $\delta(t)$ generates the homology $H_1(f^{-1}(t), \Q)$.
		\item The polynomial $g$ is decomposable (i.e., $f=g_2\circ g_1$), and $\pi_*\delta(t)$ is homotopic to zero in $\{y^3+g_2(z)=t\}$, where $\pi(x,y)=(g_1(x),y)=(z,y)$.
	\end{enumerate}
\end{theorem}
Some parts in the proof are  done numerically using computer, thus we have the restriction $d\leq 100$ in the degree  of the polynomial $g$.

The monodromy problem for polynomials of degree 4, on the other hand,  is very  interesting, because  the classification of the irreducible components of $\mathcal{Mu}(3)$ is still an open problem. In fact, the only case which has a complete classification is $\mathcal{Mu}(2)$ (see \cite{Dulac}\cite[p. 601]{Netoirreducible}). For polynomials $f(x,y)=h(y)+g(x)$ where $\text{deg}(h)=\text{deg}(g)=4$, we  determine in the Theorem \ref{orbitspacehg}, a relation between the subspaces of $H_1(f^{-1}(b), \Q)$ generated  by the monodromy action of the vanishing cycles, and the property of $f$ being decomposable. In oder to do that, we provide  an explicit description of the space of parameters of the polynomials  $h(y)+g(y)$ which satisfies some conditions in the critical values. \\

\textbf{Organization.} In section 2, we provide some definitions in Picard-Lefschetz theory and describe the Dynkin diagram for direct sum of polynomials in two variable. In section 3, we analyze the particular case, in which there is only one critical value. For this case, we compute the vector space generated by the monodromy action on the vanishing cycles. In section 4, we prove  Theorems \ref{Main2S3Intro} and \ref{Main1S3Intro}. Finally, in Section 5, we solve the monodromy problem for polynomials  $h(x)+g(y)$ with $\text{deg}(h)=\text{deg}(g)=4$, in this case there is another pullback to be considered,  associated to the map $(x,y)\to (xy,x+y)$. For this reason, we do not know a geometrical characterization of some of vanishing cycles which do not generate the whole $H_1(f^{-1}(b), \Q)$.\\

\textbf{Acknowledgements}. I thank Hossein Movasati for suggesting to study the action of monodromy and for introducing me to the center problem. I thank Lubomir Gavrilov for hosting me at the University of Toulouse during a short visit and for his helpful discussions. 

\section{Lefschetz fibrations and monodromy action on direct sum of polynomials }
\label{dynkinrules}
Let $f\in \C[x,y]$ with the set of critical values $C$ and a regular value $b$. Suppose that the origin is an isolated critical point of the highest-grade homogeneous piece of $f$.  Hence, the  Milnor number $\mu$ of $f$ is finite, and  there are \textbf{vanishing cycles} $\delta_1,\delta_2, \ldots, \delta_{\mu}$ associated to the critical values, such that they generate freely the $1$-homology of the fiber $X_b:=f^{-1}(b)$, i.e. $H_{1}(X_b, \Z)=span\{\delta_i\}_{i=1}^{\mu}$ (see   \cite[Chs. 1,2]{singularnold},\cite[\S 7.5] {hodgehossein},\cite{Lamotke}). Moreover, there is an action $\pi_1(\C\setminus C)\times H_1(X_b)\xto{mon}H_1(X_b)$ called the \textbf{monodromy action} given by local trivialization  of  $f^{-1}(\gamma)$, where $\gamma$ is any loop in $\C\setminus C$, see \cite[\S 6.3]{hodgehossein}.  A homological cycle  in $X_b$ such that its orbit by the monodromy action generates the whole homology group $H_1(X_b, \Z)$ is called \textbf{simple cycle}. This definition of simple cycle was introduced in \cite{hosseinabelian}, and it is different from the definition of simple cycle  used in \cite{ChMonodromy,GavrilovMovasati}.

Let  $f(x, y)=h(y)+g(x)$ be a polynomial with real coefficients, such that the critical points of $h$ and $g$ are reals. By considering a deformation  of $f$,  we can suppose that $f$ is a Morse function and  all its critical values are different pairwise.  Furthermore,  by doing a translation  we can suppose that the critical values of $g$ are positive, the critical values of $h$ are negatives, and  $b=0$. Let $c^h_i$ and $c^g_j$ be the critical values of $h$ and $g$, respectively.

Consider the paths $r_i$ and $s_j$,  from $0$ to $c_i^h$ and $c^g_j$, respectively. Furthermore, the paths are without self-intersection,  and they intersect each others only in $0$. Also, $(s_1,s_2,\ldots, s_{d-1}, r_1,r_2,\ldots r_{e-1})$ near $0$ is the anticlockwise direction, as in Figure \ref{setofpaths}. 
 Moreover, we have chosen the enumeration of the paths such that   $0< c^g_1< c^g_2< \cdots< c^g_{d-1}$ and  $c^h_{e-1}> c^h_{e-2}> \cdots > c^h_{1}> 0$.
\begin{figure}[h!]
	\centering
	\begin{tikzpicture}
	\matrix (m) [matrix of math nodes, row sep=0.4em,
	column sep=2em]{
		c^h_{e-1} &c^h_{e-2}&\ldots		&c^h_1&0&c^g_1&c^g_2&\ldots&	c^g_{d-1},\\
	};			
	\path[-stealth]
	(m-1-5) edge [bend left=45] node [above] {$r_{e-1}$}  (m-1-1) 
	edge [bend left=45] node [above] {$r_{e-2}$} (m-1-2)
	edge [bend left=45] node [above] {$r_{1}$}  (m-1-4) 
	
	edge [bend left=45] node [below] {$s_{1}$} (m-1-6)
	edge [bend left=45] node [below] {$s_{2}$} (m-1-7)
	edge [bend left=45] node [below] {$s_{d-1}$} (m-1-9);
	\end{tikzpicture}
		\caption{\footnotesize A distinguished set of paths from $b$ to the critical values of $h$ and $g$}
		\label{setofpaths}
\end{figure}
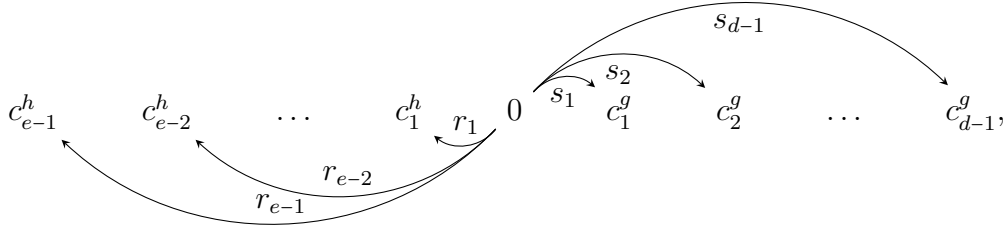  

Let  $C_g=\{(x, g(x))\hspace{2mm}|\hspace{2mm}x\in \R\}$   be the real curve associated to $g$. For each $j=1,\ldots, d-1$, let  $p_j$ be the critical point associated to $c^g_j$, and we define a real number $\varepsilon_j$ as follows. If $p_j$ is a minimum, then we take $c_{j}^g<\varepsilon_j<c_{j+1}^g$, otherwise we take $c_{j-1}^g<\varepsilon_j<c_{j}^g$.
 Consider the real line $L_j=\{(x, \varepsilon_j)\hspace{2mm}|\hspace{2mm}x\in \R\}$, thus there are two points in $L_j\cap C_g$, such that, they go to $(p_j, c^g_j)$ when $\varepsilon_j$ goes to $c^g_j$. Thus, we define the 0-vanishing cycle $\sigma_j\in H_0(g^{-1}(\varepsilon_j), \Z)$, as the formal sum of these two points, with coefficients 1 and -1. Note that these vanishing cycles are \textit{simple cycles} according to the definition in \cite{ChMonodromy,GavrilovMovasati}, however, they are not simple cycles according to our definition.

By taking a simple path from $0$ to $\varepsilon_j$, without encircling or passing through critical values, we can consider $\sigma_j$ in $ H_0(g^{-1}(0), \Z)$. Note, that it is possible  to chose a path from 0 to $\varepsilon_j$ and compose it with the  segment $\overline{\varepsilon_j c_j^g}$, such this path is homotopic equivalent to $s_j$. Therefore,  we can suppose that the cycle $\sigma_j$ vanishes along the path $s_j$. Analogously for the polynomial  $h$, we define  the 0-vanishing cycle $\gamma_{i}\in H_0(h^{-1}(0), \Z)$. 

Consider the path $\lambda=s_jr_i^{-1}$, starting in $c^h_i$ and ending in $c^g_j$, as in \cite[\S 7.9]{hodgehossein}, we define the \textbf{join cycle}
\begin{equation*}
\gamma_i*\sigma_j:=\bigcup_{t\in[0,1]} \gamma_{i }(\lambda_t)\times \sigma_{j}( \lambda_t),
\end{equation*}  
where $\gamma_{i}(\lambda_t)\in  H_0(h^{-1}(\lambda_t))$ and $\sigma_{j }(\lambda_t)\in H_0(g^{-1}(\lambda_t))$. The join cycle $\gamma_i*\sigma_j$ is homeomorphic to a circle $S^1$, the Figure \ref{0Joincycle} shows this construction.
\begin{figure}[h!]
	\centering
	\includegraphics[width=8cm]{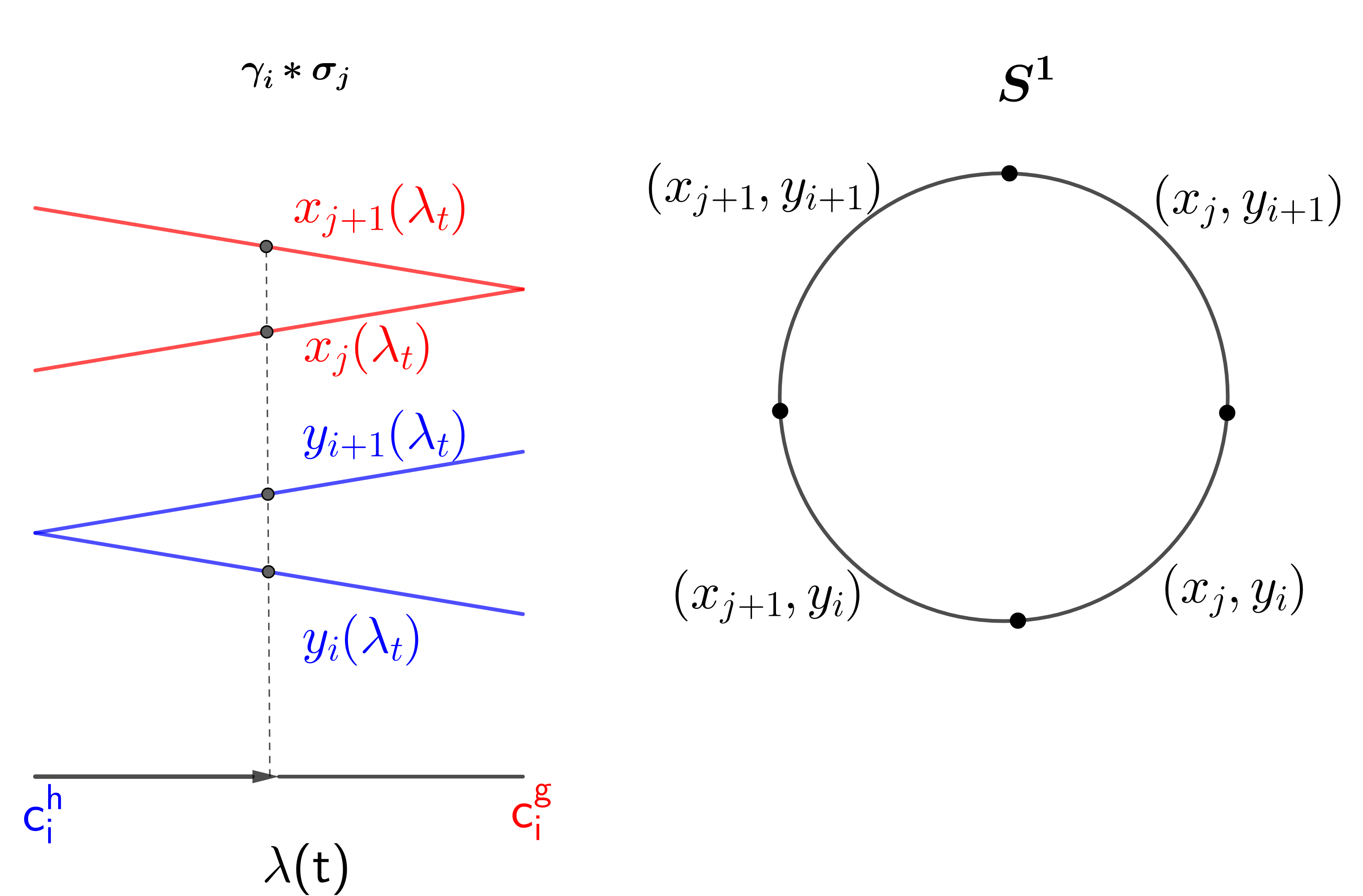}
	\caption{\footnotesize Join cycle $\gamma_j*\sigma_i$ as $S^1$}
	\label{0Joincycle}
\end{figure}
On the other hand, note that the join cycle $\gamma_i*\sigma_j$ is a vanishing cycle of the fibration given by $f$ along the path $s_jr_i^{-1}+c^h_i$. Therefore, the join cycles generate the homology $H_1(f^{-1}(b), \Z)$. Next, we  compute the intersection of two  join cycles. The local  formula for the intersection form of two vanishing cycles is due to A.M. Gabrielov (see \cite[Thm 2.11]{singularnold}). Its reproduction in the global context of tame polynomials is done in \cite[\S 7.10]{hodgehossein}. Since the  particular case of 1-dimension fibers is simple, we reproduce the proof in the next proposition.
\begin{proposition}
	\label{intersectionformula} Let  $\gamma_i*\sigma_j$ and $\gamma_{i'}*\sigma_{j'}$ be two join cycles along the paths  $\lambda:=s_jr_i^{-1}$ and $\lambda':=s_{j'}r_{i'}^{-1}$, respectively. Then
	\begin{equation}
	\label{intersectionofjointcycles}
	\langle \gamma_i*\sigma_j, \gamma_{i'}*\sigma_{j'}\rangle=
	\left\{
	\begin{array}{lcl}
	sgn(j'-j)\langle \sigma_j, \sigma_{j'}\rangle &\text{ if }\hspace{2mm}i=i'\text{ and }j\neq j'\\
	sgn(i'-i)\langle \gamma_i, \gamma_{i'}\rangle &\text{ if }\hspace{2mm}j=j'\text{ and }i\neq i'\\
	sgn(i'-i)\langle \gamma_i, \gamma_{i'}\rangle \langle \sigma_j, \sigma_{j'}\rangle &\hspace{3mm}\text{ if }\hspace{2mm} (i'-i)(j'-j)>0\\
	\hspace{2cm}0 &\hspace{4mm}\text{ if }\hspace{2mm} (i'-i)(j'-j)< 0.
	\end{array}
	\right.
	\end{equation}
\end{proposition}
\begin{proof}
	Suppose that the paths intersect each other transversally in $b$. The join cycle $\gamma_i*\sigma_j$  intersects  $\gamma_{i'}*\sigma_{j'}$ at one point if the 0-cycles  intersect each other, and at zero points otherwise. The orientation of the intersection of the join cycles is given by  $d\lambda\wedge d\lambda'$ times the sign of the intersection of the 0-cycles. Moreover, we consider as positive orientation the canonical orientation  of $\C$ given by $dz_1\wedge dz_2$ where $z=z_1+\sqrt{-1}dz_2$. 
	
	Suppose that $i=i'$, thus the path $\lambda$ and $\lambda'$ intersect transversally  in positive direction if $j'>j$. The intersection of the 0-cycles only depends on the intersection $\langle \sigma_j, \sigma_{j'}\rangle$. Analogously, for $j=j'$. When $(i'-i)(j'-j)>0$ the intersection of $\lambda$ and $\lambda'$ is transversal again, with positive direction if $i'>i$, and the intersection of the 0-cycles is $\langle \gamma_i,\gamma_{i'}\rangle\langle \sigma_j, \sigma_{j'}\rangle$. Finally, if $(i'-i)(j'-j)<0$, the paths do not intersect transversally. Furthermore, after doing a homotopy we can suppose that the path $\lambda$ and $\lambda'$ do not have intersection points, therefore the intersection of the join cycles is zero.
\end{proof}
 Next,  we present a combinatorial way of representing the intersection form, which is described in \cite[\S 2.8]{singularnold} and \cite[\S 7.10]{hodgehossein}.
\begin{definition}
 The \textbf{Dynkin diagram} of  $f(x,y)=h(y)+g(y)$, is a directed graph where the vertices are the vanishing cycles in a regular fiber. The vertices $v_i$ and $v_j$ are joined by an edge with multiplicity $|\langle v_i, v_j\rangle |$.  If $\langle v_i, v_j\rangle>0$, then the direction goes from $v_i$ to $v_j$. 
\end{definition}
We can also relate  the vertex in a Dynkin diagram with the critical value associated to the vanishing cycle. In order to define the Dynkin diagram of the polynomial $f(x,y)=h(y)+g(x)$, we consider a deformation $\tilde f$  such that the critical values are different pairwise. Although the Dynkin diagram for $f$ and $\tilde f$  are equals (the same vertices and edge), the Dynkin diagram associated to $f$ has relations among the vertices according to the equalities of the critical values.  From the previous  choice of  paths $r_i$ and $s_j$, we have rules in the Dynkin diagram which establish the possibilities to relate the critical values:
\begin{enumerate}
	\item The Dynkin diagram associated to $f(x,y)=h(y)+g(x)$, can be thought as a two-dimensional array, where the rows are the critical values $c_i^h+c_j^g$ for a fixed $i$ and $j=1,\ldots, d-1$. 
	Thus, if two critical values of $f(x,y)$  in the same row of the Dynkin diagram are equals, then for the columns of these critical values there are equalities in the rows. This is obviously because if  $c_i^h+c_j^g=c_i^h+c_l^g$ then $c_j^g=c_l^g$, consequently $c_k^h+c_j^g=c_k^h+c_l^g$ for all $k=1\ldots e-1$. This happens in an analogous way for the columns.
	
	
	\item If $c_i^h+c_j^g=c_k^h+c_l^g$ and additionally $i<k$, $j>l$ then $c_i^h=c_k^h$ and $c_j^g=c_l^g$. This follows from the choice of distinguished paths  because $i<k$ implies that $c_i^h \geq c_k^h$ and $j>l$ implies $c_j^h\geq c_l^h $ then $c_i^h+c_j^g\geq c_k^h+c_l^g$ since $c_i^h+c_j^g=c_k^h+c_l^g$ then the inequalities actually are equalities.
\end{enumerate}
Similarly, a Dynkin diagram in dimension 0 consists of vertex which are the vanishing cycles associated to a polynomial, and dashed edges representing an intersection of $-1$, in this case the edges do not have direction. Note that the vanishing cycles associated to the critical values $c^g_i$ with $i\leq [d/2]$ can only intersect  vanishing cycles associated to critical values $c^g_j$ with $j\geq [d/2]$, and similarly for the critical values $c^h_k$, for example, see the figure \ref{realcriticalvalues}.  
\begin{figure}[h!]
	\centering
	\includegraphics[width=0.45\textwidth]{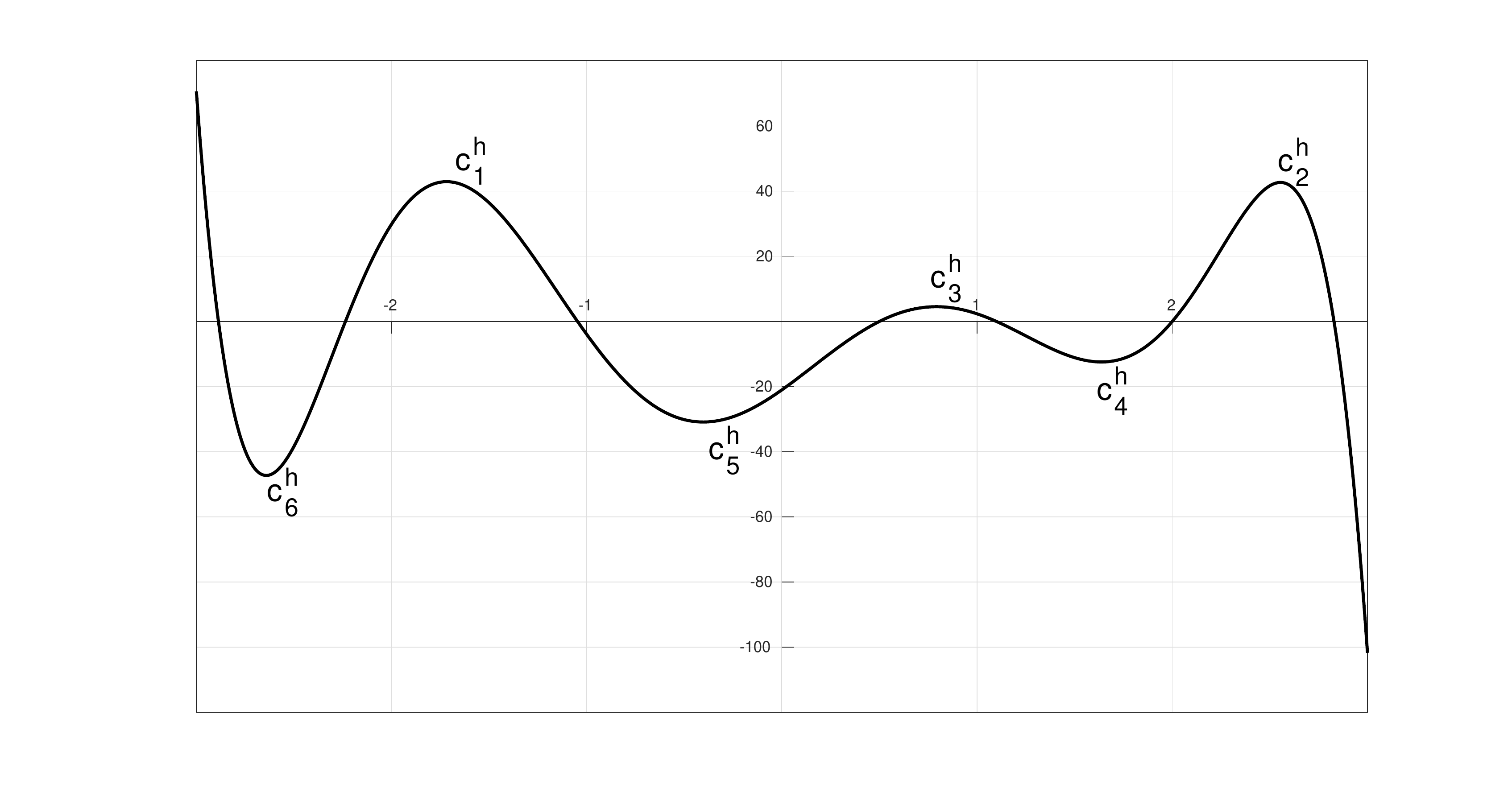}
	\includegraphics[width=0.45\textwidth]{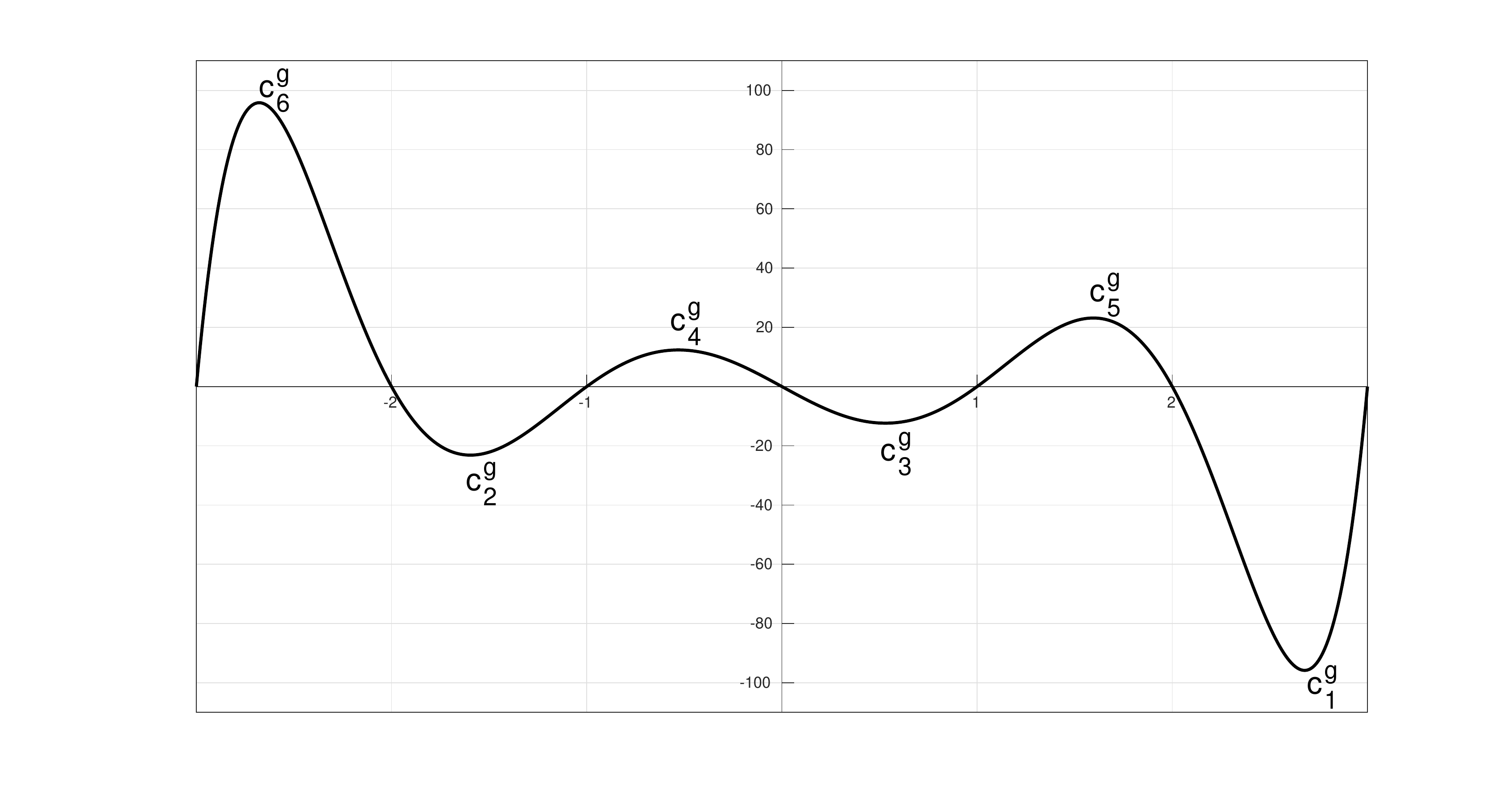}
	\caption{\footnotesize Real part of  the polynomials $h(y)=-(y+\frac{5\sqrt{3}}{3})(y+\sqrt{5})(y+\frac{\pi}{3})(y-\frac{1}{2})(y-\ln(3))(y-2)(y-2\sqrt{2})$ and  $g(x)=(x+3)(x+2)(x+1)x(x-1)(x-2)(x-3)$, with its real critical values. On the left is $h(y)$ and on the right $g(x)$.}
	\label{realcriticalvalues}
\end{figure}

The 0-dimensional Dynkin diagrams associated to $h$ and $g$ of the Figure \ref{realcriticalvalues} are:
\begin{center}
	\begin{tikzpicture}           
	\matrix (m) [matrix of math nodes, row sep=0.7em,
	column sep=0.7em]{
		\gamma_2&\gamma_4 &\gamma_3 & \gamma_5& \gamma_1 & \gamma_6 \\
	};			
	\path[-stealth]
	(m-1-1) edge [-, densely dashed]   (m-1-2)
	(m-1-2)	edge [-, densely dashed]	(m-1-3)
	(m-1-3)	edge [-, densely dashed]	(m-1-4)
	(m-1-4)	edge [-, densely dashed]	(m-1-5)
	(m-1-5)	edge [-, densely dashed]	(m-1-6);
	\end{tikzpicture},
	\begin{tikzpicture}           
	\matrix (m) [matrix of math nodes, row sep=0.5em,
	column sep=0.7em]{
		\sigma_1&\sigma_5 &\sigma_3 & \sigma_4& \sigma_2 & \sigma_6, \\
	};			
	\path[-stealth]
	(m-1-1) edge [-, densely dashed]   (m-1-2)
	(m-1-2)	edge [-, densely dashed]	(m-1-3)
	(m-1-3)	edge [-, densely dashed]	(m-1-4)
	(m-1-4)	edge [-, densely dashed]	(m-1-5)
	(m-1-5)	edge [-, densely dashed]	(m-1-6);
	\end{tikzpicture}
\end{center}
thus, by using   (\ref{intersectionofjointcycles}) we get the next Dynkin diagram for $f(x,y)=h(y)+g(x)$  (in terms of critical values),
\begin{footnotesize}
	\begin{equation}
	\label{Dynkindiagramgenerald}
	\begin{tikzpicture}           
	\matrix (m) [matrix of math nodes, row sep=1em,
	column sep=1em]{
		c^h_2+c^g_1&c^h_4+c^g_1&c^h_3+c^g_1&c^h_5+c^g_1&c^h_1+c^g_1&c^h_6+c^g_1\\
		c^h_2+c^g_5&c^h_4+c^g_5&c^h_3+c^g_5&c^h_5+c^g_5&c^h_1+c^g_5&c^h_6+c^g_5\\
		c^h_2+c^g_3&c^h_4+c^g_3&c^h_3+c^g_3&c^h_5+c^g_3&c^h_1+c^g_3&c^h_6+c^g_3\\
		c^h_2+c^g_4&c^h_4+c^g_4&c^h_3+c^g_4&c^h_5+c^g_4&c^h_1+c^g_4&c^h_6+c^g_4\\
		c^h_2+c^g_2&c^h_4+c^g_2&c^h_3+c^g_2&c^h_5+c^g_2&c^h_1+c^g_2&c^h_6+c^g_2\\
		c^h_2+c^g_6&c^h_4+c^g_6&c^h_3+c^g_6&c^h_5+c^g_6&c^h_1+c^g_6&c^h_6+c^g_6.\\
	};			
	\path[-stealth]
	(m-1-2) edge (m-1-1)
	(m-2-1)	edge (m-1-1)
	(m-1-1) edge (m-2-2)
	(m-2-2) edge (m-2-1)
	(m-2-1) edge (m-3-1)
	(m-4-1)	edge (m-3-1)
	(m-3-2)	edge (m-3-1)
	(m-4-1) edge (m-5-1)
	(m-6-1) edge (m-5-1)
	(m-4-2) edge (m-5-2)
	(m-6-2) edge (m-5-2)
	(m-4-3) edge (m-5-3)
	(m-6-3) edge (m-5-3)
	(m-4-4) edge (m-5-4)
	(m-6-4) edge (m-5-4)
	(m-4-5) edge (m-5-5)
	(m-6-5) edge (m-5-5)
	(m-4-6) edge (m-5-6)
	(m-6-6) edge (m-5-6)
	(m-3-1)	edge (m-2-2)
	(m-3-1)	edge (m-4-2)
	(m-4-2) edge (m-4-1)
	(m-5-2) edge (m-5-1)	
	(m-6-2) edge (m-6-1)
	(m-1-2) edge (m-1-3)
	(m-2-2) edge (m-1-2)
	(m-2-2) edge (m-2-3)
	(m-2-2) edge (m-3-2)
	(m-4-2) edge (m-3-2)
	(m-3-2) edge (m-3-3)
	(m-4-2) edge (m-4-3)
	(m-5-2) edge (m-5-3)
	(m-6-2) edge (m-6-3)
	(m-1-4) edge (m-1-3)
	(m-2-3) edge (m-1-3)
	(m-1-3) edge (m-2-2)
	(m-1-3) edge (m-2-4)
	(m-2-4) edge (m-2-3)
	(m-2-3) edge (m-3-3)
	(m-4-3) edge (m-3-3)
	(m-3-4) edge (m-3-3)
	(m-3-3) edge (m-2-2)
	(m-3-3) edge (m-2-4)
	(m-3-3) edge (m-4-2)
	(m-3-3) edge (m-4-4)
	(m-4-4) edge (m-4-3)
	(m-5-4)	edge (m-5-3)
	(m-6-4)	edge (m-6-3)
	(m-2-4)	edge (m-1-4)
	(m-2-4) edge (m-3-4)
	(m-4-4) edge (m-3-4)
	(m-2-5) edge (m-1-5)
	(m-2-5) edge (m-3-5)
	(m-4-5) edge (m-3-5)
	(m-2-6) edge (m-1-6)
	(m-2-6) edge (m-3-6)
	(m-4-6) edge (m-3-6)
	(m-1-4) edge (m-1-5)
	(m-1-6) edge (m-1-5)
	(m-2-4) edge (m-2-5)
	(m-2-6) edge (m-2-5)
	(m-3-4) edge (m-3-5)
	(m-3-6) edge (m-3-5)
	(m-4-4) edge (m-4-5)
	(m-4-6) edge (m-4-5)
	(m-5-4) edge (m-5-5)
	(m-5-6) edge (m-5-5)
	(m-6-4) edge (m-6-5)
	(m-6-6) edge (m-6-5)
	(m-1-5) edge (m-2-4)
	(m-1-5) edge (m-2-6)
	(m-3-5) edge (m-2-4)
	(m-3-5) edge (m-2-6)
	(m-3-5) edge (m-4-4)
	(m-3-5) edge (m-4-6)
	(m-5-1) edge (m-4-2)
	(m-5-1) edge (m-6-2)
	(m-5-3) edge (m-4-2)
	(m-5-3) edge (m-6-2)
	(m-5-3) edge (m-4-4)
	(m-5-3) edge (m-6-4)
	(m-5-5) edge (m-4-4)
	(m-5-5) edge (m-6-4)
	(m-5-5) edge (m-4-6)
	(m-5-5) edge (m-6-6)
	;
	\end{tikzpicture}
	\end{equation}
\end{footnotesize}
The \textbf{Picard-Lefschetz formula} give us an explicit computation of the monodromy of a cycle $\delta$, around to a critical value $c_{ij}:=c^h_i+c^g_j$. Namely, it is
\begin{equation}
\label{PLformula}
\text{Mon}_{c_{ij}}(\delta)=\delta-\sum_k\langle \delta, \delta_k\rangle\delta_k,
\end{equation}
where $k$ runs through all the join cycles in the singularities of $f^{-1}(c_{ij})$ (see \cite[\S 6.6]{hodgehossein} ).
Therefore, in order to compute the monodromy of the fibration given by the polynomial $f(x,y)=h(y)+g(x)$, we just need to handle combinatorial aspects of Dynkin diagrams. In the remainder of the text, we denote as \textbf{$\text{Mon}(\delta)$}, the subspace generated by the orbit of $\delta$ by monodromy action.

\section{Monodromy for direct sum of polynomials with one critical value}
\label{Monfordirectsumonecrivalue}
In this section, we provide the monodromy matrix around $0$ for the polynomial $f(x,y)=y^{e}+x^{d}$, with $e=2,3,4$. For simplicity, we denote by $\delta^j_{i}$  the vanishing cycles in the row $i$ and column $j$. Thus we have the Dynkin diagram for $e=2,3,4$, 
\begin{center}
	\begin{tikzpicture}           
	\matrix (m) [matrix of math nodes, row sep=1em,
	column sep=1em]{
		\delta^1_1 & \delta^2_1& \delta^3_1&\delta^4_1&\delta^5_1& \delta^6_1 &\cdots &\delta^d_1 \\
	};			
	\path[-stealth]
	(m-1-1) edge (m-1-2)
	(m-1-3) edge (m-1-2)	
	(m-1-3) edge (m-1-4)	
	(m-1-5)	edge (m-1-4)	
	(m-1-7)	edge (m-1-6)
	(m-1-5)	edge (m-1-6)	
	;		
	\end{tikzpicture} 
\end{center}
\begin{center}
	\begin{tikzpicture}           
	\matrix (m) [matrix of math nodes, row sep=1em,
	column sep=1em]{
		\delta^1_1 & \delta^2_1& \delta^3_1&\delta^4_1&\delta^5_1& \delta^6_1 &\cdots &\delta^d_1 \\
		\delta^1_2 & \delta^2_2& \delta^3_2&\delta^4_2&\delta^5_2& \delta^6_2 &\cdots &\delta^d_2 \\
	};			
	\path[-stealth]
	(m-1-1) edge (m-1-2)
	(m-1-3) edge (m-1-2)	
	(m-1-3) edge (m-1-4)	
	(m-1-5)	edge (m-1-4)	
	(m-1-7)	edge (m-1-6)
	(m-1-5)	edge (m-1-6)	
	
	(m-2-1) edge (m-2-2)
	(m-2-3) edge (m-2-2)	
	(m-2-3) edge (m-2-4)	
	(m-2-5)	edge (m-2-4)	
	(m-2-7)	edge (m-2-6)
	(m-2-5)	edge (m-2-6)	
	
	(m-1-1)	edge (m-2-1)
	(m-1-2)	edge (m-2-2)
	(m-1-3)	edge (m-2-3)
	(m-1-4)	edge (m-2-4)	
	(m-1-5)	edge (m-2-5)	
	(m-1-6)	edge (m-2-6)	
	(m-1-8)	edge (m-2-8)
	
	(m-2-2) edge (m-1-1)
	(m-2-2) edge (m-1-3)
	(m-2-4) edge (m-1-3)
	(m-2-4) edge (m-1-5)
	(m-2-6) edge (m-1-5)
	;		
	\end{tikzpicture}
\end{center}
\begin{equation}
\label{dynkindeltaij}
	\begin{tikzpicture}           
	\matrix (m) [matrix of math nodes, row sep=1em,
	column sep=1em]{
		\delta^1_1 & \delta^2_1& \delta^3_1&\delta^4_1&\delta^5_1& \delta^6_1 &\cdots &\delta^d_1 \\
		\delta^1_2 & \delta^2_2& \delta^3_2&\delta^4_2&\delta^5_2& \delta^6_2 &\cdots &\delta^d_2 \\
		\delta^1_3 & \delta^2_3& \delta^3_3&\delta^4_3&\delta^5_3& \delta^6_3 &\cdots &\delta^d_3\ \\
	};			
	\path[-stealth]
	(m-1-1) edge (m-1-2)
	(m-1-3) edge (m-1-2)	
	(m-1-3) edge (m-1-4)	
	(m-1-5)	edge (m-1-4)	
	(m-1-7)	edge (m-1-6)
	(m-1-5)	edge (m-1-6)	
	
	(m-2-1) edge (m-2-2)
	(m-2-3) edge (m-2-2)	
	(m-2-3) edge (m-2-4)	
	(m-2-5)	edge (m-2-4)	
	(m-2-7)	edge (m-2-6)
	(m-2-5)	edge (m-2-6)
	
	(m-3-1) edge (m-3-2)
	(m-3-3) edge (m-3-2)	
	(m-3-3) edge (m-3-4)	
	(m-3-5)	edge (m-3-4)	
	(m-3-7)	edge (m-3-6)
	(m-3-5)	edge (m-3-6)	
	
	(m-1-1)	edge (m-2-1)
	(m-1-2)	edge (m-2-2)
	(m-1-3)	edge (m-2-3)
	(m-1-4)	edge (m-2-4)	
	(m-1-5)	edge (m-2-5)	
	(m-1-6)	edge (m-2-6)	
	(m-1-8)	edge (m-2-8)
	
	(m-3-1)	edge (m-2-1)
	(m-3-2)	edge (m-2-2)
	(m-3-3)	edge (m-2-3)
	(m-3-4)	edge (m-2-4)	
	(m-3-5)	edge (m-2-5)	
	(m-3-6)	edge (m-2-6)	
	(m-3-8)	edge (m-2-8)
	
	(m-2-2) edge (m-1-1)
	(m-2-2) edge (m-1-3)
	(m-2-4) edge (m-1-3)
	(m-2-4) edge (m-1-5)
	(m-2-6) edge (m-1-5)
	
	(m-2-2) edge (m-3-1)
	(m-2-2) edge (m-3-3)
	(m-2-4) edge (m-3-3)
	(m-2-4) edge (m-3-5)
	(m-2-6) edge (m-3-5)
	;		
	\end{tikzpicture},
\end{equation}
respectively. By Proposition \ref{intersectionformula}, the intersection matrix in the ordered vector basis \newline
$\delta^1_1, \ldots, \delta^1_e, \delta^2_1,\ldots, \delta^2_e,\ldots, \delta^d_1,\ldots, \delta^d_e$ for these Dynkin diagram are
\begin{equation}
\label{matrixM2}
\Psi_2=\begin{pmatrix}
0&-1&0&0&\ldots\\
1&0&1&0&\ldots\\
0&-1&0&-1&\ldots\\
0&0&1&0&\ldots\\
\vdots&\vdots&\vdots&\vdots&\\
\end{pmatrix}, 
\end{equation}
\begin{equation}
\label{matrixM3}
\Psi_3=\begin{pmatrix}
0&-1&-1&1&0&0&0&0&\ldots\\
1&0&0&-1&0&0&0&0&\ldots\\
1&0&0&-1&1&0&0&0&\ldots\\
-1&1&1&0&-1&1&0&0&\ldots\\
0&0&-1&1&0&-1&-1&1&\ldots\\
0&0&0&-1&1&0&0&-1&\ldots\\
0&0&0&0&1&0&0&-1&\ldots\\
0&0&0&0&-1&1&1&0&\ldots\\
\vdots&\vdots&\vdots&\vdots&\vdots&\vdots&\vdots&\vdots&\\
\end{pmatrix},
\end{equation}	
\begin{equation}
\label{matrixM4}
\Psi_4=\left(\begin{array} {@{}*{13}{c}@{}}
0&-1&0&-1&1&0&0&0&0&0&0&0&\ldots\\
1&0&1&0&-1&0&0&0&0&0&0&0&\ldots\\
0&-1&0&0&1&-1&0&0&0&0&0&0&\ldots\\
1&0&0&0&-1&0&1&0&0&0&0&0&\ldots\\
-1&1&-1&1&0&1&-1&1&-1&0&0&0&\ldots\\
0&0&1&0&-1&0&0&0&1&0&0&0&\ldots\\
0&0&0&-1&1&0&0&-1&0&-1&1&0&\ldots\\
0&0&0&0&-1&0&1&0&1&0&-1&0&\ldots\\
0&0&0&0&1&-1&0&-1&0&0&1&-1&\ldots\\
0&0&0&0&0&0&1&0&0&0&-1&0&\ldots\\
0&0&0&0&0&0&-1&1&-1&1&0&1&\ldots\\
0&0&0&0&0&0&0&0&1&0&-1&0&\ldots\\
\vdots&\vdots&\vdots&\vdots&\vdots&\vdots&\vdots&\vdots&\vdots&\vdots&\vdots&\vdots&\\
\end{array}\right).
\end{equation}
The matrices are antisymmetric, and the superior diagonals are periodic sequences.  For $\Psi_2$ the sequence is $(-1, 1,\ldots)$. For $\Psi_3$ the sequence are $(-1, 0, -1, -1, -1, 0, \ldots)$, $(-1, -1, 1, 1\ldots)$ and $(1, 0, 0, 0,\ldots)$. For $\Psi_4$ the sequence are $(-1,1,0, \ldots)$, $(0, 0, 1, 0 , -1, 0, \ldots)$, $(-1, -1, -1, 1, 1, 1,\ldots)$ and $(-1, 0, 0, 0, 1, 0,\ldots)$. From the Picard-Lefschetz formula (\ref{PLformula}), it follows that the monodromy matrices for $f(x,y)=y^e+x^d$ with $e=2,3,4$, are $$M_e=I_N-\Psi_e,$$
where $I_N$ is the identity matrix of rank $N=(d-1)(e-1)$.

For a vector $v$ and a matrix $M\in \mathcal{M}_N(\R)$, the Krylov subspace is the vectorial space generated by the vectors $M^lv$ where $l=0,2, \ldots,N-1$. Therefore, by taking $M$ as one of the monodromy matrices $M_2$, $M_3$ or $M_4$, and $v=v_k$ a vector of the canonical basis of $\R^N$, the Krylov subspace is 
\begin{equation}
\label{Mondeltaonecri}
\text{Mon}\left(\delta_{\mod_e(k)}^{ \lfloor{\frac{k}{e}} \rfloor }\right)
\end{equation}
for the fibration $y^e+x^d$. In the next proposition we provide the vanishing cycles  that are in (\ref{Mondeltaonecri}).
\begin{proposition}
	\label{freevanishingcycles}
	For the polynomial $y^2+x^d$,  the vanishing cycles in the subspace $\text{Mon}(\delta_i^j) $ are 
	\begin{table}[H]
		\footnotesize
		\centering
		\begin{tabular}{|l|l|}
			\hline
			Vanishing cycle $\delta_1^j$ & Vanishing cycles  $\delta_1^l$ in  $\text{Mon}(\delta_1^j)$\\
			\hline
			\hline	
			$gcd(d,j)=r$& $l=r n$ with $n=1,\ldots, \frac{d}{r}-1$.\\
			\hline	
		\end{tabular}
	\end{table} 
	\noindent For the polynomial $y^3+x^d$ with $d\leq 100$ and $3\nmid d$, the vanishing cycles in the subspace $\text{Mon}(\delta_i^j) $ are 
	\begin{table}[H]
		\footnotesize
		\centering
		\begin{tabular}{|l|l|}
			\hline
			Vanishing cycle $\delta_i^j$ & Vanishing cycles  $\delta_m^l$ in  $\text{Mon}(\delta_i^j)$\\
			\hline
			\hline			
			$i=1,2$ and $gcd(d,j)=r$&$m=1,2$ and $l=r n$ with $n=1,\ldots, \frac{d}{r}-1.$\\
			\hline		
		\end{tabular}
	\end{table}
	\noindent When $3\mid d$, the number of different eigenvalues is less than $2(d-1)$. For the polynomial $y^4+x^d$ with $d\leq 100$ and $4\nmid d$, the vanishing cycles in the subspace $\text{Mon}(\delta_i^j) $ are 
	\begin{table}[H]
		\footnotesize
		\centering
		\begin{tabular}{|l|l|}
			\hline
			Vanishing cycle $\delta_i^j$ & Vanishing cycles  $\delta_m^l$ in  $\text{Mon}(\delta_i^j)$\\
			\hline
			\hline
			$i=1,3$ and $gcd(d,j)=r$&$m=1,2,3$ and $l=r n$ with $n=1,\ldots, \frac{d}{r}-1$\\
			\hline		
			$i=2$ and $gcd(d,j)=r$&$m=2$ and $l=r n$ with $n=1,\ldots, \frac{d}{r}-1$.\\
			\hline	
		\end{tabular}
	\end{table}
	\noindent When$4\mid d$, the number of different eigenvalues is less than $3(d-1)$.
\end{proposition}
\begin{proof}
	Let $M$ be one of the matrices $M_2, M_3$ or $M_4$, and $v:=v_k=(0,0, \ldots,0,1,0\ldots, 0)$ for $k=1,\ldots, (d-1)(e-1)$. The corresponding  vanishing cycle to $v_k$ is $\delta_a^b$, with $a=\mod_e(k)$ and $b=\lfloor{\frac{k}{e}} \rfloor$. Since the matrices $\Psi_e$ are skew-symmetric, then $M$ is a normal matrix, consequently it is diagonalizable. Hence, its eigenvectors $u_j$ are a basis for $\R^N$. Then we can write $v=\sum_j r_j u_j$ for scalars $r_j$. Let $\lambda_j$ be the eigenvalue associated to $u_j$, thus we have \begin{equation*}
	M^lv=\sum_{j=1}^{N} r_j \lambda_j^l u_j,\hspace{4mm}\text{ where } N=(d-1)(e-1),
	\end{equation*}
	and the matrix $\{v, Mv, M^2v,\ldots, M^nv\}$ is $$\begin{pmatrix}
	r_1 u_1&r_2u_2&\ldots&r_Nu_n
	\end{pmatrix}
	\begin{pmatrix}
	1&\lambda_1&\lambda_1^2&\ldots&\lambda_1^{N-1}\\
	1&\lambda_2&\lambda_2^2&\ldots&\lambda_2^{N-1}\\
	\vdots&\vdots&\vdots&&\vdots\\
	1&\lambda_N&\lambda_N^2&\ldots&\lambda_N^{N-1}\\
	\end{pmatrix}.$$
	The matrix in the right, is the Vandermonde matrix with determinant $\prod_{i<j}(\lambda_j-\lambda_i)$. Hence, if the eigenvalues are different, then the Krylov subspace is the span of the vectors $u_l$ such that $r_l\neq 0$.
	
	For $e=2$, it is possible to show that the matrix $M_2$ is similar to a tridiagonal  matrix with main diagonal of 1s, first diagonal below of -1s, and first diagonal above of 1s. The change of basis is given by the diagonal matrix, whose diagonal is $(-1, 1, 1,-1,-1,1, 1, \ldots)$. Hence, there is a known closed form for the eigenvalues and eigenvectors of the matrix $M_2$ (see \cite{eigenvaluestri}). Namely, the eigenvalues are given by 
	\begin{equation*}
	\lambda_j=1+2\sqrt{-1}\cos\left(\frac{j\pi}{d}\right),\hspace{4mm}\text{ whit }j=1,\ldots, d-1.
	\end{equation*}
	If the vector $u_j=\left(u_j^{(1)}, u_j^{(2)},\ldots, u_j^{(d-1)}\right)^{T}$ is the eigenvector associated to $\lambda_j$,  then the $k-$th coordinate  satisfies
	\begin{equation*}
	u_j^{(k)}=(\sqrt{-1})^{k-1}\sqrt{\frac{2}{d}}\sin\left(\frac{kj\pi}{d}\right).
	\end{equation*}
	If we denote  $U=[u_1\hspace{4mm}u_2\hspace{4mm}\cdots\hspace{4mm} u_{d-1}]$ the matrix whose columns are the eigenvalues, then $UU^*=Id(d-1)$. Hence, if we want to know which eigenvectors are used in the representation of $\delta_1^l$, it is enough to note which terms in the row $l$ of $U$ are zero. This happens when $\frac{jl}{d}\in \Z$. Furthermore, the Krylov space of $\delta_1^l$ is contained in the Krylov space of $\delta_1^{l'}$, provided that the $j$'s such that $\frac{jl'}{d}\in \Z$, satisfy $\frac{jl}{d}\in \Z$. It is equivalent to $gcd(d,l') \mid gcd(d,l)$. 
	
	For $e=3,4$, we do not know a close form for the eigenvalues. However,  for given values of $d$, on a computer we can compute explicitly the eigenvalues, and a basis for the subspace generated by  these eigenvectors.  Next, we determine which vectors of the canonical base $\R^N$ are in this subspace.	If $e=3$ and $3\mid d$, then the number of different eigenvalues is less than $N$. The same is true for $e=4$ and $4\mid d$. In other cases the number of different eigenvalues is $N$. The reader can use the functions written in MATLAB, MonMatrix and VanCycleSub \footnote{https://github.com/danfelmath/Intersection-matrix-for-polynomials-with-1-crit-value.git}, for a numerical supplement of this proof (see \S \ref{numericalsection}). 
\end{proof}
\begin{remark}
	The condition $3\nmid d$ in the case  $M_3$  may be related with the fact that $y^3+x^d$ is a pullback with the map $(x,y)\to (x^{\frac{d}{3}}, y)$. Analogously, the condition $4\nmid d$ in the case $M_4$ associated to $y^4+x^d$.	

	On the other hand, in general a monodromy matrices is not diagonalizable. For example, the monodromy matrices of the mirror quintic Calabi-Yau threefold (see \cite{doranmorgan, DanielLopezquintic})
\end{remark}

\section{Monodromy problem for $y^4+g(x)$}
\label{maintheoremDynkin}
Let $g\in \R[x]_{\leq d}$ be a polynomial with real critical points. Consider the polynomial $f(x,y):=y^e+g(x)$ which  has critical values equal to the critical values of $g$. Recall, in some cases we relate the vertices in the Dynkin diagram to the critical values associated with the vanishing cycles. Thus, we denote by $C_j$ the critical value in the column $j$ from left to right in the Dynkin diagram, and $\delta^j_{i}$ to the vanishing cycle in the row $i$ over $C_j$. For example, if we suppose that $d$ is even and $C_1$ is a local maximum, then the Dynkin diagram looks like
\begin{equation}
\label{Dynkindiagramrealcriticalvalues}
\begin{tikzpicture}           
\matrix (m) [matrix of math nodes, row sep=1em,
column sep=0.7em]{
	C_{1} & C_{2}& C_{3}&C_{4}&C_{5}& C_6 &\cdots & C_{d-2}& C_{d-1} \\
	C_{1} & C_{2}& C_{3}&C_{4}&C_{5}& C_6 &\cdots & C_{d-2}& C_{d-1} \\
	C_{1} & C_{2}& C_{3}&C_{4}&C_{5}& C_6 &\cdots & C_{d-2}& C_{d-1} \\
};			
\path[-stealth]
(m-1-1) edge (m-1-2)
(m-1-3) edge (m-1-2)	
(m-1-3) edge (m-1-4)	
(m-1-5)	edge (m-1-4)	
(m-1-7)	edge (m-1-6)
(m-1-9)	edge (m-1-8)	
(m-1-5)	edge (m-1-6)	
(m-1-7)	edge (m-1-8)

(m-2-1) edge (m-2-2)
(m-2-3) edge (m-2-2)	
(m-2-3) edge (m-2-4)	
(m-2-5)	edge (m-2-4)	
(m-2-7)	edge (m-2-6)
(m-2-9)	edge (m-2-8)	
(m-2-5)	edge (m-2-6)	
(m-2-7)	edge (m-2-8)

(m-3-1) edge (m-3-2)
(m-3-3) edge (m-3-2)	
(m-3-3) edge (m-3-4)	
(m-3-5)	edge (m-3-4)	
(m-3-7)	edge (m-3-6)
(m-3-9)	edge (m-3-8)	
(m-3-5)	edge (m-3-6)	
(m-3-7)	edge (m-3-8)

(m-1-1)	edge (m-2-1)
(m-1-2)	edge (m-2-2)
(m-1-3)	edge (m-2-3)
(m-1-4)	edge (m-2-4)	
(m-1-5)	edge (m-2-5)	
(m-1-6)	edge (m-2-6)	
(m-1-8)	edge (m-2-8)	
(m-1-9)	edge (m-2-9)	

(m-3-1)	edge (m-2-1)
(m-3-2)	edge (m-2-2)
(m-3-3)	edge (m-2-3)
(m-3-4)	edge (m-2-4)	
(m-3-5)	edge (m-2-5)	
(m-3-6)	edge (m-2-6)	
(m-3-8)	edge (m-2-8)	
(m-3-9)	edge (m-2-9)

(m-2-2) edge (m-1-1)
(m-2-2) edge (m-1-3)
(m-2-4) edge (m-1-3)
(m-2-4) edge (m-1-5)
(m-2-6) edge (m-1-5)
(m-2-8) edge (m-1-9)

(m-2-2) edge (m-3-1)
(m-2-2) edge (m-3-3)
(m-2-4) edge (m-3-3)
(m-2-4) edge (m-3-5)
(m-2-6) edge (m-3-5)
(m-2-8) edge (m-3-9)
;		
\end{tikzpicture}. 
\end{equation}
\begin{definition}
	\label{def:horsym}
	We say that the Dynkin diagram of $y^e+g(x)$ with $g\in\C[x]_{\leq d}$ has \textbf{horizontal symmetry} if there exits integer $r>1$ such that for any $j$ with $\text{g.c.d}(j,d)=r$   the critical values  satisfy $$C_{j-k}=C_{j+k} \text { where } k=1,\ldots, r-1.$$
	The vanishing cycles $\delta_i^{l\cdot r}$ with $l=1,\ldots, \frac{d}{r}-1$ are called \textbf{vanishing cycles with horizontal symmetry}. We can define the  \textbf{vertical symmetry} analogously.For the Dynkin diagram (\ref{Dynkindiagramrealcriticalvalues}) the cycles $\delta_2^j$  are \textbf{vanishing cycles with vertical symmetry}.
\end{definition}
From a direct computation in the Dynkin diagram (\ref{Dynkindiagramrealcriticalvalues})and Picard-Lefschetz formula,  we observe that only the terms $$\delta_i^{j-k}+\delta_i^{j+k}\text{ with }\text{g.c.d}(j,d)=r$$ 
and the cycles with horizontal symmetry appear in the subspace generated by the monodromy action on a cycle  with horizontal
symmetry. This happens in an analogous way for the vertical symmetry. Therefore, the subspace generated by the monodromy action on a cycle  with horizontal
symmetry or vertical symmetry is different to $H_1(f^{-1}(b),\Q)$. 

On the other hand, the definition of horizontal symmetry only depend on the relation among the critical values of $g$. Hence, for $p,q$ integers greater than 1, there are cycles with horizontal symmetry in the Dynkin diagram associated to $y^p+g(x)$ if and only if there are  in the Dynkin diagram associated to  $y^q+g(x)$.

For the vertical symmetry, in  the next lemma we provide a geometric characterization of  the cycles $\delta_2^j$ with $j=1,\ldots, d-1$. 
\begin{lemma}
	\label{ciclovertical}
	Consider the map $\C^2\xto{\pi} \C^2$, given by $\pi(x,y)=(x,y^2)$. The cycles $\delta_2^j\in H_1((y^4+g(x))^{-1}(b))$ for $j=1,\ldots, d-1$, are in the kernel of $$\pi_*: H_1((y^4+g(x))^{-1}(b))\to H_1((y^2+g(x))^{-1}(b)).$$
\end{lemma}
\begin{proof}
	Consider the perturbation $h_{\varepsilon}(y):=y^4+\varepsilon(-y^2+\frac{\varepsilon}{8})$ of $y^4$, where $\varepsilon\geq 0$. The roots of $h_{\varepsilon}(y)$ are $\frac{\pm 1}{2}\sqrt{\varepsilon (2\pm \sqrt{2})}$. Therefore, the 0-cycle associated to $\delta_2^j$, is $$\gamma_1=\left(\frac{1}{2}\sqrt{\varepsilon (2- \sqrt{2})},0\right)-\left(\frac{-1}{2}\sqrt{\varepsilon (2- \sqrt{2})},0\right)$$
	(see Figure \ref{perturbationy}). In the projection by $y^2$, these points are identified with  $\left(\frac{1}{4}\varepsilon(2-\sqrt{2}), 0\right)$. Consequently, the image of the  vanishing cycles $\delta_2^j=\gamma_1*\sigma_j$ by the map $\pi$ is trivial.
\end{proof}

Note that the kernel of $\pi_*$ is generated by the cycles  $$\gamma_1*\sigma_j\hspace{4mm}\text{ and }\hspace{4mm} 
(\gamma_3-\gamma_2)*\sigma_j, \hspace{4mm}\text{ for } j=1,\ldots d-1,$$ however, the first ones generate the others by monodromy.
\begin{figure}[h!]
	\centering
	\includegraphics[width=0.6\textwidth]{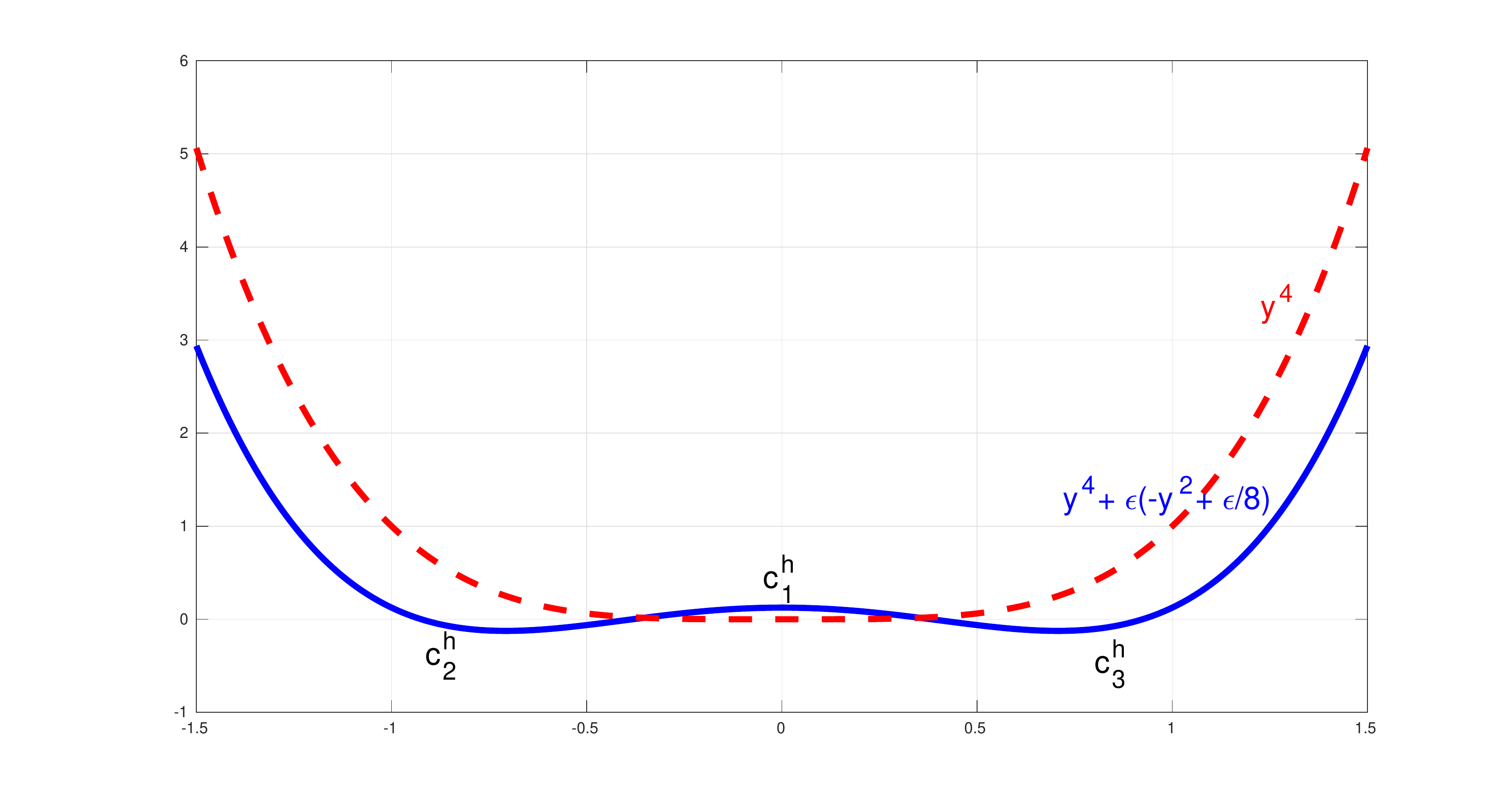}
	\caption{\footnotesize Critical values for $y^4+\varepsilon(-y^2+\frac{\varepsilon}{8})$, a perturbation of $y^4$.}
	\label{perturbationy}
\end{figure}

We want to study the vanishing cycles in the subspace generated by the monodromy action when there is no horizontal symmetry in the Dynkin diagram. By using Lemma \ref{monodromysubgroup}, we can reduce this analysis to the cases $\pi_1(\p^1\setminus C)\equiv \Z$ and $\pi_1(\p^1\setminus C)\equiv \Z^2$. 
\begin{lemma}
	\label{monodromysubgroup}
	Let $g$ be a polynomial of degree $d$,  and let $G_1$ and $G_g$ be the monodromy groups associated to $y^4+x^d$ and $y^4+g(x)$, respectively. For any $v\in H_1((y^4+g)^{-1}(b),\Z)$ the subspaces generated by the orbits satisfies $\langle G_1\cdot v\rangle \subset \langle G_g\cdot v\rangle$. Besides, if $C$ has more than one element, then there exist a group of two elements $G_2<Aut(V_g)$ such that $\langle G_2\cdot v\rangle \subset \langle G_g\cdot v\rangle$.
\end{lemma}
\begin{proof}
	Any element in $G_g$ can be written as a matrix $I_N-A_j\in \mathcal M_{3(d-1)}(\R)$. Moreover, this $A_j$ is constructed by putting  rows of zeros in the matrix $\Psi_4$ of equation (\ref{matrixM4}). Thus, we have that $M_4=\sum_{A_j\in G_g}I_N-A_j+(1-|G_g|)I_N$. Consequently, $M_4^k v\in \langle G_g\cdot v\rangle$ for $k\in \Z$. 
	
	In order  to construct $G_2$  is enough to divide in two groups the elements of $G_g$, and define two matrix as the sum of the matrices in these groups. Note that these sums correspond with  identifications of some critical values in the Dynkin diagram.
\end{proof}
 	The next proposition follows from the proposition \ref{freevanishingcycles} and lemma \ref{monodromysubgroup}.
\begin{proposition}
	\label{mainPropDynkin}
	Let $g$ be a polynomial of degree $d$, where $d\leq 100$ and $4\nmid d$. If the Dynkin diagram of $f(x,y)=y^4+g(x)$ does not have  horizontal symmetry, then the subspace of $H_1(f^{-1}(b),\Q )$ generated by the orbit of a vanishing cycle $\delta^j_{i}$ contains  all the vanishing cycles in the row $i$. Moreover, if $i$ is 1 or 3, then the submodule generated is whole space  $H_1(f^{-1}(b),\Q )$.
\end{proposition}
\begin{proof}
	The restrictions on the degree $d$  are due to the Proposition \ref{freevanishingcycles}. By Proposition \ref{freevanishingcycles}  the result is true for $gcd(j,d)=1$ and $g(x)=x^d$, then by Lemma \ref{monodromysubgroup} is true for any $g\in \C[x]_{\leq d}$. If $gcd(j,d)=r>1$, then $\text{Mon}(\delta_i^j)$ contains  the vanishing cycles $\delta_i^l$ with $l=rn$ and $n=1,\ldots, \frac{d}{r}-1$. Since the rows do not have horizontal symmetry, then $d$ is prime or there are at least two different critical values. However if $d$ is prime, then $r=1$. Hence,  by using lemma \ref{monodromysubgroup} we suppose that there are two critical values $A$ and $B$. Thus, it is enough to consider a initial vanishing cycle $v:=\delta_{2}^j$ where the critical values $C_{j-l}$ and $C_{l+j}$ are equal for $l=1,\ldots, k-1<r-1$. Also, the critical values $C_{j-k}$ and $C_{j+k}$ are different. We can suppose that the Dynkin diagram looks like 
	\begin{center}
		\begin{tikzpicture}           
		\matrix (m) [matrix of math nodes, row sep=1em,
		column sep=0.7em]{
			* &\cdots & C_{j-k-1}& C_{j-k}&C_{j-k+1}&\cdots &A& A &A&\cdots& C_{k-1+j}& C_{k+j}&*&\cdots &* \\
			* &\cdots & D& B&A&\cdots &A& A &A&\cdots& A& A&*&\cdots &* \\
			* &\cdots & D& B&A&\cdots &A& A &A&\cdots& A& A&*&\cdots &* \\	
		};			
		\path[-stealth]
		(m-1-1) edge (m-1-2)
		(m-1-3) edge (m-1-2)	
		(m-1-3) edge (m-1-4)	
		(m-1-5)	edge (m-1-4)	
		(m-1-7)	edge (m-1-6)
		(m-1-9)	edge (m-1-8)	
		(m-1-5)	edge (m-1-6)	
		(m-1-7)	edge (m-1-8)
		(m-1-9)	edge (m-1-10)
		(m-1-11) edge (m-1-10)
		(m-1-11) edge (m-1-12)
		(m-1-13) edge (m-1-12)
		(m-1-13) edge (m-1-14)
		(m-1-15) edge (m-1-14)
		
		(m-2-1) edge (m-2-2)
		(m-2-3) edge (m-2-2)	
		(m-2-3) edge (m-2-4)	
		(m-2-5)	edge (m-2-4)	
		(m-2-7)	edge (m-2-6)
		(m-2-9)	edge (m-2-8)	
		(m-2-5)	edge (m-2-6)	
		(m-2-7)	edge (m-2-8)
		(m-2-9)	edge (m-2-10)
		(m-2-11) edge (m-2-10)
		(m-2-11) edge (m-2-12)
		(m-2-13) edge (m-2-12)
		(m-2-13) edge (m-2-14)
		(m-2-15) edge (m-2-14)
		
		(m-3-1) edge (m-3-2)
		(m-3-3) edge (m-3-2)	
		(m-3-3) edge (m-3-4)	
		(m-3-5)	edge (m-3-4)	
		(m-3-7)	edge (m-3-6)
		(m-3-9)	edge (m-3-8)	
		(m-3-5)	edge (m-3-6)	
		(m-3-7)	edge (m-3-8)
		(m-3-9)	edge (m-3-10)
		(m-3-11) edge (m-3-10)
		(m-3-11) edge (m-3-12)
		(m-3-13) edge (m-3-12)
		(m-3-13) edge (m-3-14)
		(m-3-15) edge (m-3-14)
		
		(m-1-1)	edge(m-2-1)
		(m-1-3)	edge (m-2-3)
		(m-1-4)	edge (m-2-4)	
		(m-1-5)	edge (m-2-5)	
		(m-1-7)	edge (m-2-7)
		(m-1-8)	edge (m-2-8)	
		(m-1-9)	edge (m-2-9)	
		(m-1-11) edge (m-2-11)
		(m-1-12) edge (m-2-12)
		(m-1-13) edge (m-2-13)
		(m-1-15) edge (m-2-15)
		
		(m-3-1)	edge(m-2-1)
		(m-3-3)	edge (m-2-3)
		(m-3-4)	edge (m-2-4)	
		(m-3-5)	edge (m-2-5)	
		(m-3-7)	edge (m-2-7)
		(m-3-8)	edge (m-2-8)	
		(m-3-9)	edge (m-2-9)	
		(m-3-11) edge (m-2-11)
		(m-3-12) edge (m-2-12)
		(m-3-13) edge (m-2-13)
		(m-3-15) edge (m-2-15)
		
		(m-2-4) edge (m-1-3)
		(m-2-4) edge (m-1-5)
		(m-2-8) edge (m-1-7)
		(m-2-8) edge (m-1-9)	
		(m-2-12) edge (m-1-11)
		(m-2-12) edge (m-1-13)
		
		(m-2-4) edge (m-3-3)
		(m-2-4) edge (m-3-5)
		(m-2-8) edge (m-3-7)
		(m-2-8) edge (m-3-9)	
		(m-2-12) edge (m-3-11)
		(m-2-12) edge (m-3-13)
		;		
		\end{tikzpicture}. 
	\end{center}	
	where $D$ can be $A$ or $B$, and $*$ means that no matter what value it is. We denote by $\text{Mon}_A$ and $\text{Mon}_B$ the monodromy action around to the critical values $A$ and $B$, respectively. By doing $\text{Mon}_B(\text{Mon}_A)^{k-1}(v)$, we get  a linear combination of cycles in the column $k$. In fact we have one of the next possibilities
	$$(m_2m_1)^s v, m_1(m_2m_1)^s v, (m_1m_2)^s v, m_2(m_1m_2)^s v, \text{ }s\in \N,$$ where
	$$m_1=\begin{pmatrix}
	-1&1&0\\
	0&-1&0\\
	0&1&-1
	\end{pmatrix}\text{ , }m_2=\begin{pmatrix}
	1&0&0\\
	-1&1&-1\\
	0&0&1
	\end{pmatrix}
	$$ and the matrices are in the basis $\delta^{j-k}_1, \delta^{j-k}_2, \delta^{j-k}_3$. Hence, we generate the linear combination  $w:=m  \delta^{j-k}_1+n\delta^{j-k}_2+m\delta^{j-k}_3\text{ where }m\in \Z \text{ and }n\in \Z^*.$
	If $D=B$,  taking $\text{Mon}_B\text{Mon}_A(w)$, we get \begin{equation*}
	(n-3m)\delta^{j-k}_1+(2m-n)\delta^{j-k}_2+(n-3m)\delta^{j-k}_3\text{, or }
	\end{equation*} \begin{equation*}
	(n-m)\delta^{j-k}_1+(2m-3n)\delta^{j-k}_2+(n-m)\delta^{j-k}_3.
	\end{equation*}
	Any of the linear combinations in the previous equation and $w$ generate the vanishing cycle $\delta^{j-k}_2$. 
	If $D=A$, considering  $\text{Mon}_B(w)$  and $w$ we also generate the cycle $\delta^{j-k}_2$. If $gcd(j-k,d)=1$, then the results follows from proposition \ref{freevanishingcycles}. If  $gcd(j-k,d)=r'$, then we repeat the previous analysis with $r'$ instead of $r$, thus the proof follows from $r'<r$.
\end{proof}
The  next propositions  are proved with  analogous arguments as in the proof of Proposition \ref{mainPropDynkin}. For this reason, in their proof we only indicate the corresponding matrices $m_1$ and $m_2$ .
\begin{proposition}
	\label{CoroDynkin3}
	Let $g$  be a polynomial of degree $d$, where $d\leq 100$ and $3\nmid d$. If the Dynkin diagram of $f(x,y)=y^3+g(x)$ does not have  horizontal symmetry, then the subspace generated by the orbit of a vanishing cycle $\delta^j_{i}$ is the whole space $H_1(f^{-1}(b),\Q )$.
\end{proposition}
\begin{proof}
	Consider the matrices  $$m_1=\begin{pmatrix}
	-1&1\\
	0&-1
	\end{pmatrix}\text{ , }m_2=\begin{pmatrix}
	1&0\\
	-1&1
	\end{pmatrix},$$ thus considering the initial vanishing cycle $v:=\delta_2^j$, we get a vector $w=m\delta_1^j+n\delta_2^j$ where $m\in \Z$ and $n\in \Z^*$.
\end{proof}
\begin{proposition}
	\label{CoroDynkin2}
If the Dynkin diagram of $f(x,y)=y^2+g(x)$ does not have  horizontal symmetry, then the subspace generated by the orbit of a vanishing cycle $\delta^j_{i}$ is the whole space $H_1(f^{-1}(b),\Q )$.
\end{proposition}
\begin{proof}
	In this case the matrices are $m_1=-1$ and $m_2=1$. 
\end{proof}
The next theorem is the main result in \cite{ChMonodromy}, it is a solution of the monodromy problem for hyperelliptic curves $y^2+g(x)$. We will use it, in order to solve the monodromy problem for $y^3+g(x)$ and $y^4+g(x)$. Although, this theorem holds for $g\in \C[x]$, we are interested in the case of $g$ being a real polynomial with real critical points.
\begin{theorem}[C. Christopher and P. Marde$\check{s}$i\'c, 2008]
	\label{ChThm}
	Let $f(x,y)=y^2+g(x)$, and let $\delta(t)$ be an associated vanishing cycle at a Morse point; then one of the following assertions holds.
	\begin{enumerate}
		\item The monodromy of $\delta(t)$  generates the whole homology $H_1(f^{-1}(b), \Q)$.
		\item The polynomial $g$ is decomposable (i.e., $g=g_2\circ g_1$), and $\pi_*\delta(t)$ is homotopic to zero in $\{y^2+g_2(z)=t\}$, where $\pi(x,y)=(g_1(x),y)=(z,y)$.
	\end{enumerate}
\end{theorem} 	
From this theorem, we can show that the condition of $g$ being decomposable as $g=g_2\circ g_1$, is equivalent to the Dynkin diagram associated to $y^e+g(x)$ has horizontal symmetry, with $e>1$. In fact, as we mentioned above, the horizontal symmetry condition only depend on $g$. 
\begin{proposition}
	\label{pullbackhorsymm}
	The next assertions are equivalents.
	\begin{enumerate}
		\item The polynomial  can be written as $g=g_2\circ g_1$, where $g_1, g_2$ are polynomials such that $\deg (g_1), \deg(g_2)>1$.
		\item The Dynkin diagram associated to $y^e+g(x)$ has horizontal symmetry, for some $e>1$ (and hence for all $e>1$).
	\end{enumerate}
\end{proposition}
\begin{proof}From Proposition \ref{CoroDynkin2}, follows that the condition 2 implies that there are vanishing cycles for the fibration $f(x,y)=y^2+g(x)$, such that they do not generates  the whole $H_1(f^{-1}(b))$. Thus, by the Theorem \ref{ChThm}, we conclude that 2 implies 1. The other implication  follows by a direct computation on a Dynkin diagram similar to  \ref{Dynkindiagramrealcriticalvalues}, but  with $e-1$ rows. 
\end{proof}
Although the horizontal symmetry is just a condition on the polynomial $g$, this proposition allows to extend  the result in the Theorem \ref{ChThm} to the fibrations defined by $y^4+g(x)$ and $y^3+g(x)$. The non trivial part for this generalization are due to the   Propositions \ref{mainPropDynkin} and \ref{CoroDynkin3}.  However, since  Proposition \ref{freevanishingcycles} is  numerically proven, we have restrictions in the degree of $g$. 	 
\begin{theorem}
	\label{Main2S3}
	Let $g$ be a polynomial with real critical points,  and degree $d$ such that,  $4\nmid d$ and $d\leq 100$. Consider the polynomial  $f(x,y)=y^4+g(x)$ , and let $\delta(t)$ be an associated vanishing cycle at a Morse point; then one of the following assertions holds.
	\begin{enumerate}
		\item The monodromy of $\delta(t)$ generates the homology $H_1(f^{-1}(t), \Q)$.
		\item The polynomial $g$ is decomposable (i.e., $g=g_2\circ g_1$), and $\pi_*\delta(t)$ is homotopic to zero in $\{y^4+g_2(z)=t\}$, where $\pi(x,y)=(g_1(x),y)=(z,y)$. Or,
		the cycle $\pi_*\delta(t)$ is homotopic to zero in $\{z^2+g(x)=t\}$, where $\pi(x,y)=(x,y^2)=(x,z)$.
	\end{enumerate}
\end{theorem} 
\begin{proof}
	The restrictions on the degree $d$  are due to the Proposition \ref{freevanishingcycles}, which is used in proposition \ref{mainPropDynkin}. Let $\delta_i^j:=\delta(t)$ be a vanishing cycle. If the monodromy of $\delta_i^j$ does not generate the homology $H_1(f^{-1}(t), \Q)$, then  considering the contrapositive of Proposition \ref{mainPropDynkin}, we have the next possibilities: The index $i$ is $2$ or the cycle $\delta_i^j$ has horizontal symmetry. If $i=2$, then by the Lemma \ref{ciclovertical}, we have that $\pi_*\delta_2^j$ is trivial, where $\pi(x,y)=(x,y^2)$. 	If $\delta_i^j$ has horizontal symmetry, then by using Proposition \ref{pullbackhorsymm} we conclude that $g=g_2\circ g_1$. Furthermore, $\delta_i^j$  is in correspondence with a cycle with horizontal symmetry in the Dynkin diagram of $y^2+g(x)$. Consequently,  $\delta_i^j$ is in the kernel of $\pi_*$, where $\pi(x,y)=(g_1(x),y)$.

	On the other hand, if the condition 2 is true, then the vanishing cycle  $\delta_i^j$ has  vertical or horizontal symmetry. In any of these cases, it follows  by direct computation in the Dynkin diagram \ref{Dynkindiagramrealcriticalvalues},
	that  the subspace generated by the orbit of the monodromy action on $\delta_i^j$ is different to $H_1(f^{-1}(t), \Q)$.
\end{proof}
We have an analogous result for degree $y^3+g(x)$.  Note that  in this case there are not cycles with vertical symmetry. 
\begin{theorem}
	\label{Main1S3}	
	Let $g$ be a polynomial with real critical points,  and degree $d$ such that,  $3\nmid d$ and $d\leq 100$. Consider the polynomial  $f(x,y)=y^3+g(x)$ , and let $\delta(t)$ be an associated vanishing cycle at a Morse point; then one of the following assertions holds.
	\begin{enumerate}
		\item The monodromy of $\delta(t)$ generates the homology $H_1(f^{-1}(t), \Q)$.
		\item The polynomial $g$ is decomposable (i.e., $f=g_2\circ g_1$), and $\pi_*\delta(t)$ is homotopic to zero in $\{y^3+g_2(z)=t\}$, where $\pi(x,y)=(g_1(x),y))=(z,y)$.
	\end{enumerate}
\end{theorem} 
\begin{proof}
	The restrictions on the degree $d$  are due to the Proposition \ref{freevanishingcycles}, which is used in proposition \ref{CoroDynkin3}. 
	Let $\delta_i^j:=\delta(t)$ be a vanishing cycle. If the monodromy of $\delta_i^j$ does not generate the homology $H_1(f^{-1}(t), \Q)$, then  by  Proposition \ref{CoroDynkin3}, the cycle   $\delta_i^j$ has horizontal symmetry. Hence, by using Proposition \ref{pullbackhorsymm} we conclude that $g=g_2\circ g_1$.
\end{proof}

\section{Monodromy problem for 4th degree polynomials \newline $h(y)+g(x)$}
\label{Smonodromydegree4}
Consider $f(x,y)=h(y)+g(x)$ where $h\in\R[y]_{\leq 4}$ and $g\in \R[x]_{\leq 4}$, and $b$ is a regular value. Moreover, we suppose that the critical points of $h$ and $g$ are reals. The aim of this section is to compute the part of the homology $H_1(f^{-1}(b))$ generated by the action of the monodromy. From the equation (\ref{intersectionofjointcycles}) it follows that  the 1-dimensional Dynkin diagram depends on the 0-dimensional Dynkin diagrams of $h$ and $g$. Let $\gamma_i\in H_0(h^{-1}(b), \Z)$ and $\sigma_i\in H_0(g^{-1}(b), \Z)$ be the 0-cycles, where $i=1,2,3$.  Thus, using the enumeration of vanishing cycles, indicated in $\S$ \ref{dynkinrules}, we have the next three cases,
\begin{center}
	\begin{tikzpicture}           
	\matrix (m) [matrix of math nodes, row sep=0.7em,
	column sep=0.7em]{
		\gamma_1&\gamma_3 &\gamma_2\\
	};			
	\path[-stealth]
	(m-1-2) edge [-, densely dashed]   (m-1-1)
	edge [-, densely dashed]	(m-1-3);
	\end{tikzpicture},
	\begin{tikzpicture}           
	\matrix (m) [matrix of math nodes, row sep=0.5em,
	column sep=0.7em]{
		\sigma_2&\sigma_1&\sigma_3\\
	};			
	\path[-stealth]
	(m-1-2) edge [-, densely dashed]   (m-1-1)
	edge [-, densely dashed]   (m-1-3);
	\end{tikzpicture}resulting in: 
	\begin{tikzpicture}           
	\matrix (m) [matrix of math nodes, row sep=0.5em,
	column sep=1em]{
		\gamma_{1}*\sigma_{2}&\gamma_{3}*\sigma_{2}&\gamma_{2}*\sigma_{2}\\
		\gamma_{1}*\sigma_{1}&\gamma_{3}*\sigma_{1}&\gamma_{2}*\sigma_{1}\\
		\gamma_{1}*\sigma_{3}&\gamma_{3}*\sigma_{3}&\gamma_{2}*\sigma_{3}\\
	};			
	\path[-stealth]
	(m-1-2) edge (m-1-1)
	(m-1-2) edge (m-1-3)		
	(m-1-1) edge (m-2-1)
	(m-2-2)	edge (m-2-1)
	(m-3-1) edge (m-2-1)
	(m-2-1)	edge (m-1-2)
	(m-2-1) edge (m-3-2)
	(m-1-2) edge (m-2-2)
	(m-3-2) edge (m-2-2)
	(m-1-3) edge (m-2-3)
	(m-2-2) edge (m-2-3)
	(m-3-3) edge (m-2-3)
	(m-2-3) edge (m-1-2)
	(m-2-3) edge (m-3-2)
	(m-3-2) edge (m-3-1)		
	(m-3-2) edge (m-3-3);
	\end{tikzpicture}
\end{center}
\begin{center}
	\begin{tikzpicture}           
	\matrix (m) [matrix of math nodes, row sep=0.5em,
	column sep=0.7em]{
		\gamma_1&\gamma_3&\gamma_2\\
	};			
	\path[-stealth]
	(m-1-2) edge [-, densely dashed]   (m-1-1)
	edge [-, densely dashed]	(m-1-3);
	\end{tikzpicture},
	\begin{tikzpicture}           
	\matrix (m) [matrix of math nodes, row sep=0.5em,
	column sep=0.7em]{
		\sigma_1&\sigma_3 &\sigma_2\\
	};			
	\path[-stealth]
	(m-1-2) edge [-, densely dashed]  (m-1-1)
	edge [-, densely dashed]  (m-1-3);
	\end{tikzpicture} resulting in:
	\begin{tikzpicture}           
	\matrix (m) [matrix of math nodes, row sep=0.5em,
	column sep=1em]{
		\gamma_{1}*\sigma_{1}&\gamma_{3}*\sigma_{1}&\gamma_{2}*\sigma_{1}\\
		\gamma_{1}*\sigma_{3}&\gamma_{3}*\sigma_{3}&\gamma_{2}*\sigma_{3}\\
		\gamma_{1}*\sigma_{2}&\gamma_{3}*\sigma_{2}&\gamma_{2}*\sigma_{2}\\
	};			
	\path[-stealth]
	(m-2-1) edge (m-1-1)
	(m-1-2) edge (m-1-1)
	(m-1-1) edge (m-2-2)
	(m-2-2) edge (m-1-2)
	(m-2-3) edge (m-1-3)		
	(m-1-2) edge (m-1-3)		
	(m-1-3) edge (m-2-2)
	(m-2-1) edge (m-3-1)
	(m-3-2) edge (m-3-1)
	(m-3-1) edge (m-2-2)
	(m-2-2) edge (m-3-2)
	(m-2-3) edge (m-3-3)		
	(m-3-2) edge (m-3-3)		
	(m-3-3) edge (m-2-2)
	(m-2-2) edge (m-2-1)
	(m-2-2) edge (m-2-3)
	;
	\end{tikzpicture}
\end{center}
\begin{center}
	\begin{tikzpicture}           
	\matrix (m) [matrix of math nodes, row sep=0.7em,
	column sep=0.7em]{
		\gamma_2&\gamma_1& \gamma_3&\\
	};			
	\path[-stealth]
	(m-1-2) edge [-, densely dashed] (m-1-1)
	edge [-, densely dashed]	(m-1-3);
	\end{tikzpicture}, 
	\begin{tikzpicture}           
	\matrix (m) [matrix of math nodes, row sep=0.5em,
	column sep=0.7em]{
		\sigma_2 &\sigma_1& \sigma_3&\\
	};			
	\path[-stealth]
	(m-1-2) edge [-, densely dashed] (m-1-1)
	edge [-, densely dashed] (m-1-3);
	\end{tikzpicture} resulting in:
	\begin{tikzpicture}           
	\matrix (m) [matrix of math nodes, row sep=0.5em,
	column sep=1em]{
		\gamma_{2}*\sigma_{2}&\gamma_{1}*\sigma_{2}&\gamma_{3}*\sigma_{2}\\
		\gamma_{2}*\sigma_{1}&\gamma_{1}*\sigma_{1}&\gamma_{3}*\sigma_{1}\\
		\gamma_{2}*\sigma_{3}&\gamma_{1}*\sigma_{3}&\gamma_{3}*\sigma_{3}.\\
	};			
	\path[-stealth]
	(m-1-1) edge (m-2-1)
	(m-1-1) edge (m-1-2)
	(m-2-2) edge (m-1-1)
	(m-1-2) edge (m-2-2)
	(m-1-3) edge (m-2-3)		
	(m-1-3) edge (m-1-2)		
	(m-2-2) edge (m-1-3)
	(m-3-1) edge (m-2-1)
	(m-3-1) edge (m-3-2)
	(m-2-2) edge (m-3-1)
	(m-3-2) edge (m-2-2)
	(m-3-3) edge (m-2-3)		
	(m-3-3) edge (m-3-2)		
	(m-2-2) edge (m-3-3)
	(m-2-1) edge (m-2-2)
	(m-2-3) edge (m-2-2)
	;
	\end{tikzpicture} 
\end{center}
If we consider $-f$ instead of $f$, the two last Dynkin diagram coincide. Hence, we only focus in the first two 1-dimensional Dynkin diagrams. 
\begin{example}
	\label{exampleh+g}
	Consider the polynomials $h(y)=-y^4+9y^2$ and $g(x)=-x^4+16x^2+8x$. In Figure \ref{hg}, we show the real part of this polynomials with the critical values indexed according  to \S \ref{dynkinrules}. Let $f_1(x,y)=h(y)+g(x)$, thus the Dynkin diagram associated to $f_1$ is the first one.
	\begin{figure}[h!]
		\centering
		\includegraphics[width=0.3\textwidth]{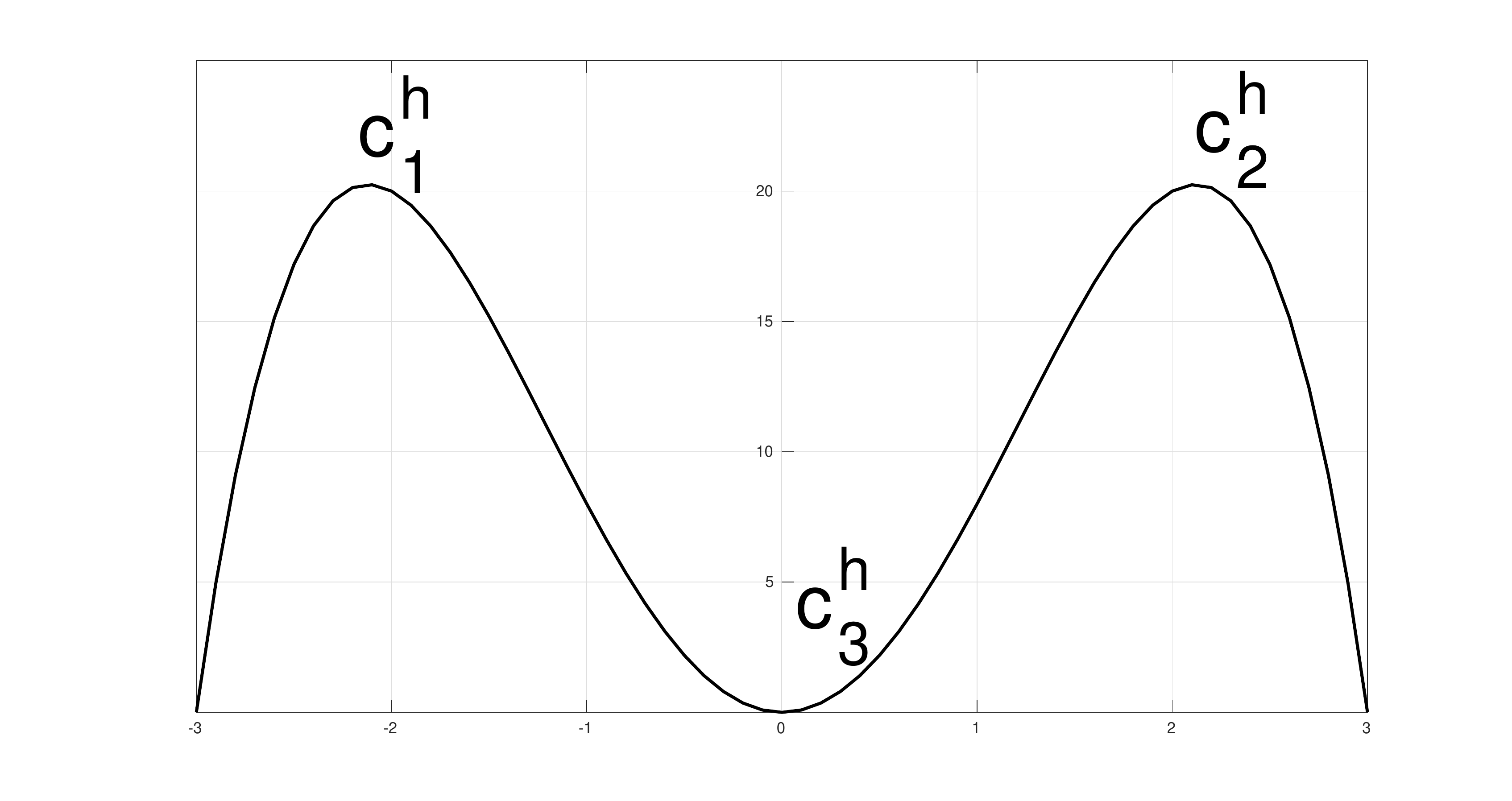}
		\includegraphics[width=0.3\textwidth]{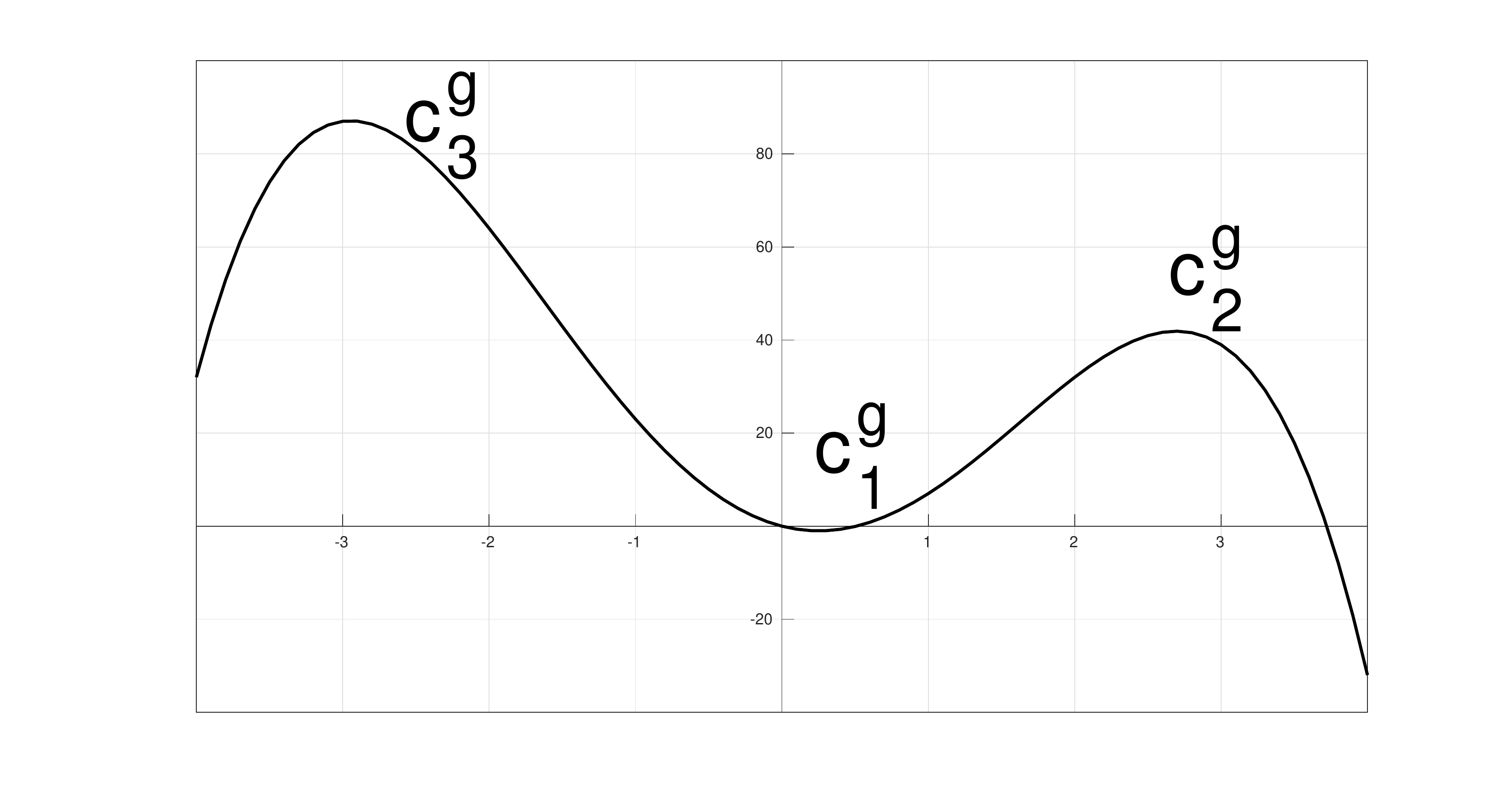}
		\caption{\footnotesize Real part of  the polynomials $h(y)=-y^4+9y^2$ and $g(x)=-x^4+16x^2+8x$, with its critical values. On the left is $h(y)$ and on the right $g(x)$.}
		\label{hg}
	\end{figure} 
\end{example}
\begin{example} 
	\label{exampleh-g}
	Consider the polynomials $h(y)=-y^4+9y^2$ and $g(x)=x^4-16x^2-8x$. In Figure \ref{h-g}, we present the real part of this polynomials with the critical values indexed according  to \S \ref{dynkinrules}. Let $f_2(x,y)=h(y)+g(x)$, thus the Dynkin diagram associated to $f_2$ is the second.
	\begin{figure}[h!]
		\centering
		\includegraphics[width=0.3\textwidth]{h.pdf}
		\includegraphics[width=0.3\textwidth]{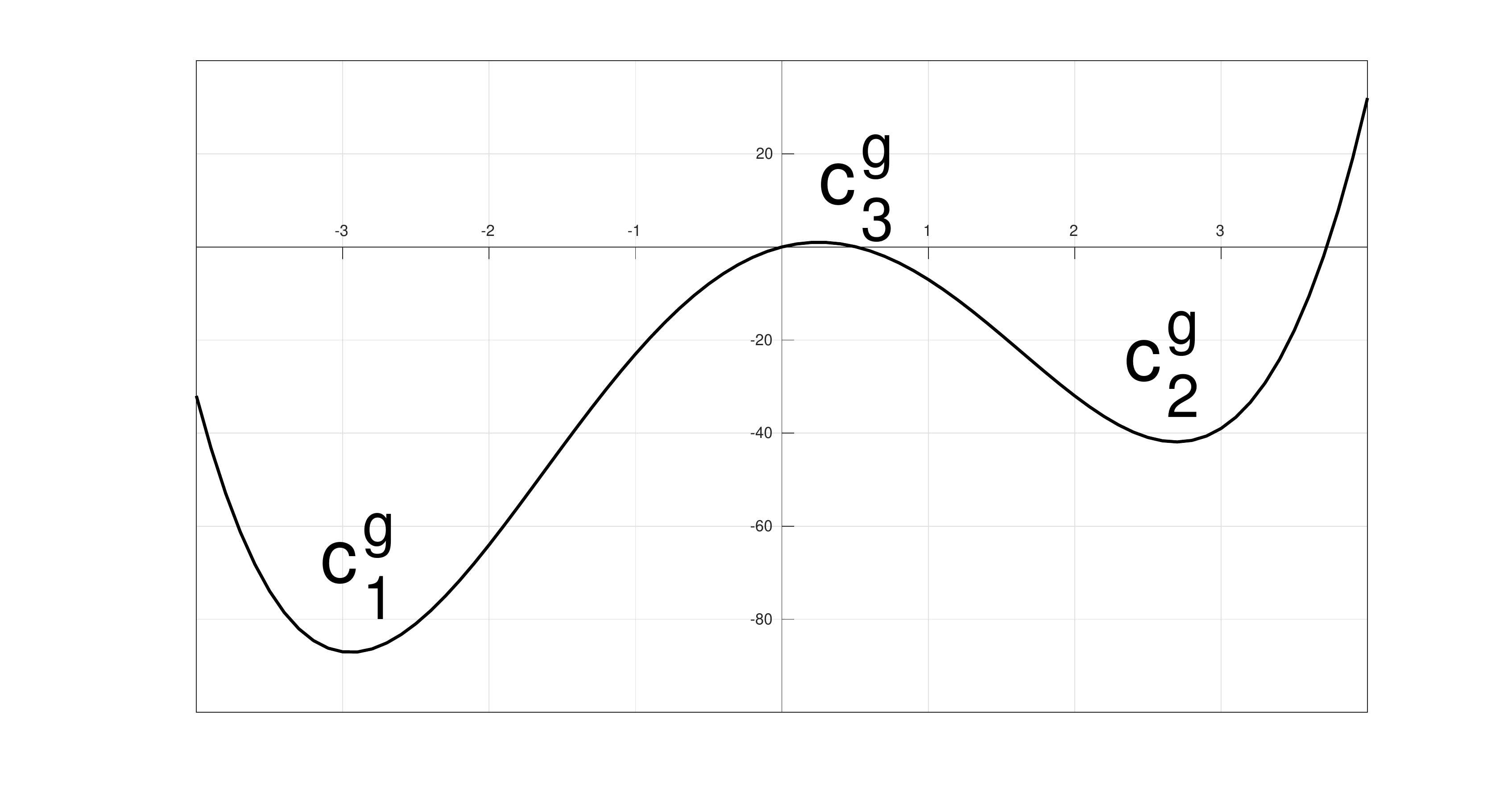}
		\caption{\footnotesize Real part of  the polynomials $h(y)=-y^4+9y^2$ and $g(x)=x^4-16x^2-8x$, with its critical values. On the left is $h(y)$ and on the right $g(x)$.}
		\label{h-g}
	\end{figure}   
\end{example}
In Figures \ref{h+g_3d} and  \ref{h-g_3d}, we present the real part of the fibration $f_1(x,y)=t$ and $f_2(x,y)=t$, respectively. Note that the maximum corresponds with the addition of the maximum of $h$ and $g$, analogously for the minimum. The others critical points are known as saddles points.
\begin{figure}[h!]
	\centering
	\includegraphics[width=0.7\textwidth]{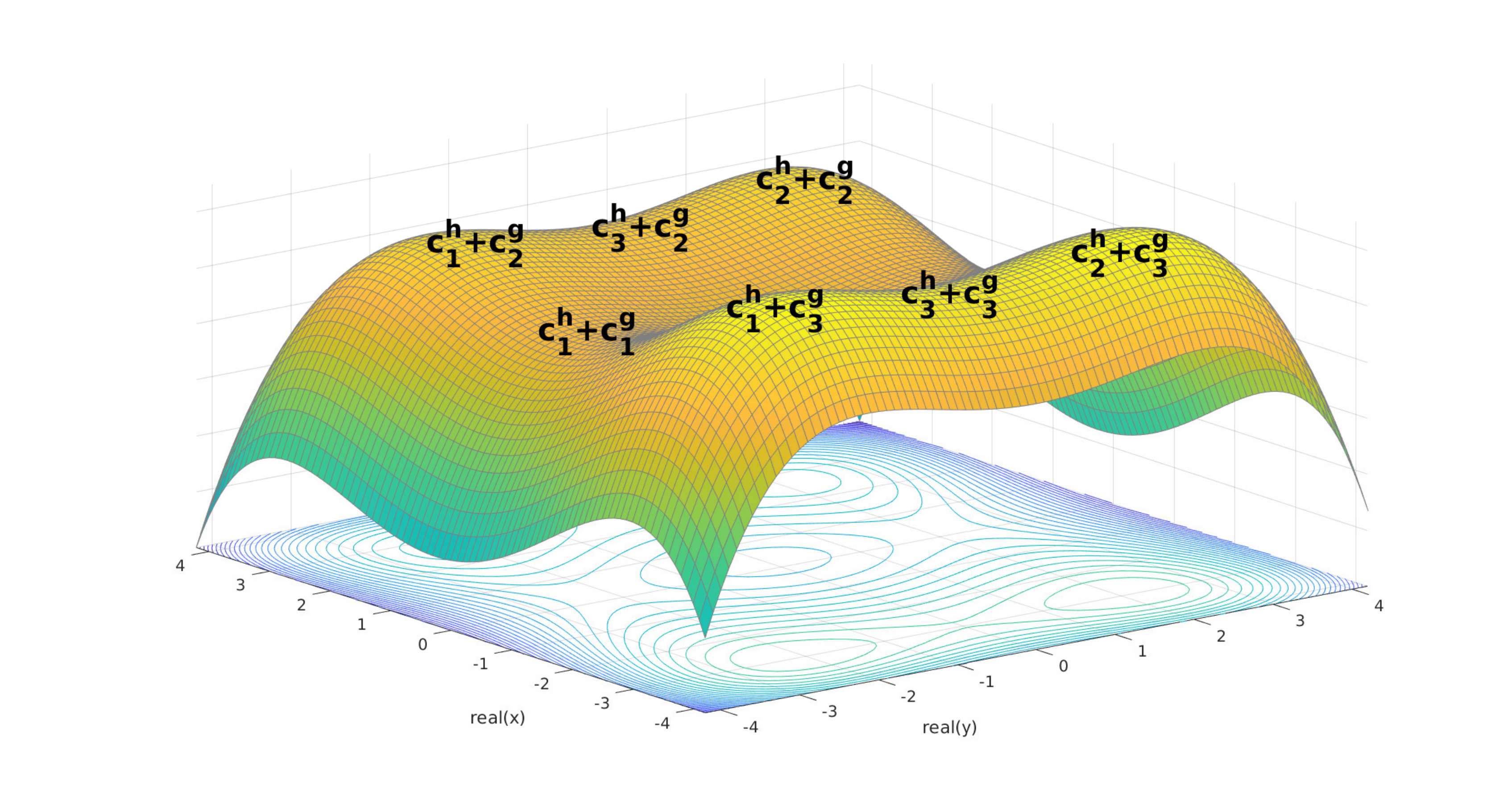}
	\caption{\footnotesize  Real part of graph defined by $f_1(x,y)=-y^4+9y^2-x^4+16x^2+8x=t$, with its critical values.}
	\label{h+g_3d}
\end{figure}
\begin{figure}[h!]
	\centering
	\includegraphics[width=0.7\textwidth]{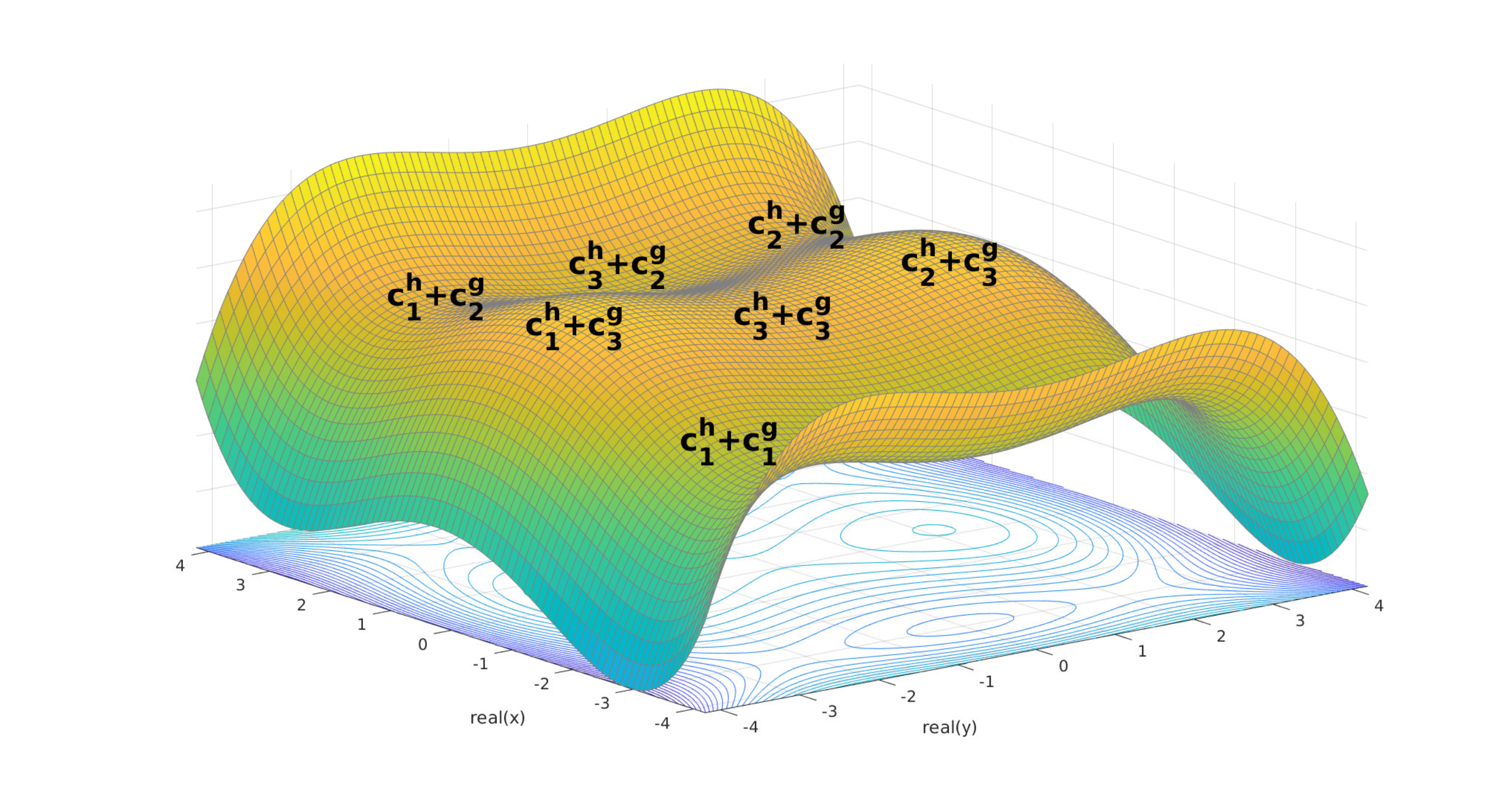}
	\caption{\footnotesize   Real part of graph defined by $f_2(x,y)=-y^4+9y^2+x^4-16x^2-8x=t$, with its critical values.}
	\label{h-g_3d}
\end{figure}

We denote the critical values by
$$a_1=c_1^h+c_1^g\text{ , }a_2=c_1^h+c_2^g \text{ , } a_3=c_1^h+c_3^g$$ 
$$a_4=c_2^h+c_1^g\text{ , }a_5=c_2^h+c_2^g \text{ , } a_6=c_2^h+c_3^g$$ 
$$a_7=c_3^h+c_1^g\text{ , }a_8=c_3^h+c_2^g \text{ , } a_9=c_3^h+c_3^g.$$ 

If we consider the contour lines associated to  the Figure \ref{h+g_3d}, then we obtain a drawing in the plane which represent the vertex in the Dynkin diagram associated to the polynomial $f_1$. In fact, the correspondence between the Dynkin diagram and the curve in the plane is  shown  in \cite{ACampo}. In Figure \ref{h+g_2d}, we present the contour lines of the real part of the polynomial $f_1(x,y)$.  The critical values $a_2, a_3, a_5, a_6$ and $a_7$ correspond with ovals contained in the real part of the foliation defined by $df_1$.

Analogously, the contour lines of the Figure \ref{h-g_3d} give us a drawing in the plane which represent the vertex in the Dynkin diagram associated to $f_2$. In Figure \ref{h-g_2d}, we present the contour lines of the real part of the polynomial $f_2(x,y)$. In this case, the critical values $a_3, a_6, a_7$ and $a_8$ correspond with ovals contained in the real part of the foliation defined by $df_2$.
\begin{figure}[h!]
	\centering
	\includegraphics[width=0.6\textwidth]{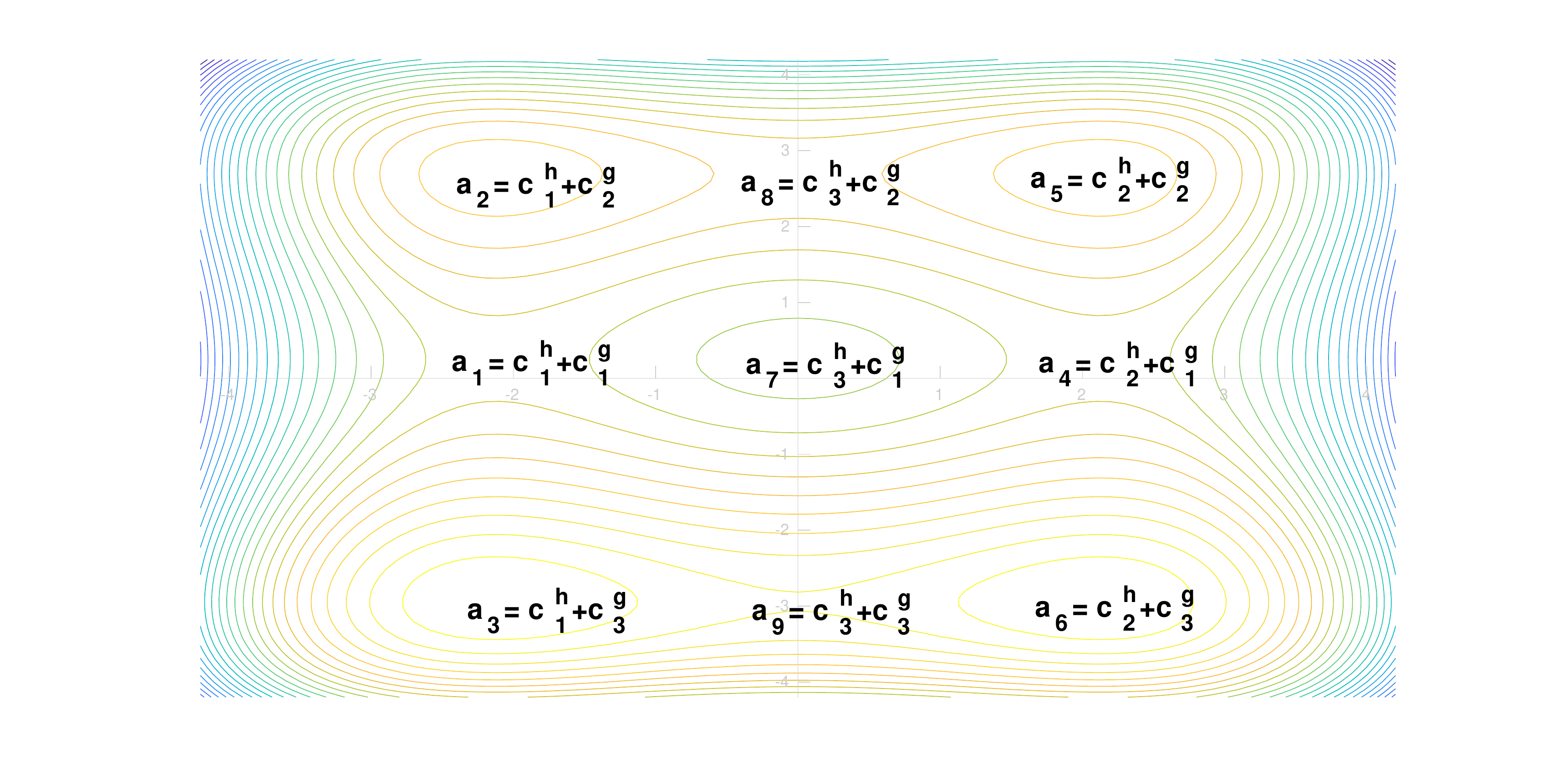}
	\caption{\footnotesize  Contour lines of the real part of  $f_1(x,y)=-y^4+9y^2-x^4+16x^2+8x=t$. }
	\label{h+g_2d}
\end{figure}
\begin{figure}[h!]
	\centering
	\includegraphics[width=0.6\textwidth]{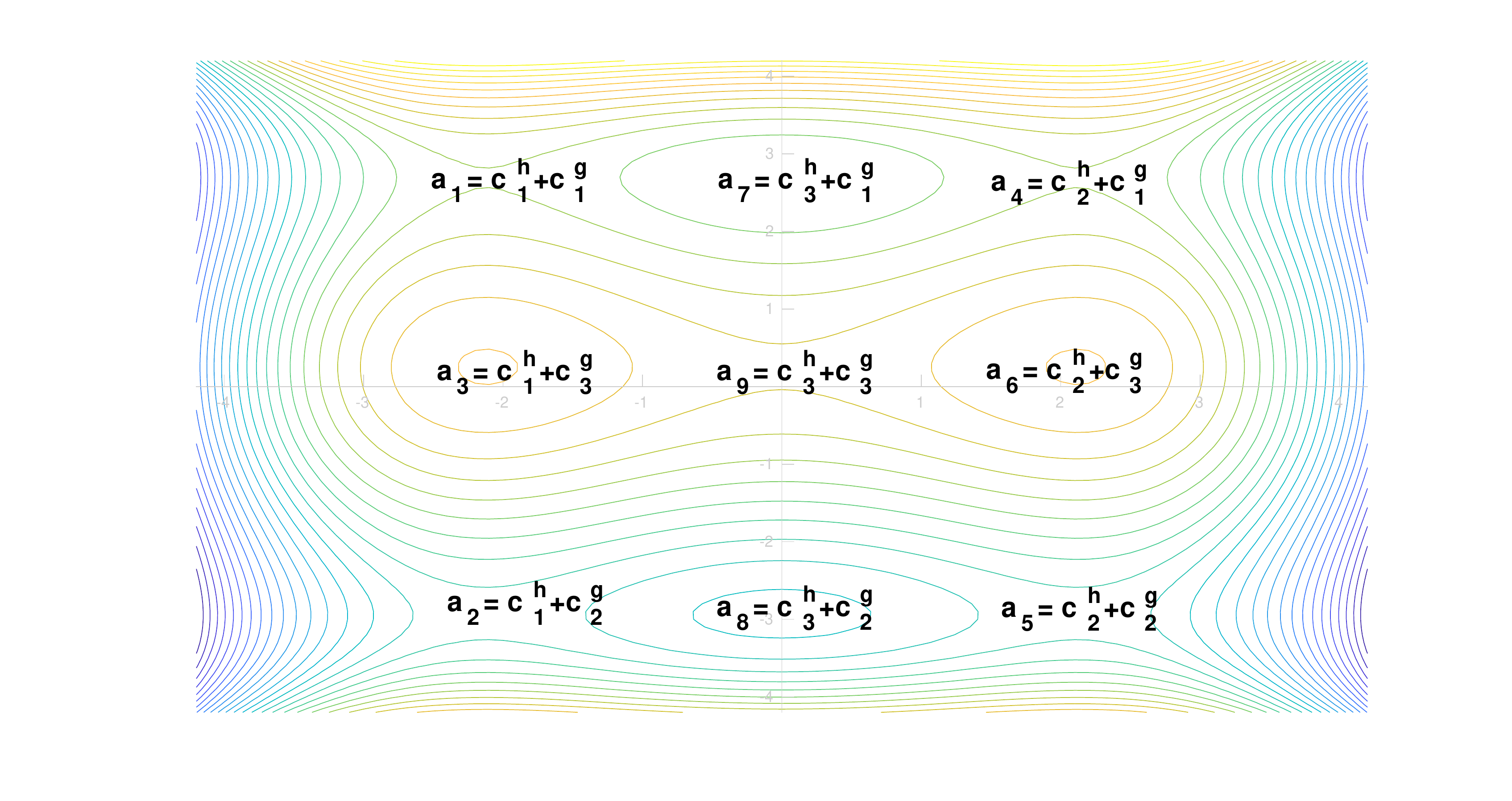}
	\caption{\footnotesize Contour lines of the real part of  $f_1(x,y)=-y^4+9y^2+x^4-16x^2-8x=t$.}
	\label{h-g_2d}
\end{figure}

Let $\A_i$  be the cycle which vanishes in the critical point corresponding to the critical value $a_i$ for $i=1,...,9$.  For simplicity in the notation we call the possible critical values as $a, b, c, d,e,f ,g ,h, i$ if all are different and we are removing from right to left as soon as the critical values are repeated. Moreover, we indicate with "*" on the right, the vanishing cycles which is not contained in the real plane (or its associated critical value). For instance,  the polynomial $f_1(x,y)$ of the  Example \ref{exampleh+g} satisfies that $a_1=a_4, a_2=a_5, a_3=a_6$ and the other critical values are different, therefore its  next Dynkin diagram is
\begin{equation*}
\begin{tikzpicture}           
\matrix (m) [matrix of math nodes, row sep=0.9em,
column sep=0.9em]{
	b&e^*&b\\
	a^*&d&a^*\\
	c&f^*&c,\\
};			
\path[-stealth]
(m-1-2) edge (m-1-1)
(m-1-2) edge(m-1-3)		
(m-1-1) edge (m-2-1)
(m-2-2)	edge (m-2-1)
(m-3-1) edge (m-2-1)
(m-2-1)	edge (m-1-2)
(m-2-1) edge (m-3-2)
(m-1-2) edge (m-2-2)
(m-3-2) edge (m-2-2)
(m-1-3) edge (m-2-3)
(m-2-2) edge (m-2-3)
(m-3-3) edge (m-2-3)
(m-2-3) edge (m-1-2)
(m-2-3) edge (m-3-2)
(m-3-2) edge (m-3-1)		
(m-3-2) edge (m-3-3);
\end{tikzpicture}
\end{equation*}
Similarly, the polynomial $f_2(x,y)$ of the Example \ref{exampleh-g}, has Dynkin diagram 
\begin{equation*}
\begin{tikzpicture}           
\matrix (m) [matrix of math nodes, row sep=0.9em,
column sep=0.9em]{
	a^*&d&a^*\\
	c&e^*&c\\
	b^*&f&b^*.\\
};			
\path[-stealth]
(m-2-1) edge (m-1-1)
(m-1-2) edge (m-1-1)
(m-1-1) edge (m-2-2)
(m-2-2) edge (m-1-2)
(m-2-3) edge (m-1-3)		
(m-1-2) edge (m-1-3)		
(m-1-3) edge (m-2-2)
(m-2-1) edge (m-3-1)
(m-3-2) edge (m-3-1)
(m-3-1) edge (m-2-2)
(m-2-2) edge (m-3-2)
(m-2-3) edge (m-3-3)		
(m-3-2) edge (m-3-3)		
(m-3-3) edge (m-2-2)
(m-2-2) edge (m-2-1)
(m-2-2) edge (m-2-3)
;
\end{tikzpicture}
\end{equation*}
The subspace of the homology $H_1(f^{-1}(b))$  generated by the monodromy action on a vanishing cycle $\A_i$ is denoted  $\text{Mon}(\A_i)$ $i=1,\ldots, 9$.  We will compute the monodromy for any $\A_i$ depending on the number of different critical values.  

For the Dynkin diagram
\begin{equation}
\label{Dynkincase1}
\begin{tikzpicture}           
\matrix (m) [matrix of math nodes, row sep=1em,
column sep=1em]{
	a_2&a_8^*&a_5\\
	a_1^*&a_7&a_4^*\\
	a_3&a_9^*&a_6,\\
};			
\path[-stealth]
(m-1-2) edge (m-1-1)
(m-1-2) edge (m-1-3)		
(m-1-1) edge (m-2-1)
(m-2-2)	edge (m-2-1)
(m-3-1) edge (m-2-1)
(m-2-1)	edge (m-1-2)
(m-2-1) edge (m-3-2)
(m-1-2) edge (m-2-2)
(m-3-2) edge (m-2-2)
(m-1-3) edge (m-2-3)
(m-2-2) edge (m-2-3)
(m-3-3) edge (m-2-3)
(m-2-3) edge (m-1-2)
(m-2-3) edge (m-3-2)
(m-3-2) edge (m-3-1)		
(m-3-2) edge (m-3-3);
\end{tikzpicture}
\end{equation}
when there is  one critical value the ranks of the subspaces are: $\rank(\text{Mon}(\A_i))=5$ for $i\neq 7$ and $\rank(\text{Mon}(\A_7))=3$. For more than one critical  value, in Table \ref{case2b}  we present the cases where the vanishing cycles are not simple cycles.

For the Dynkin diagram
\begin{equation}
\label{Dynkincase2}
\begin{tikzpicture}           
\matrix (m) [matrix of math nodes, row sep=1em,
column sep=1em]{
	a_1^*&a_7&a_4^*\\
	a_3&a_9^*&a_6\\
	a_2^*&a_8&a_5^*,\\
};			
\path[-stealth]
(m-2-1) edge (m-1-1)
(m-1-2) edge (m-1-1)
(m-1-1) edge (m-2-2)
(m-2-2) edge (m-1-2)
(m-2-3) edge (m-1-3)		
(m-1-2) edge (m-1-3)		
(m-1-3) edge (m-2-2)
(m-2-1) edge (m-3-1)
(m-3-2) edge (m-3-1)
(m-3-1) edge (m-2-2)
(m-2-2) edge (m-3-2)
(m-2-3) edge (m-3-3)		
(m-3-2) edge (m-3-3)		
(m-3-3) edge (m-2-2)
(m-2-2) edge (m-2-1)
(m-2-2) edge (m-2-3)
;
\end{tikzpicture} 
\end{equation}
when there is  one critical value the ranks of the subspaces are: $\rank(\text{Mon}(\A_i))=5$ for $i\neq 9$ and $\rank(\text{Mon}(\A_9))=3$. For more than one critical  value, in Table \ref{case4b}  we present the cases where the vanishing cycles are not simple cycles.

In Tables \ref{case2b} and \ref{case4b}, the first column is the number of different critical values. In the second column  are written the vanishing cycles which are not simple cycles. Right in front of any  non simple vanishing cycle $\alpha_i$, in the the third column,  it is a basis for the subspace $\text{Mon}(\alpha_i)$. Note that there are vanishing cycles which generate the same subspace, then they are on the same line in the second column. In fourth column  are the corresponding Dynkin diagram. Finally, in the last column, we add  information about an equivalence class, which is explained below. This last column together with the Theorem \ref{orbitspacehg} give us examples of polynomials that satisfy these diagrams. 
\begin{table}[htbp]
	\tiny
	\centering
	\begin{tabular}{|l|l|l|l|l|}
		\hline
		$\#$ critical&$\A_i$&$\text{Mon}(\A_i)$& Dynkin diagram of&$[f]$   \\values&&&$f(x,y)=h(x)+g(y)$&\\
		\hline \hline 
		2
		&
		\begin{minipage}{1cm}
			\begin{align*}
			\A_1^*, \A_4^*\\
			\A_8^*, \A_9^*\\
			\A_7
			\end{align*}
		\end{minipage}
		&
		\begin{minipage}{4cm}
			\begin{align*}
			\langle \A_1^*, \A_4^*, \A_7, \A_2+\A_3, \A_5+\A_6, \A_8^*+\A_9^*\rangle\\
			\langle \A_7, \A_8^*, \A_9^*, \A_1^*+\A_4^*, \A_2+\A_5, \A_3+\A_6\rangle\\
			\langle \A_7, \A_1^*+\A_4^*, \A_8^*+\A_9^*, \A_2+\A_3+\A_5+\A_6\rangle
			\end{align*}
		\end{minipage}
		&
		\begin{minipage}{5cm}\begin{tikzpicture}           
			\matrix (m) [matrix of math nodes, row sep=0.5em,
			column sep=0.5em]{
				a&b^*&a\\
				a^*&b&a^*\\
				a&b^*&a\\
			};			
			\path[-stealth]
			(m-1-2) edge (m-1-1)
			(m-1-2) edge (m-1-3)		
			(m-1-1) edge (m-2-1)
			(m-2-2)	edge (m-2-1)
			(m-3-1) edge (m-2-1)
			(m-2-1)	edge (m-1-2)
			(m-2-1) edge (m-3-2)
			(m-1-2) edge (m-2-2)
			(m-3-2) edge (m-2-2)
			(m-1-3) edge (m-2-3)
			(m-2-2) edge (m-2-3)
			(m-3-3) edge (m-2-3)
			(m-2-3) edge (m-1-2)
			(m-2-3) edge (m-3-2)
			(m-3-2) edge (m-3-1)		
			(m-3-2) edge (m-3-3);
			\end{tikzpicture}
			\begin{tikzpicture}           
			\matrix (m) [matrix of math nodes, row sep=0.5em,
			column sep=0.5em]{
				a&a^*&a\\
				b^*&b&b^*\\
				a&a^*&a\\
			};			
			\path[-stealth]
			(m-1-2) edge (m-1-1)
			(m-1-2) edge (m-1-3)		
			(m-1-1) edge (m-2-1)
			(m-2-2)	edge (m-2-1)
			(m-3-1) edge (m-2-1)
			(m-2-1)	edge (m-1-2)
			(m-2-1) edge (m-3-2)
			(m-1-2) edge (m-2-2)
			(m-3-2) edge (m-2-2)
			(m-1-3) edge (m-2-3)
			(m-2-2) edge (m-2-3)
			(m-3-3) edge (m-2-3)
			(m-2-3) edge (m-1-2)
			(m-2-3) edge (m-3-2)
			(m-3-2) edge (m-3-1)		
			(m-3-2) edge (m-3-3);
			\end{tikzpicture}
		\end{minipage}&$\Or_3$
		\\ \hline
		2&
		$\A_7, \A_8^*, \A_9^*$
		&$\langle \A_7,\A_8^*,\A_9^*, \A_1^*+\A_4^*,\A_2+\A_5, \A_3+\A_6 \rangle
		$
		&  
		\begin{minipage}{4cm}
			\begin{tikzpicture}           
			\matrix (m) [matrix of math nodes, row sep=0.5em,
			column sep=0.5em]{
				a&a^*&a\\
				a^*&a&a^*\\
				b&b^*&b\\
			};			
			\path[-stealth]
			(m-1-2) edge (m-1-1)
			(m-1-2) edge(m-1-3)		
			(m-1-1) edge (m-2-1)
			(m-2-2)	edge (m-2-1)
			(m-3-1) edge (m-2-1)
			(m-2-1)	edge (m-1-2)
			(m-2-1) edge (m-3-2)
			(m-1-2) edge (m-2-2)
			(m-3-2) edge (m-2-2)
			(m-1-3) edge (m-2-3)
			(m-2-2) edge (m-2-3)
			(m-3-3) edge (m-2-3)
			(m-2-3) edge (m-1-2)
			(m-2-3) edge (m-3-2)
			(m-3-2) edge (m-3-1)		
			(m-3-2) edge (m-3-3);
			\end{tikzpicture}
		\end{minipage}&$\Or_2$\\
		\hline
		2&$\A_1^*, \A_4^*, \A_7$&$\langle \A_1^*,\A_4^*,\A_7, \A_2+\A_3,\A_5+\A_6, \A_8^*+\A_9^* \rangle$&  
		\begin{minipage}{4cm}
			\begin{tikzpicture}           
			\matrix (m) [matrix of math nodes, row sep=0.5em,
			column sep=0.5em]{
				b&a^*&a\\
				b^*&a&a^*\\
				b&a^*&a\\
			};			
			\path[-stealth]
			(m-1-2) edge (m-1-1)
			(m-1-2) edge(m-1-3)		
			(m-1-1) edge (m-2-1)
			(m-2-2)	edge (m-2-1)
			(m-3-1) edge (m-2-1)
			(m-2-1)	edge (m-1-2)
			(m-2-1) edge (m-3-2)
			(m-1-2) edge (m-2-2)
			(m-3-2) edge (m-2-2)
			(m-1-3) edge (m-2-3)
			(m-2-2) edge (m-2-3)
			(m-3-3) edge (m-2-3)
			(m-2-3) edge (m-1-2)
			(m-2-3) edge (m-3-2)
			(m-3-2) edge (m-3-1)		
			(m-3-2) edge (m-3-3);
			\end{tikzpicture}
		\end{minipage}&$\Or_2$\\
		\hline
		3&
		\begin{minipage}{1cm}
			\begin{align*}
			\A_1^*, \A_4^*\\
			\A_2, \A_6\\
			\A_3, \A_5\\
			\A_8^*, \A_9^*\\
			\A_7
			\end{align*}
		\end{minipage}& 
		\begin{minipage}{4cm}
			\begin{align*}
			\langle \A_1^*, \A_4^*, \A_7, \A_2+\A_3, \A_5+\A_6, \A_8^*+\A_9^*\rangle\\
			\langle \A_2, \A_6,\A_7, \A_1^*-\A_8^*,\A_4^*-\A_9^*,\A_3+\A_5,\A_1^*+\A_4^*+\A_8^*+\A_9^*\rangle\\
			\langle \A_3, \A_5,\A_7, \A_1^*-\A_9^*,\A_4^*-\A_8^*,\A_2+\A_6,\A_1^*+\A_4^*+\A_8^*+\A_9^*\rangle\\
			\langle \A_7, \A_8^*, \A_9^*, \A_1^*+\A_4^*, \A_2+\A_5, \A_3+\A_6\rangle\\
			\langle \A_7, \A_1^*+\A_4^*, \A_8^*+\A_9^*, \A_2+\A_3+\A_5+\A_6\rangle
			\end{align*}
		\end{minipage}
		&
		\begin{minipage}{4cm}
			\begin{tikzpicture}           
			\matrix (m) [matrix of math nodes, row sep=0.5em,
			column sep=0.5em]{
				b&a^*&b\\
				a^*&c&a^*\\
				b&a^*&b\\
			};			
			\path[-stealth] 
			(m-1-2) edge (m-1-1)
			(m-1-2) edge(m-1-3)		
			(m-1-1) edge (m-2-1)
			(m-2-2)	edge (m-2-1)
			(m-3-1) edge (m-2-1)
			(m-2-1)	edge (m-1-2)
			(m-2-1) edge (m-3-2)
			(m-1-2) edge (m-2-2)
			(m-3-2) edge (m-2-2)
			(m-1-3) edge (m-2-3)
			(m-2-2) edge (m-2-3)
			(m-3-3) edge (m-2-3)
			(m-2-3) edge (m-1-2)
			(m-2-3) edge (m-3-2)
			(m-3-2) edge (m-3-1)		
			(m-3-2) edge (m-3-3);
			\end{tikzpicture}
		\end{minipage}&$\Or_4$\\
		\hline

		3&$\A_7, \A_8^*, \A_9^*$&$\langle \A_7,\A_8^*,\A_9^*, \A_1^*+\A_4^*,\A_2+\A_5, \A_3+\A_6 \rangle$&  
		\begin{minipage}{4cm}
			\begin{tikzpicture}           
			\matrix (m) [matrix of math nodes, row sep=0.5em,
			column sep=0.5em]{
				b&b^*&b\\
				a^*&a&a^*\\
				c&c^*&c\\
			};			
			\path[-stealth]
			(m-1-2) edge (m-1-1)
			(m-1-2) edge(m-1-3)		
			(m-1-1) edge (m-2-1)
			(m-2-2)	edge (m-2-1)
			(m-3-1) edge (m-2-1)
			(m-2-1)	edge (m-1-2)
			(m-2-1) edge (m-3-2)
			(m-1-2) edge (m-2-2)
			(m-3-2) edge (m-2-2)
			(m-1-3) edge (m-2-3)
			(m-2-2) edge (m-2-3)
			(m-3-3) edge (m-2-3)
			(m-2-3) edge (m-1-2)
			(m-2-3) edge (m-3-2)
			(m-3-2) edge (m-3-1)		
			(m-3-2) edge (m-3-3);
			\end{tikzpicture}\begin{tikzpicture}           
			\matrix (m) [matrix of math nodes, row sep=0.5em,
			column sep=0.5em]{
				a&b^*&a\\
				a^*&b&a^*\\
				c&a^*&c\\
			};			
			\path[-stealth]
			(m-1-2) edge (m-1-1)
			(m-1-2) edge(m-1-3)		
			(m-1-1) edge (m-2-1)
			(m-2-2)	edge (m-2-1)
			(m-3-1) edge (m-2-1)
			(m-2-1)	edge (m-1-2)
			(m-2-1) edge (m-3-2)
			(m-1-2) edge (m-2-2)
			(m-3-2) edge (m-2-2)
			(m-1-3) edge (m-2-3)
			(m-2-2) edge (m-2-3)
			(m-3-3) edge (m-2-3)
			(m-2-3) edge (m-1-2)
			(m-2-3) edge (m-3-2)
			(m-3-2) edge (m-3-1)		
			(m-3-2) edge (m-3-3);
			\end{tikzpicture}
		\end{minipage}&$\Or_2$\\
		\hline
		3&$\A_1^*, \A_4^*, \A_7$&$\langle \A_1^*,\A_4^*,\A_7, \A_2+\A_3,\A_5+\A_6, \A_8^*+\A_9^* \rangle$&  
		\begin{minipage}{4cm}
			\begin{tikzpicture}           
			\matrix (m) [matrix of math nodes, row sep=0.5em,
			column sep=0.5em]{
				a&c^*&b\\
				a^*&c&b^*\\
				a&c^*&b\\
			};			
			\path[-stealth]
			(m-1-2) edge (m-1-1)
			(m-1-2) edge(m-1-3)		
			(m-1-1) edge (m-2-1)
			(m-2-2)	edge (m-2-1)
			(m-3-1) edge (m-2-1)
			(m-2-1)	edge (m-1-2)
			(m-2-1) edge (m-3-2)
			(m-1-2) edge (m-2-2)
			(m-3-2) edge (m-2-2)
			(m-1-3) edge (m-2-3)
			(m-2-2) edge (m-2-3)
			(m-3-3) edge (m-2-3)
			(m-2-3) edge (m-1-2)
			(m-2-3) edge (m-3-2)
			(m-3-2) edge (m-3-1)		
			(m-3-2) edge (m-3-3);
			\end{tikzpicture}\begin{tikzpicture}           
			\matrix (m) [matrix of math nodes, row sep=0.5em,
			column sep=0.5em]{
				b&a^*&a\\
				a^*&c&c^*\\
				b&a^*&a\\
			};			
			\path[-stealth]
			(m-1-2) edge (m-1-1)
			(m-1-2) edge(m-1-3)		
			(m-1-1) edge (m-2-1)
			(m-2-2)	edge (m-2-1)
			(m-3-1) edge (m-2-1)
			(m-2-1)	edge (m-1-2)
			(m-2-1) edge (m-3-2)
			(m-1-2) edge (m-2-2)
			(m-3-2) edge (m-2-2)
			(m-1-3) edge (m-2-3)
			(m-2-2) edge (m-2-3)
			(m-3-3) edge (m-2-3)
			(m-2-3) edge (m-1-2)
			(m-2-3) edge (m-3-2)
			(m-3-2) edge (m-3-1)		
			(m-3-2) edge (m-3-3);
			\end{tikzpicture}
		\end{minipage}&$\Or_2$\\
		\hline

		4&
		\begin{minipage}{1cm}
			\begin{align*}
			\A_1^*, \A_4^*\\
			\A_8^*, \A_9^*\\
			\A_7
			\end{align*}	
		\end{minipage}& 
		\begin{minipage}{4cm}
			\begin{align*}
			\langle \A_1^*, \A_4^*, \A_7, \A_2+\A_3, \A_5+\A_6, \A_8^*+\A_9^*\rangle\\
			\langle \A_7, \A_8^*, \A_9^*, \A_1^*+\A_4^*, \A_2+\A_5, \A_3+\A_6\rangle\\
			\langle \A_7, \A_1^*+\A_4^*, \A_8^*+\A_9^*, \A_2+\A_3+\A_5+\A_6\rangle
			\end{align*}
		\end{minipage}&
		\begin{minipage}{4cm}
			\begin{tikzpicture}           
			\matrix (m) [matrix of math nodes, row sep=0.5em,
			column sep=0.5em]{
				a&b^*&a\\
				c^*&d&c^*\\
				a&b^*&a\\
			};			
			\path[-stealth]
			(m-1-2) edge (m-1-1)
			(m-1-2) edge(m-1-3)		
			(m-1-1) edge (m-2-1)
			(m-2-2)	edge (m-2-1)
			(m-3-1) edge (m-2-1)
			(m-2-1)	edge (m-1-2)
			(m-2-1) edge (m-3-2)
			(m-1-2) edge (m-2-2)
			(m-3-2) edge (m-2-2)
			(m-1-3) edge (m-2-3)
			(m-2-2) edge (m-2-3)
			(m-3-3) edge (m-2-3)
			(m-2-3) edge (m-1-2)
			(m-2-3) edge (m-3-2)
			(m-3-2) edge (m-3-1)		
			(m-3-2) edge (m-3-3);
			\end{tikzpicture}
		\end{minipage}&$\Or_3$\\
		\hline
		4&$\A_7, \A_8^*, \A_9^*$&$\langle \A_7,\A_8^*,\A_9^*, \A_1^*+\A_4^*,\A_2+\A_5, \A_3+\A_6 \rangle$&  
		\begin{minipage}{4cm}
			\begin{tikzpicture}           
			\matrix (m) [matrix of math nodes, row sep=0.5em,
			column sep=0.5em]{
				a&b^*&a\\
				a^*&b&a^*\\
				c&d^*&c\\
			};			
			\path[-stealth]
			(m-1-2) edge (m-1-1)
			(m-1-2) edge(m-1-3)		
			(m-1-1) edge (m-2-1)
			(m-2-2)	edge (m-2-1)
			(m-3-1) edge (m-2-1)
			(m-2-1)	edge (m-1-2)
			(m-2-1) edge (m-3-2)
			(m-1-2) edge (m-2-2)
			(m-3-2) edge (m-2-2)
			(m-1-3) edge (m-2-3)
			(m-2-2) edge (m-2-3)
			(m-3-3) edge (m-2-3)
			(m-2-3) edge (m-1-2)
			(m-2-3) edge (m-3-2)
			(m-3-2) edge (m-3-1)		
			(m-3-2) edge (m-3-3);
			\end{tikzpicture}
			\begin{tikzpicture}           
			\matrix (m) [matrix of math nodes, row sep=0.5em,
			column sep=0.5em]{
				b&a^*&b\\
				a^*&d&a^*\\
				c&b^*&c\\
			};			
			\path[-stealth]
			(m-1-2) edge (m-1-1)
			(m-1-2) edge(m-1-3)		
			(m-1-1) edge (m-2-1)
			(m-2-2)	edge (m-2-1)
			(m-3-1) edge (m-2-1)
			(m-2-1)	edge (m-1-2)
			(m-2-1) edge (m-3-2)
			(m-1-2) edge (m-2-2)
			(m-3-2) edge (m-2-2)
			(m-1-3) edge (m-2-3)
			(m-2-2) edge (m-2-3)
			(m-3-3) edge (m-2-3)
			(m-2-3) edge (m-1-2)
			(m-2-3) edge (m-3-2)
			(m-3-2) edge (m-3-1)		
			(m-3-2) edge (m-3-3);
			\end{tikzpicture}
		\end{minipage}&$\Or_2$\\
		\hline
		4&$\A_1^*, \A_4^*, \A_7$&$\langle \A_1^*,\A_4^*,\A_7, \A_2+\A_3,\A_5+\A_6, \A_8^*+\A_9^* \rangle$&  
		\begin{minipage}{4cm}
			\begin{tikzpicture}           
			\matrix (m) [matrix of math nodes, row sep=0.5em,
			column sep=0.5em]{
				b&a^*&a\\
				d^*&c&c^*\\
				b&a^*&a\\
			};			
			\path[-stealth]
			(m-1-2) edge (m-1-1)
			(m-1-2) edge(m-1-3)		
			(m-1-1) edge (m-2-1)
			(m-2-2)	edge (m-2-1)
			(m-3-1) edge (m-2-1)
			(m-2-1)	edge (m-1-2)
			(m-2-1) edge (m-3-2)
			(m-1-2) edge (m-2-2)
			(m-3-2) edge (m-2-2)
			(m-1-3) edge (m-2-3)
			(m-2-2) edge (m-2-3)
			(m-3-3) edge (m-2-3)
			(m-2-3) edge (m-1-2)
			(m-2-3) edge (m-3-2)
			(m-3-2) edge (m-3-1)		
			(m-3-2) edge (m-3-3);
			\end{tikzpicture}\begin{tikzpicture}           
			\matrix (m) [matrix of math nodes, row sep=0.5em,
			column sep=0.5em]{
				c&a^*&b\\
				b^*&d&a^*\\
				c&a^*&b\\
			};			
			\path[-stealth]
			(m-1-2) edge (m-1-1)
			(m-1-2) edge(m-1-3)		
			(m-1-1) edge (m-2-1)
			(m-2-2)	edge (m-2-1)
			(m-3-1) edge (m-2-1)
			(m-2-1)	edge (m-1-2)
			(m-2-1) edge (m-3-2)
			(m-1-2) edge (m-2-2)
			(m-3-2) edge (m-2-2)
			(m-1-3) edge (m-2-3)
			(m-2-2) edge (m-2-3)
			(m-3-3) edge (m-2-3)
			(m-2-3) edge (m-1-2)
			(m-2-3) edge (m-3-2)
			(m-3-2) edge (m-3-1)		
			(m-3-2) edge (m-3-3);
			\end{tikzpicture}
		\end{minipage}&$\Or_2$\\
		\hline
		
		5&$\A_7, \A_8^*, \A_9^*$&$\langle \A_7,\A_8,\A_9, \A_1^*+\A_4^*,\A_2+\A_5, \A_3+\A_6 \rangle$&  
		\begin{minipage}{5cm}
			\begin{tikzpicture}           
			\matrix (m) [matrix of math nodes, row sep=0.5em,
			column sep=0.5em]{
				b&a^*&b\\
				a^*&d&a^*\\
				c&e^*&c\\
			};			
			\path[-stealth]
			(m-1-2) edge (m-1-1)
			(m-1-2) edge(m-1-3)		
			(m-1-1) edge (m-2-1)
			(m-2-2)	edge (m-2-1)
			(m-3-1) edge (m-2-1)
			(m-2-1)	edge (m-1-2)
			(m-2-1) edge (m-3-2)
			(m-1-2) edge (m-2-2)
			(m-3-2) edge (m-2-2)
			(m-1-3) edge (m-2-3)
			(m-2-2) edge (m-2-3)
			(m-3-3) edge (m-2-3)
			(m-2-3) edge (m-1-2)
			(m-2-3) edge (m-3-2)
			(m-3-2) edge (m-3-1)		
			(m-3-2) edge (m-3-3);
			\end{tikzpicture}\begin{tikzpicture}           
			\matrix (m) [matrix of math nodes, row sep=0.5em,
			column sep=0.5em]{
				b&e^*&b\\
				a^*&d&a^*\\
				c&a^*&c\\
			};			
			\path[-stealth]
			(m-1-2) edge (m-1-1)
			(m-1-2) edge(m-1-3)		
			(m-1-1) edge (m-2-1)
			(m-2-2)	edge (m-2-1)
			(m-3-1) edge (m-2-1)
			(m-2-1)	edge (m-1-2)
			(m-2-1) edge (m-3-2)
			(m-1-2) edge (m-2-2)
			(m-3-2) edge (m-2-2)
			(m-1-3) edge (m-2-3)
			(m-2-2) edge (m-2-3)
			(m-3-3) edge (m-2-3)
			(m-2-3) edge (m-1-2)
			(m-2-3) edge (m-3-2)
			(m-3-2) edge (m-3-1)		
			(m-3-2) edge (m-3-3);
			\end{tikzpicture}\begin{tikzpicture}           
			\matrix (m) [matrix of math nodes, row sep=0.5em,
			column sep=0.5em]{
				a&e^*&a\\
				b^*&d&b^*\\
				c&a^*&c\\
			};			
			\path[-stealth]
			(m-1-2) edge (m-1-1)
			(m-1-2) edge(m-1-3)		
			(m-1-1) edge (m-2-1)
			(m-2-2)	edge (m-2-1)
			(m-3-1) edge (m-2-1)
			(m-2-1)	edge (m-1-2)
			(m-2-1) edge (m-3-2)
			(m-1-2) edge (m-2-2)
			(m-3-2) edge (m-2-2)
			(m-1-3) edge (m-2-3)
			(m-2-2) edge (m-2-3)
			(m-3-3) edge (m-2-3)
			(m-2-3) edge (m-1-2)
			(m-2-3) edge (m-3-2)
			(m-3-2) edge (m-3-1)		
			(m-3-2) edge (m-3-3);
			\end{tikzpicture}
		\end{minipage}&$\Or_2$\\
		\hline
		5&$\A_1^*, \A_4^*, \A_7$&$\langle \A_1^*,\A_4^*,\A_7, \A_2+\A_3,\A_5+\A_6, \A_8^*+\A_9^* \rangle$&  
		\begin{minipage}{5cm}
			\begin{tikzpicture}           
			\matrix (m) [matrix of math nodes, row sep=0.5em,
			column sep=0.5em]{
				c&a^*&b\\
				e^*&d&a^*\\
				c&a^*&b\\
			};			
			\path[-stealth]
			(m-1-2) edge (m-1-1)
			(m-1-2) edge(m-1-3)		
			(m-1-1) edge (m-2-1)
			(m-2-2)	edge (m-2-1)
			(m-3-1) edge (m-2-1)
			(m-2-1)	edge (m-1-2)
			(m-2-1) edge (m-3-2)
			(m-1-2) edge (m-2-2)
			(m-3-2) edge (m-2-2)
			(m-1-3) edge (m-2-3)
			(m-2-2) edge (m-2-3)
			(m-3-3) edge (m-2-3)
			(m-2-3) edge (m-1-2)
			(m-2-3) edge (m-3-2)
			(m-3-2) edge (m-3-1)		
			(m-3-2) edge (m-3-3);
			\end{tikzpicture}\begin{tikzpicture}           
			\matrix (m) [matrix of math nodes, row sep=0.5em,
			column sep=0.5em]{
				c&a^*&b\\
				a^*&d&e^*\\
				c&a^*&b\\
			};			
			\path[-stealth]
			(m-1-2) edge (m-1-1)
			(m-1-2) edge(m-1-3)		
			(m-1-1) edge (m-2-1)
			(m-2-2)	edge (m-2-1)
			(m-3-1) edge (m-2-1)
			(m-2-1)	edge (m-1-2)
			(m-2-1) edge (m-3-2)
			(m-1-2) edge (m-2-2)
			(m-3-2) edge (m-2-2)
			(m-1-3) edge (m-2-3)
			(m-2-2) edge (m-2-3)
			(m-3-3) edge (m-2-3)
			(m-2-3) edge (m-1-2)
			(m-2-3) edge (m-3-2)
			(m-3-2) edge (m-3-1)		
			(m-3-2) edge (m-3-3);
			\end{tikzpicture}\begin{tikzpicture}           
			\matrix (m) [matrix of math nodes, row sep=0.5em,
			column sep=0.5em]{
				c&b^*&a\\
				a^*&d&e^*\\
				c&b^*&a\\
			};			
			\path[-stealth]
			(m-1-2) edge (m-1-1)
			(m-1-2) edge(m-1-3)		
			(m-1-1) edge (m-2-1)
			(m-2-2)	edge (m-2-1)
			(m-3-1) edge (m-2-1)
			(m-2-1)	edge (m-1-2)
			(m-2-1) edge (m-3-2)
			(m-1-2) edge (m-2-2)
			(m-3-2) edge (m-2-2)
			(m-1-3) edge (m-2-3)
			(m-2-2) edge (m-2-3)
			(m-3-3) edge (m-2-3)
			(m-2-3) edge (m-1-2)
			(m-2-3) edge (m-3-2)
			(m-3-2) edge (m-3-1)		
			(m-3-2) edge (m-3-3);
			\end{tikzpicture}
		\end{minipage}&$\Or_2$\\
		\hline
		6&$\A_7, \A_8^*, \A_9^*$&$\langle \A_7,\A_8^*,\A_9^*, \A_1^*+\A_4^*,\A_2+\A_5, \A_3+\A_6 \rangle$&  
		\begin{minipage}{4cm}
			\begin{tikzpicture}           
			\matrix (m) [matrix of math nodes, row sep=0.5em,
			column sep=0.5em]{
				b&e^*&b\\
				a^*&d&a^*\\
				c&f^*&c\\
			};			
			\path[-stealth]
			(m-1-2) edge (m-1-1)
			(m-1-2) edge(m-1-3)		
			(m-1-1) edge (m-2-1)
			(m-2-2)	edge (m-2-1)
			(m-3-1) edge (m-2-1)
			(m-2-1)	edge (m-1-2)
			(m-2-1) edge (m-3-2)
			(m-1-2) edge (m-2-2)
			(m-3-2) edge (m-2-2)
			(m-1-3) edge (m-2-3)
			(m-2-2) edge (m-2-3)
			(m-3-3) edge (m-2-3)
			(m-2-3) edge (m-1-2)
			(m-2-3) edge (m-3-2)
			(m-3-2) edge (m-3-1)		
			(m-3-2) edge (m-3-3);
			\end{tikzpicture}
		\end{minipage}&$\Or_2$\\
		\hline
		
		6&$\A_1^*, \A_4^*, \A_7$&$\langle \A_1^*,\A_4^*,\A_7, \A_2+\A_3,\A_5+\A_6, \A_8^*+\A_9^* \rangle$&  
		\begin{minipage}{4cm}
			\begin{tikzpicture}           
			\matrix (m) [matrix of math nodes, row sep=0.5em,
			column sep=0.5em]{
				a&c^*&b\\
				d^*&f&e^*\\
				a&c^*&b\\
			};			
			\path[-stealth]
			(m-1-2) edge (m-1-1)
			(m-1-2) edge(m-1-3)		
			(m-1-1) edge (m-2-1)
			(m-2-2)	edge (m-2-1)
			(m-3-1) edge (m-2-1)
			(m-2-1)	edge (m-1-2)
			(m-2-1) edge (m-3-2)
			(m-1-2) edge (m-2-2)
			(m-3-2) edge (m-2-2)
			(m-1-3) edge (m-2-3)
			(m-2-2) edge (m-2-3)
			(m-3-3) edge (m-2-3)
			(m-2-3) edge (m-1-2)
			(m-2-3) edge (m-3-2)
			(m-3-2) edge (m-3-1)		
			(m-3-2) edge (m-3-3);
			\end{tikzpicture}
		\end{minipage}&$\Or_2$\\
		\hline
	\end{tabular}
	\caption{Monodromy for $h(x)+g(y)\in \R[x,y]_{d\leq 4}$ and Dynkin diagram (\ref{Dynkincase1})}
	\label{case2b}
\end{table}
\begin{table}[htbp]
	\tiny	
	\centering
	\begin{tabular}{|l|l|l|l|l|}
		\hline
		$\#$ critical&$\A_i$&$\text{Mon}(\A_i)$& Dynkin diagram of& $[f]$  \\values&&&$f(x,y=)h(x)+g(y)$&\\
		\hline \hline 
		2&
		\begin{minipage}{1cm}
			\begin{align*}
			\A_3, \A_6\\
			\A_7, \A_8\\
			\A_9^*
			\end{align*}
		\end{minipage}
		& 
		\begin{minipage}{4cm}
			\begin{align*}
			\langle \A_3, \A_6, \A_9^*, \A_1^*+\A_2^*, \A_7+\A_8, \A_4^*+\A_5^*\rangle\\
			\langle \A_7, \A_8, \A_9^*, \A_1^*+\A_4^*, \A_2^*+\A_5^*, \A_3+\A_6\rangle\\
			\langle \A_9^*, \A_3+\A_6, \A_7+\A_8, \A_1^*+\A_2^*+\A_4^*+\A_5^*\rangle
			\end{align*}
		\end{minipage}
		&
		\begin{minipage}{5cm}
			\begin{tikzpicture}           
			\matrix (m) [matrix of math nodes, row sep=0.5em,
			column sep=0.5em]{
				a^*&a&a^*\\
				b&b^*&b\\
				a^*&a&a^*\\
			};			
			\path[-stealth]
			(m-2-1) edge (m-1-1)
			(m-1-2) edge (m-1-1)
			(m-1-1) edge (m-2-2)
			(m-2-2) edge (m-1-2)
			(m-2-3) edge (m-1-3)		
			(m-1-2) edge (m-1-3)		
			(m-1-3) edge (m-2-2)
			(m-2-1) edge (m-3-1)
			(m-3-2) edge (m-3-1)
			(m-3-1) edge (m-2-2)
			(m-2-2) edge (m-3-2)
			(m-2-3) edge (m-3-3)		
			(m-3-2) edge (m-3-3)		
			(m-3-3) edge (m-2-2)
			(m-2-2) edge (m-2-1)
			(m-2-2) edge (m-2-3)
			;
			\end{tikzpicture}\begin{tikzpicture}           
			\matrix (m) [matrix of math nodes, row sep=0.5em,
			column sep=0.5em]{
				a^*&b&a^*\\
				a&b^*&a\\
				a^*&b&a^*\\
			};			
			\path[-stealth]
			(m-2-1) edge (m-1-1)
			(m-1-2) edge (m-1-1)
			(m-1-1) edge (m-2-2)
			(m-2-2) edge (m-1-2)
			(m-2-3) edge (m-1-3)		
			(m-1-2) edge (m-1-3)		
			(m-1-3) edge (m-2-2)
			(m-2-1) edge (m-3-1)
			(m-3-2) edge (m-3-1)
			(m-3-1) edge (m-2-2)
			(m-2-2) edge (m-3-2)
			(m-2-3) edge (m-3-3)		
			(m-3-2) edge (m-3-3)		
			(m-3-3) edge (m-2-2)
			(m-2-2) edge (m-2-1)
			(m-2-2) edge (m-2-3)
			;
			\end{tikzpicture}
		\end{minipage}&$\Or_3$
		\\ \hline
		2&$\A_7, \A_8, \A_9^*$&$\langle \A_7,\A_8,\A_9^*, \A_1^*+\A_4^*,\A_2^*+\A_5^*, \A_3+\A_6 \rangle$&  
		\begin{minipage}{5cm}
			\begin{tikzpicture}           
			\matrix (m) [matrix of math nodes, row sep=0.5em,
			column sep=0.5em]{
				b^*&b&b^*\\
				a&a^*&a\\
				a^*&a&a^*\\
			};			
			\path[-stealth]
			(m-2-1) edge (m-1-1)
			(m-1-2) edge (m-1-1)
			(m-1-1) edge (m-2-2)
			(m-2-2) edge (m-1-2)
			(m-2-3) edge (m-1-3)		
			(m-1-2) edge (m-1-3)		
			(m-1-3) edge (m-2-2)
			(m-2-1) edge (m-3-1)
			(m-3-2) edge (m-3-1)
			(m-3-1) edge (m-2-2)
			(m-2-2) edge (m-3-2)
			(m-2-3) edge (m-3-3)		
			(m-3-2) edge (m-3-3)		
			(m-3-3) edge (m-2-2)
			(m-2-2) edge (m-2-1)
			(m-2-2) edge (m-2-3)
			;
			\end{tikzpicture}
		\end{minipage}&$\Or_2$\\
		\hline
		2&$\A_3, \A_6, \A_9^*$&$\langle \A_3,\A_6,\A_9^*, \A_1^*+\A_2^*,\A_4^*+\A_5^*, \A_7+\A_8  \rangle$&  
		\begin{minipage}{5cm}
			\begin{tikzpicture}           
			\matrix (m) [matrix of math nodes, row sep=0.5em,
			column sep=0.5em]{
				b^*&a&a^*\\
				b&a^*&a\\
				b^*&a&a^*\\
			};			
			\path[-stealth]
			(m-2-1) edge (m-1-1)
			(m-1-2) edge (m-1-1)
			(m-1-1) edge (m-2-2)
			(m-2-2) edge (m-1-2)
			(m-2-3) edge (m-1-3)		
			(m-1-2) edge (m-1-3)		
			(m-1-3) edge (m-2-2)
			(m-2-1) edge (m-3-1)
			(m-3-2) edge (m-3-1)
			(m-3-1) edge (m-2-2)
			(m-2-2) edge (m-3-2)
			(m-2-3) edge (m-3-3)		
			(m-3-2) edge (m-3-3)		
			(m-3-3) edge (m-2-2)
			(m-2-2) edge (m-2-1)
			(m-2-2) edge (m-2-3)
			;
			\end{tikzpicture}
		\end{minipage}&$\Or_2$\\
		\hline
		
		3&
		\begin{minipage}{1cm}
			\begin{align*}
			\A_3, \A_6\\	
			\A_7, \A_8\\
			\A_9^*
			\end{align*}
		\end{minipage}
		& 
		\begin{minipage}{4cm}
			\begin{align*}
			\langle \A_3, \A_6, \A_9^*, \A_1^*+\A_2^*, \A_7+\A_8, \A_4^*+\A_5^*\rangle\\
			\langle \A_7, \A_8, \A_9^*, \A_1^*+\A_4^*, \A_2^*+\A_5^*, \A_3+\A_6\rangle\\
			\langle \A_9^*, \A_3+\A_6, \A_7+\A_8, \A_1^*+\A_2^*+\A_4^*+\A_5^*\rangle
			\end{align*}
		\end{minipage}
		&
		\begin{minipage}{5cm}
			\begin{tikzpicture}           
			\matrix (m) [matrix of math nodes, row sep=0.5em,
			column sep=0.5em]{
				a^*&b&a^*\\
				c&a^*&c\\
				a^*&b&a^*\\
			};			
			\path[-stealth]
			(m-2-1) edge (m-1-1)
			(m-1-2) edge (m-1-1)
			(m-1-1) edge (m-2-2)
			(m-2-2) edge (m-1-2)
			(m-2-3) edge (m-1-3)		
			(m-1-2) edge (m-1-3)		
			(m-1-3) edge (m-2-2)
			(m-2-1) edge (m-3-1)
			(m-3-2) edge (m-3-1)
			(m-3-1) edge (m-2-2)
			(m-2-2) edge (m-3-2)
			(m-2-3) edge (m-3-3)		
			(m-3-2) edge (m-3-3)		
			(m-3-3) edge (m-2-2)
			(m-2-2) edge (m-2-1)
			(m-2-2) edge (m-2-3)
			;
			\end{tikzpicture}
		\end{minipage}&$\Or_3$
		\\ \hline
		3&$\A_7, \A_8, \A_9^*$&$\langle \A_7,\A_8,\A_9^*, \A_1^*+\A_4^*,\A_2^*+\A_5^*, \A_3+\A_6 \rangle$&  
		\begin{minipage}{5cm}
			\begin{tikzpicture}           
			\matrix (m) [matrix of math nodes, row sep=0.5em,
			column sep=0.5em]{
				a^*&a&a^*\\
				c&c^*&c\\
				b^*&b&b^*\\
			};			
			\path[-stealth]
			(m-2-1) edge (m-1-1)
			(m-1-2) edge (m-1-1)
			(m-1-1) edge (m-2-2)
			(m-2-2) edge (m-1-2)
			(m-2-3) edge (m-1-3)		
			(m-1-2) edge (m-1-3)		
			(m-1-3) edge (m-2-2)
			(m-2-1) edge (m-3-1)
			(m-3-2) edge (m-3-1)
			(m-3-1) edge (m-2-2)
			(m-2-2) edge (m-3-2)
			(m-2-3) edge (m-3-3)		
			(m-3-2) edge (m-3-3)		
			(m-3-3) edge (m-2-2)
			(m-2-2) edge (m-2-1)
			(m-2-2) edge (m-2-3)
			;
			\end{tikzpicture}
			\begin{tikzpicture}           
			\matrix (m) [matrix of math nodes, row sep=0.5em,
			column sep=0.5em]{
				a^*&c&a^*\\
				b&a^*&b\\
				b^*&a&b^*\\
			};
			\path[-stealth]
			(m-2-1) edge (m-1-1)
			(m-1-2) edge (m-1-1)
			(m-1-1) edge (m-2-2)
			(m-2-2) edge (m-1-2)
			(m-2-3) edge (m-1-3)		
			(m-1-2) edge (m-1-3)		
			(m-1-3) edge (m-2-2)
			(m-2-1) edge (m-3-1)
			(m-3-2) edge (m-3-1)
			(m-3-1) edge (m-2-2)
			(m-2-2) edge (m-3-2)
			(m-2-3) edge (m-3-3)		
			(m-3-2) edge (m-3-3)		
			(m-3-3) edge (m-2-2)
			(m-2-2) edge (m-2-1)
			(m-2-2) edge (m-2-3)
			;
			\end{tikzpicture}
		\end{minipage}&$\Or_2$
		\\ \hline
		3&$\A_3, \A_6, \A_9^*$&$\langle \A_3,\A_6,\A_9^*, \A_1^*+\A_2^*, \A_4^*+\A_5^*, \A_7+\A_8 \rangle$&  
		\begin{minipage}{5cm}
			\begin{tikzpicture}           
			\matrix (m) [matrix of math nodes, row sep=0.5em,
			column sep=0.5em]{
				a^*&c&b^*\\
				a&c^*&b\\
				a^*&c&b^*\\
			};			
			\path[-stealth]
			(m-2-1) edge (m-1-1)
			(m-1-2) edge (m-1-1)
			(m-1-1) edge (m-2-2)
			(m-2-2) edge (m-1-2)
			(m-2-3) edge (m-1-3)		
			(m-1-2) edge (m-1-3)		
			(m-1-3) edge (m-2-2)
			(m-2-1) edge (m-3-1)
			(m-3-2) edge (m-3-1)
			(m-3-1) edge (m-2-2)
			(m-2-2) edge (m-3-2)
			(m-2-3) edge (m-3-3)		
			(m-3-2) edge (m-3-3)		
			(m-3-3) edge (m-2-2)
			(m-2-2) edge (m-2-1)
			(m-2-2) edge (m-2-3)
			;
			\end{tikzpicture}
			\begin{tikzpicture}           
			\matrix (m) [matrix of math nodes, row sep=0.5em,
			column sep=0.5em]{
				a^*&b&b^*\\
				c&a^*&a\\
				a^*&b&b^*\\
			};			
			\path[-stealth]
			(m-2-1) edge (m-1-1)
			(m-1-2) edge (m-1-1)
			(m-1-1) edge (m-2-2)
			(m-2-2) edge (m-1-2)
			(m-2-3) edge (m-1-3)		
			(m-1-2) edge (m-1-3)		
			(m-1-3) edge (m-2-2)
			(m-2-1) edge (m-3-1)
			(m-3-2) edge (m-3-1)
			(m-3-1) edge (m-2-2)
			(m-2-2) edge (m-3-2)
			(m-2-3) edge (m-3-3)		
			(m-3-2) edge (m-3-3)		
			(m-3-3) edge (m-2-2)
			(m-2-2) edge (m-2-1)
			(m-2-2) edge (m-2-3)
			;
			\end{tikzpicture}
		\end{minipage}&$\Or_2$
		\\ \hline
		4&
		\begin{minipage}{1cm}
			\begin{align*}
			\A_3, \A_6\\
			\A_7, \A_8\\
			\A_9^*
			\end{align*}
		\end{minipage}
		& 
		\begin{minipage}{4cm}
			\begin{align*}
			\langle \A_3, \A_6, \A_9^*, \A_1^*+\A_2^*, \A_7+\A_8, \A_4^*+\A_5^*\rangle\\	
			\langle \A_7, \A_8, \A_9^*, \A_1^*+\A_4^*, \A_2^*+\A_5^*, \A_3+\A_6\rangle\\
			\langle \A_9^*, \A_3+\A_6, \A_7+\A_8, \A_1^*+\A_2^*+\A_4^*+\A_5^*\rangle
			\end{align*}
		\end{minipage}
		&
		\begin{minipage}{5cm}
			\begin{tikzpicture}           
			\matrix (m) [matrix of math nodes, row sep=0.5em,
			column sep=0.5em]{
				a^*&b&a^*\\
				c&d^*&c\\
				a^*&b&a^*\\
			};			
			\path[-stealth]
			(m-2-1) edge (m-1-1)
			(m-1-2) edge (m-1-1)
			(m-1-1) edge (m-2-2)
			(m-2-2) edge (m-1-2)
			(m-2-3) edge (m-1-3)		
			(m-1-2) edge (m-1-3)		
			(m-1-3) edge (m-2-2)
			(m-2-1) edge (m-3-1)
			(m-3-2) edge (m-3-1)
			(m-3-1) edge (m-2-2)
			(m-2-2) edge (m-3-2)
			(m-2-3) edge (m-3-3)		
			(m-3-2) edge (m-3-3)		
			(m-3-3) edge (m-2-2)
			(m-2-2) edge (m-2-1)
			(m-2-2) edge (m-2-3)
			;		
			\end{tikzpicture}
		\end{minipage}&$\Or_3$\\
		\hline
		4&$\A_7, \A_8, \A_9^*$&$\langle \A_7,\A_8,\A_9^*, \A_1^*+\A_4^*,\A_2^*+\A_5^*, \A_3+\A_6 \rangle$&  
		\begin{minipage}{5cm}
			\begin{tikzpicture}           
			\matrix (m) [matrix of math nodes, row sep=0.5em,
			column sep=0.5em]{
				c^*&d&c^*\\
				a&b^*&a\\
				a^*&b&a^*\\
			};			
			\path[-stealth]
			(m-2-1) edge (m-1-1)
			(m-1-2) edge (m-1-1)
			(m-1-1) edge (m-2-2)
			(m-2-2) edge (m-1-2)
			(m-2-3) edge (m-1-3)		
			(m-1-2) edge (m-1-3)		
			(m-1-3) edge (m-2-2)
			(m-2-1) edge (m-3-1)
			(m-3-2) edge (m-3-1)
			(m-3-1) edge (m-2-2)
			(m-2-2) edge (m-3-2)
			(m-2-3) edge (m-3-3)		
			(m-3-2) edge (m-3-3)		
			(m-3-3) edge (m-2-2)
			(m-2-2) edge (m-2-1)
			(m-2-2) edge (m-2-3)
			;
			\end{tikzpicture}
			\begin{tikzpicture}           
			\matrix (m) [matrix of math nodes, row sep=0.5em,
			column sep=0.5em]{
				a^*&d&a^*\\
				c&b^*&c\\
				b^*&a&b^*\\
			};			
			\path[-stealth]
			(m-2-1) edge (m-1-1)
			(m-1-2) edge (m-1-1)
			(m-1-1) edge (m-2-2)
			(m-2-2) edge (m-1-2)
			(m-2-3) edge (m-1-3)		
			(m-1-2) edge (m-1-3)		
			(m-1-3) edge (m-2-2)
			(m-2-1) edge (m-3-1)
			(m-3-2) edge (m-3-1)
			(m-3-1) edge (m-2-2)
			(m-2-2) edge (m-3-2)
			(m-2-3) edge (m-3-3)		
			(m-3-2) edge (m-3-3)		
			(m-3-3) edge (m-2-2)
			(m-2-2) edge (m-2-1)
			(m-2-2) edge (m-2-3)
			;
			\end{tikzpicture}
		\end{minipage}&$\Or_2$
		\\ \hline
		4&$\A_3, \A_6, \A_9^*$&$\langle \A_3,\A_6,\A_9^*, \A_1^*+\A_2^*, \A_4^*+\A_5^*, \A_7+\A_8 \rangle$&  
		\begin{minipage}{5cm}
			\begin{tikzpicture}           
			\matrix (m) [matrix of math nodes, row sep=0.5em,
			column sep=0.5em]{
				b^*&a&a^*\\
				d&c^*&c\\
				b^*&a&a^*\\
			};			
			\path[-stealth]
			(m-2-1) edge (m-1-1)
			(m-1-2) edge (m-1-1)
			(m-1-1) edge (m-2-2)
			(m-2-2) edge (m-1-2)
			(m-2-3) edge (m-1-3)		
			(m-1-2) edge (m-1-3)		
			(m-1-3) edge (m-2-2)
			(m-2-1) edge (m-3-1)
			(m-3-2) edge (m-3-1)
			(m-3-1) edge (m-2-2)
			(m-2-2) edge (m-3-2)
			(m-2-3) edge (m-3-3)		
			(m-3-2) edge (m-3-3)		
			(m-3-3) edge (m-2-2)
			(m-2-2) edge (m-2-1)
			(m-2-2) edge (m-2-3)
			;
			\end{tikzpicture}
			\begin{tikzpicture}           
			\matrix (m) [matrix of math nodes, row sep=0.5em,
			column sep=0.5em]{
				a^*&c&b^*\\
				d&b^*&a\\
				a^*&c&b^*\\
			};			
			\path[-stealth]
			(m-2-1) edge (m-1-1)
			(m-1-2) edge (m-1-1)
			(m-1-1) edge (m-2-2)
			(m-2-2) edge (m-1-2)
			(m-2-3) edge (m-1-3)		
			(m-1-2) edge (m-1-3)		
			(m-1-3) edge (m-2-2)
			(m-2-1) edge (m-3-1)
			(m-3-2) edge (m-3-1)
			(m-3-1) edge (m-2-2)
			(m-2-2) edge (m-3-2)
			(m-2-3) edge (m-3-3)		
			(m-3-2) edge (m-3-3)		
			(m-3-3) edge (m-2-2)
			(m-2-2) edge (m-2-1)
			(m-2-2) edge (m-2-3)
			;
			\end{tikzpicture}
		\end{minipage}&$\Or_2$
		\\ \hline
		
		5&$\A_7, \A_8, \A_9^*$&$\langle \A_7,\A_8,\A_9^*, \A_1^*+\A_4^*,\A_2^*+\A_5^*, \A_3+\A_6 \rangle$&  
		\begin{minipage}{5cm}
			\begin{tikzpicture}           
			\matrix (m) [matrix of math nodes, row sep=0.5em,
			column sep=0.5em]{
				a^*&d&a^*\\
				c&e^*&c\\
				b^*&a&b^*\\
			};			
			\path[-stealth]
			(m-2-1) edge (m-1-1)
			(m-1-2) edge (m-1-1)
			(m-1-1) edge (m-2-2)
			(m-2-2) edge (m-1-2)
			(m-2-3) edge (m-1-3)		
			(m-1-2) edge (m-1-3)		
			(m-1-3) edge (m-2-2)
			(m-2-1) edge (m-3-1)
			(m-3-2) edge (m-3-1)
			(m-3-1) edge (m-2-2)
			(m-2-2) edge (m-3-2)
			(m-2-3) edge (m-3-3)		
			(m-3-2) edge (m-3-3)		
			(m-3-3) edge (m-2-2)
			(m-2-2) edge (m-2-1)
			(m-2-2) edge (m-2-3)
			;
			\end{tikzpicture}\begin{tikzpicture}           
			\matrix (m) [matrix of math nodes, row sep=0.5em,
			column sep=0.5em]{
				a^*&d&a^*\\
				c&a^*&c\\
				b^*&e&b^*\\
			};			
			\path[-stealth]
			(m-2-1) edge (m-1-1)
			(m-1-2) edge (m-1-1)
			(m-1-1) edge (m-2-2)
			(m-2-2) edge (m-1-2)
			(m-2-3) edge (m-1-3)		
			(m-1-2) edge (m-1-3)		
			(m-1-3) edge (m-2-2)
			(m-2-1) edge (m-3-1)
			(m-3-2) edge (m-3-1)
			(m-3-1) edge (m-2-2)
			(m-2-2) edge (m-3-2)
			(m-2-3) edge (m-3-3)		
			(m-3-2) edge (m-3-3)		
			(m-3-3) edge (m-2-2)
			(m-2-2) edge (m-2-1)
			(m-2-2) edge (m-2-3)
			;
			\end{tikzpicture}\begin{tikzpicture}           
			\matrix (m) [matrix of math nodes, row sep=0.5em,
			column sep=0.5em]{
				b^*&d&b^*\\
				c&a^*&c\\
				a^*&e&a^*\\
			};			
			\path[-stealth]
			(m-2-1) edge (m-1-1)
			(m-1-2) edge (m-1-1)
			(m-1-1) edge (m-2-2)
			(m-2-2) edge (m-1-2)
			(m-2-3) edge (m-1-3)		
			(m-1-2) edge (m-1-3)		
			(m-1-3) edge (m-2-2)
			(m-2-1) edge (m-3-1)
			(m-3-2) edge (m-3-1)
			(m-3-1) edge (m-2-2)
			(m-2-2) edge (m-3-2)
			(m-2-3) edge (m-3-3)		
			(m-3-2) edge (m-3-3)		
			(m-3-3) edge (m-2-2)
			(m-2-2) edge (m-2-1)
			(m-2-2) edge (m-2-3)
			;	
			\end{tikzpicture}
		\end{minipage}&$\Or_2$
		\\ \hline
		
		5&$\A_3, \A_6, \A_9^*$&$\langle \A_3,\A_6,\A_9^*, \A_1^*+\A_2^*, \A_4^*+\A_5^*, \A_7+\A_8 \rangle$&  
		\begin{minipage}{5cm}
			\begin{tikzpicture}           
			\matrix (m) [matrix of math nodes, row sep=0.5em,
			column sep=0.5em]{
				a^*&c&b^*\\
				d&e^*&a\\
				a^*&c&b^*\\
			};			
			\path[-stealth]
			(m-2-1) edge (m-1-1)
			(m-1-2) edge (m-1-1)
			(m-1-1) edge (m-2-2)
			(m-2-2) edge (m-1-2)
			(m-2-3) edge (m-1-3)		
			(m-1-2) edge (m-1-3)		
			(m-1-3) edge (m-2-2)
			(m-2-1) edge (m-3-1)
			(m-3-2) edge (m-3-1)
			(m-3-1) edge (m-2-2)
			(m-2-2) edge (m-3-2)
			(m-2-3) edge (m-3-3)		
			(m-3-2) edge (m-3-3)		
			(m-3-3) edge (m-2-2)
			(m-2-2) edge (m-2-1)
			(m-2-2) edge (m-2-3)
			;	
			\end{tikzpicture}\begin{tikzpicture}           
			\matrix (m) [matrix of math nodes, row sep=0.5em,
			column sep=0.5em]{
				a^*&c&b^*\\
				d&a^*&e\\
				a^*&c&b^*\\
			};			
			\path[-stealth]
			(m-2-1) edge (m-1-1)
			(m-1-2) edge (m-1-1)
			(m-1-1) edge (m-2-2)
			(m-2-2) edge (m-1-2)
			(m-2-3) edge (m-1-3)		
			(m-1-2) edge (m-1-3)		
			(m-1-3) edge (m-2-2)
			(m-2-1) edge (m-3-1)
			(m-3-2) edge (m-3-1)
			(m-3-1) edge (m-2-2)
			(m-2-2) edge (m-3-2)
			(m-2-3) edge (m-3-3)		
			(m-3-2) edge (m-3-3)		
			(m-3-3) edge (m-2-2)
			(m-2-2) edge (m-2-1)
			(m-2-2) edge (m-2-3)
			;
			\end{tikzpicture}\begin{tikzpicture}           
			\matrix (m) [matrix of math nodes, row sep=0.5em,
			column sep=0.5em]{
				b^*&c&a^*\\
				d&a^*&e\\
				b^*&c&a^*\\
			};			
			\path[-stealth]
			(m-2-1) edge (m-1-1)
			(m-1-2) edge (m-1-1)
			(m-1-1) edge (m-2-2)
			(m-2-2) edge (m-1-2)
			(m-2-3) edge (m-1-3)		
			(m-1-2) edge (m-1-3)		
			(m-1-3) edge (m-2-2)
			(m-2-1) edge (m-3-1)
			(m-3-2) edge (m-3-1)
			(m-3-1) edge (m-2-2)
			(m-2-2) edge (m-3-2)
			(m-2-3) edge (m-3-3)		
			(m-3-2) edge (m-3-3)		
			(m-3-3) edge (m-2-2)
			(m-2-2) edge (m-2-1)
			(m-2-2) edge (m-2-3)
			;
			\end{tikzpicture}
		\end{minipage}&$\Or_2$
		\\ \hline
		
		6&$\A_7, \A_8, \A_9^*$&$\langle \A_7,\A_8,\A_9^*, \A_1^*+\A_4^*,\A_2^*+\A_5^*, \A_3+\A_6 \rangle$&  
		\begin{minipage}{5cm}
			\begin{tikzpicture}           
			\matrix (m) [matrix of math nodes, row sep=0.5em,
			column sep=0.5em]{
				a^*&d&a^*\\
				c&e^*&c\\
				b^*&f&b^*\\
			};			
			\path[-stealth]
			(m-2-1) edge (m-1-1)
			(m-1-2) edge (m-1-1)
			(m-1-1) edge (m-2-2)
			(m-2-2) edge (m-1-2)
			(m-2-3) edge (m-1-3)		
			(m-1-2) edge (m-1-3)		
			(m-1-3) edge (m-2-2)
			(m-2-1) edge (m-3-1)
			(m-3-2) edge (m-3-1)
			(m-3-1) edge (m-2-2)
			(m-2-2) edge (m-3-2)
			(m-2-3) edge (m-3-3)		
			(m-3-2) edge (m-3-3)		
			(m-3-3) edge (m-2-2)
			(m-2-2) edge (m-2-1)
			(m-2-2) edge (m-2-3)
			;
			\end{tikzpicture}
		\end{minipage}&$\Or_2$\\ \hline
		
		6&$\A_3, \A_6, \A_9^*$&$\langle \A_3,\A_6,\A_9^*, \A_1^*+\A_2^*, \A_4^*+\A_5^*, \A_7+\A_8 \rangle$&  
		\begin{minipage}{5cm}
			\begin{tikzpicture}           
			\matrix (m) [matrix of math nodes, row sep=0.5em,
			column sep=0.5em]{
				a^*&c&b^*\\
				d&e^*&f\\
				a^*&c&b^*\\
			};			
			\path[-stealth]
			(m-2-1) edge (m-1-1)
			(m-1-2) edge (m-1-1)
			(m-1-1) edge (m-2-2)
			(m-2-2) edge (m-1-2)
			(m-2-3) edge (m-1-3)		
			(m-1-2) edge (m-1-3)		
			(m-1-3) edge (m-2-2)
			(m-2-1) edge (m-3-1)
			(m-3-2) edge (m-3-1)
			(m-3-1) edge (m-2-2)
			(m-2-2) edge (m-3-2)
			(m-2-3) edge (m-3-3)		
			(m-3-2) edge (m-3-3)		
			(m-3-3) edge (m-2-2)
			(m-2-2) edge (m-2-1)
			(m-2-2) edge (m-2-3)
			;
			\end{tikzpicture}
		\end{minipage}&$\Or_2$
		\\ \hline
	\end{tabular}
	\caption{Monodromy for $h(x)+g(y)\in \R[x,y]_{d\leq 4}$ and Dynkin diagram (\ref{Dynkincase2})}
	\label{case4b}
\end{table}

\clearpage
Let $f(x,y)=h(y)+g(y)$ be a polynomial of degree 4,  we consider the vector space $V_f=H_1(f^{-1}(b), \Q)$ with the basis given by the vanishing cycles $\{\alpha_i\}_{i=1,\ldots, 9}$. Let  $G_f=\pi_1(\C \setminus\{a_1,\ldots, a_9\})$ be a  free group acting on $V_f$ by monodromy. For polynomials $f,f'$, we  relate $f$ and $f'$ if there is a permutation $\varphi$ of the set $\{\alpha_i\}_{i=1,\ldots, 9}$, such that
$$\text{span}(G_f\cdot (\varphi(\alpha_i))=\varphi(\text{span}(G_{f'}\cdot \alpha_i))\text{,  for }i=1,\ldots,9.$$
Let us introduce the  suggestive notation $v_{11}, v_{12}, v_{13},$ $   v_{21}, v_{22}, v_{23}, v_{31}, v_{32}, v_{33}$ as a basis for $V_f$. Then, we compute the equivalence classes $[f]$ of the polynomials in $\R[x]_{\leq 4}\oplus \R[y]_{\leq 4}$ with real critical points. From Tables \ref{case2b} and \ref{case4b}, we conclude that there are 5 equivalence classes of these polynomials, they are
\begin{small}
\begin{itemize}
	\item $\Or_0$: 
	
	$\text{span}(G_f\cdot v_{ij})=V_f$, for $i,j=1,2,3.$
	
	\item $\Or_1$: In this cases  $G_f$ is a free group generated by a matrix $M$, and 
	
	$\text{span}(G_f\cdot v_{ij})=\langle M^k v_{ij}\rangle $ with $k=0,\ldots, 4$ and $(i,j)\neq (2,2)$.
	
	$\text{span}(G_f\cdot v_{22})=\langle M^k v_{ij}\rangle $ with $k=0,\ldots, 2.$ 
	
	\item $\Or_2$:
	
	$\text{span}(G_f\cdot v_{21})=\text{span}(G_f\cdot v_{22})=\text{span}(G_f \cdot v_{23})=\langle v_{21}, v_{22}, v_{23}, v_{11}+v_{31}, v_{12}+v_{32}, v_{13}+v_{33} \rangle.$
	
	$\text{span}(G_f\cdot v_{ij})=V_f$, in other cases.
	
	\item $\Or_3$:
	
	$\text{span}(G_f\cdot v_{21})=\text{span}(G_f\cdot v_{23})=\langle v_{21}, v_{22}, v_{23}, v_{11}+v_{31}, v_{12}+v_{32}, v_{13}+v_{33}.  \rangle.$
	
	$\text{span}(G_f\cdot v_{12})=\text{span}(G_f\cdot v_{32})=\langle v_{12}, v_{22}, v_{32}, v_{11}+v_{13}, v_{21}+v_{23}, v_{31}+v_{33} \rangle.$
	
	$\text{span}(G_f\cdot v_{22})=\langle v_{22}, v_{12}+v_{32}, v_{21}+v_{23}, v_{11}+v_{13}+v_{31}+v_{33} \rangle.$
	
	$\text{span}(G_f\cdot v_{ij})=V_f$, in other cases.
	
	
	
	
	\item 
	
	$\Or_4$:	
	
	$\text{span}(G_f\cdot v_{21})=\text{span}(G_f\cdot v_{23})=\langle v_{21}, v_{22}, v_{23}, v_{11}+v_{31}, v_{12}+v_{32}, v_{13}+v_{33}.  \rangle.$	
	
	$\text{span}(G_f\cdot v_{12})=\text{span}(G_f\cdot v_{32})=\langle v_{12}, v_{22}, v_{32}, v_{11}+v_{13}, v_{21}+v_{23}, v_{31}+v_{33} \rangle.$
	
	$\text{span}(G_f\cdot v_{11})=\text{span}(G_f\cdot v_{33})=\langle v_{11}, v_{22}, v_{33}, v_{12}-v_{21}, v_{23}-v_{32}, v_{13}+v_{31}, v_{12}+v_{21}+v_{23}+v_{32} \rangle.$
	
	$\text{span}(G_f\cdot v_{13})= \text{span}(G_f\cdot v_{31})=\langle v_{13}, v_{22}, v_{31}, v_{21}-v_{32}, v_{12}-v_{23}, v_{11}+v_{33}, v_{12}+v_{21}+v_{23}+v_{32} \rangle.$	
	
	$\text{span}(G_f\cdot v_{22})=\langle  v_{22} , v_{12}+v_{32}, v_{21}+v_{23}, v_{11}+v_{13}+v_{31}+v_{33} \rangle.$
\end{itemize}
\end{small}
These equivalence classes of polynomial $f(x,y)=h(x)+g(y)$ in terms of the subspaces generated by the orbit of monodromy action, can be written in terms of the polynomials $h$ and $g$. That is showed in the  theorem \ref{orbitspacehg}. We also consider polynomials up to linear transformation, because these  do not change the monodromy action. For a polynomial  $h\in\C[x]_{\leq d}$ and a partition $(d_1, d_2, \ldots, d_M)$ of $d-1$, we say that $h$ has \textbf{critical values degree} $(d_1, d_2, \ldots, d_M)$ if it has $M$ different critical values $c_1, c_2,\ldots, c_M$ and for any $i=1\ldots, M$ there are $d_i$ critical points over $c_i$, counted with multiplicity. The next lemma is proved in appendix \ref{Appenix1}, and we  give an algorithm to compute the ideals.
\begin{lemma}
\label{subalgevaiertyId}
	Given a positive integer $d$ and a partition $(d_1, d_2,\ldots, d_M)$ of $d-1$, the set of polynomials in $\C[x]_{\leq d}$ with critical values degree $(d_1, d_2, \ldots, d_M)$ is an algebraic subvariety of $\C[x]_{\leq d}$. Its ideal associated is denoted by $I_{(d_1, d_2, \ldots, d_M)}$.
\end{lemma}
For example fo $d=4$, we consider $h(x):=x^4+r_3x^3+r_2x^2+r_1x+r_0$. Since translations in the abscissa and in the ordinate do not change the monodromy action, we can suppose that $h(x):=x^4+r_2x^2+r_1x$. Thus, we have
\begin{align*}
I_{(3,0)}&=\langle r_1, r_2 \rangle,\\
I_{(2,1)}&=\langle r_1\rangle \cap \langle 27 r_1^2+8r_2^3\rangle=:\langle r_1\rangle \cap \langle H\rangle,\\
I_{(1,1,1)}&=0.
\end{align*}
Another interesting example, is when we consider the polynomials $h(x)=x^4+r_2x^2+r_1x$, $g(y)=y^4+s_2y^2+s_1y$ and we suppose that the critical values of $h$ are equal to the critical values of $g$. In this case, we have a subvariety of $\C[x,y]_{\leq d}$ given by the ideal $$\langle r_2-s_2, r_1-s_1\rangle \cap \langle r_2+s_2, r_1^2+s_1^2\rangle \cap \langle r_2-s_2,r_1+s_1\rangle \cap \langle s_2,s_1,r_2,r_1\rangle.$$
\begin{theorem}
	\label{orbitspacehg}
	Let $f(x,y)=h(x)+g(y)$, where  $h\in\R[x]_{\leq 4}$ and $g\in \R[y]_{\leq 4}$ are polynomials with real critical points. There is a characterization of the equivalence class of $f$ in terms of $h,g$ as follows, 
	\begin{enumerate}
		\item $[f]\in\Or_{1}$ iff $f(x,y)=x^4+y^4$.
		\item $[f]\in \Or_{2}$ iff $f(x,y)=(h_2\circ h_1)(x)+g(y)$, where $h_1,h_2\in \R[x]_{\leq 2}$ and  $g$ is  not decomposable.
		\item $[f]\in\Or_{3}$ iff $f(x,y)=(h_2\circ h_1)(x)+(g_2\circ g_1)(y)$, where $h_1,h_2\in \R[x]_{\leq 2}$, $g_1,g_2\in \R[y]_{\leq 2}$.
		\item $[f]\in\Or_4$ iff $f(x,y)=(h_2\circ h_1)(x)+(h_2\circ h_1)(\pm y)$, where $h_1,h_2\in \R[x]_{\leq 2}$.
		\item $[f]\in\Or_{0}$ iff $h(x)$ and $g(y)$ are   not decomposable.
	\end{enumerate}
\end{theorem}

\begin{proof}
	It is easy to show that the conditions on $h$ and $g$ are sufficient conditions. Following, we show that they are necessary conditions. Since we consider polynomials up to linear transformation, we can suppose $h(x)=x^4+r_2x^2+r_1x$, $g(y)=y^4+s_2y^2+s_1y$. Thus, for $h$ and $g$ be decomposable polynomials it is necessary that $r_1=0$ and $s_1=0$, respectively. 
	\begin{enumerate}
		\item For $[f]\in\Or_1$, the polynomial $f(x,y)$ has a critical value, thus $h$ and $g$ have only one critical value. Since $h\in I_{(3,0)}$, then $h(x)=x^4$, analogously for $g$.
		
		\item When  $[f]\in\Or_2$ we have the next possibilities: If the Dynkin diagram is (\ref{Dynkincase1}), then the critical values satisfy $a_1=a_4, a_2=a_5, a_3=a_6$ or $a_2=a_3, a_5=a_6, a_8=a_9$. If the Dynkin diagram is (\ref{Dynkincase2}), then the critical values satisfy $a_1=a_4, a_2=a_5, a_3=a_6$ or $a_1=a_2, a_4=a_5, a_7=a_8$. The first conditions in both Dynkin diagrams implies that $c_1^h=c_2^h$, the others conditions implies  $c_2^g=c_3^g$ and $c_1^g=c_3^g$, respectively.  Without loss of generality we consider $c_1^h=c_2^h$, then $h\in I_{(2,1)}$, and recall that the 0-dimensional Dynkin diagram for $h$ in this case is $\gamma_1\cdots \gamma_3\cdots  \gamma_2$.
		
		Furthermore, the discriminant of $h'(x)$ is equal to $-16H$, thus $\textbf{V}(H)$  corresponds to the polynomials with at most 2 critical points. Hence, the polynomials in $\textbf{V}(H)\setminus \textbf{V}(I_{(3,0)})$ are polynomials that have two different critical values and two critical points, and it is not the case of $c_1^h=c_2^h$ and $c_1^h\neq c_3^h$ see the Figure \ref{2critical}. Therefore, the polynomials in $\Or_2$ satisfy that $h(x)=x^4+r_2x^2$. Then $h(x)=h_2(h_1(x))$ where $h_1(x)=x^2$ and $h_2(x)=x(x+r_2)$.
		
		\item If $[f]\in \Or_3$, then the conditions $c_1^h=c_2^h$ and $c_1^g=c_2^g$, with 0-dimensional Dynkin diagram  $\delta_1\cdots \delta_3\cdots \delta_2$ associated to $g$, are satisfied simultaneously (or $c_2^g=c_3^g$, with 0-dimensional Dynkin diagram $\delta_2\cdots \delta_1\cdots \delta_3$ associated to $g$). Hence, analogously to the previous case we have  $f(x,y)=h_2(h_1(x))+g_2(g_1(y))$ where $h_1(x)=x^2$, $h_2(x)=x(x+r_2)$, $g_1(y)=y^2$, $g_2(y)=y(y+s_2)$.
		
		\item  For $[f]\in \Or_4$,  the Dynkin diagram is (\ref{Dynkincase1}) and the critical values of $f$ satisfies the relations $a_1=a_9$, $a_2=a_6$ and $a_4=a_8$, that means $c_1^h+c_1^g=c_3^h+c_3^g$, $c_1^h+c_2^g=c_2^h+c_3^g$ and  $c_2^h+c_1^g=c_3^h+c_2^g$. Thus $c_1^h=c_3^g+k$, $c_2^h=c_2^g+k$ and $c_3^h=c_1^g+k$, where $k=c_ 3^h-c_1^g$. Hence, by doing a translation, we can suppose that the critical values of the polynomial $h$ are equals to the critical values of $g$. Therefore,  $g(x)=h(\pm x)$. 	
		On the other hand, the conditions $a_1=a_4$, $a_2=a_5$, $a_3=a_6$, implies that $h$ is decomposable.
%
		\item When $[f]\in\Or_{0}$, the critical values of $h$ are different or $h\in \textbf{V}(H)\setminus \textbf{V}(I_{(3,0)})$. Therefore, $r_1\neq 0$. Analogously for $g$, we conclude that $s_1\neq 0$. 
	\end{enumerate}
\end{proof}

\begin{figure}[h!]
	\centering
	\includegraphics[width=12cm, height=4cm]{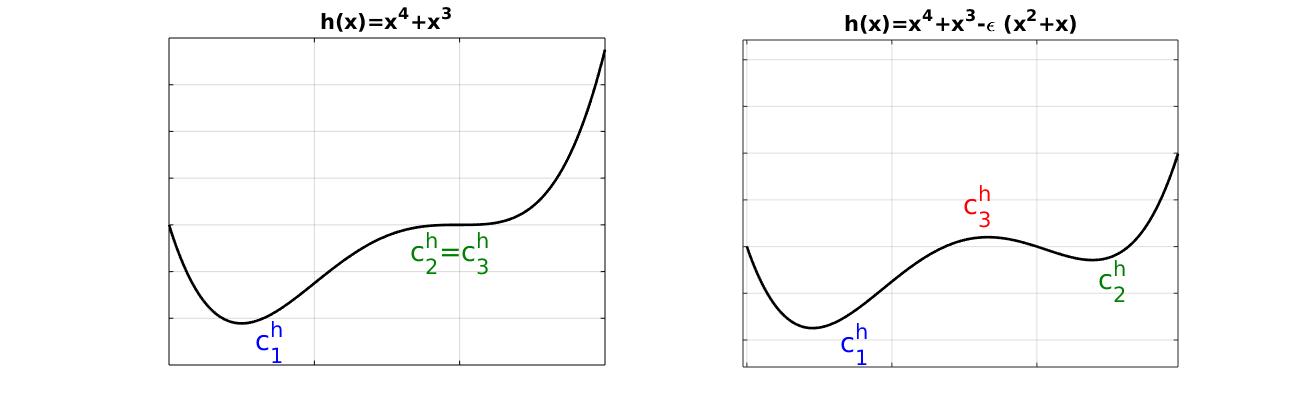}
	\caption{\footnotesize{In the left a polynomial $g\in \R[y]_{\leq 4}$ with two critical values and two critical points. In the right a perturbation of $g$ which separates the critical values. In both cases the enumeration is done according to  section \ref{dynkinrules}.} The vanishing cycle associated to $c_1^h$ and $c_3^h$ always intersect.}
		\label{2critical}
\end{figure}
\textbf{Remark.} Similar to Theorem \ref{ChThm}, if $[f]\in \Or_2,$ then the vanishing cycles  $v_{21}, v_{22}, v_{23}$ are in the kernel of the map 
\begin{equation*}
H_1(f^{-1}(b), \Q)\to H_1(\tilde f^{-1}(b), \Q),\hspace{4mm} \text{ where }\tilde f=h_2(x)+g(y)
\end{equation*} 
coming from the map  $(x,y)\to (h_1(x),y)$. If $[f]\in \Or_3$, then the vanishing cycles $v_{21}, v_{22}, v_{23}$ are as before, and the  vanishing cycles $v_{12}, v_{22}, v_{32}$ are in the kernel of the map 
\begin{equation*}
H_1(f^{-1}(b), \Q)\to H_1(\hat f^{-1}(b), \Q),\hspace{4mm} \text{ where }\hat f=h(x)+g_2(y)
\end{equation*} 
coming from the map  $(x,y)\to (x,g_1(y))$. The other not simple  vanishing cycles  appear with  the symmetry $h(x)=g(\pm 
x)$. Therefore, they may be related with the pullback \begin{align*}
\C^2&\longrightarrow\hspace{6mm}\C^2 \hspace{6mm}\longrightarrow \C\\
(x,y)&\to(x+y,xy)\to \check{f}(\check x, \check y)
\end{align*}
where $\check x=x+ y$, $\check y=xy$ and  some $\check{f}\in \R[\check x,\check y]_{\leq 4}$. So far we do not know a geometrical characterization for these vanishing cycles.

\appendix
\section{Appendix: Algebraic space  $I_{(d_1, d_2,\ldots, d_M)}$}
\label{Appenix1}
In this section we  show that the space of polynomials  $f(x)$ of degree $d$ with a given number of critical values is an algebraic subspace. Actually, we need other conditions in the cardinality of the critical values, it motivates the next definition.
\begin{definition}
	For an integer $d$ and a partition $(d_1, d_2, \ldots, d_M)$ of $d-1$, we say that the polynomial $f(x)\in \C[x]_{\leq d}$ has \textbf{critical values degree} $(d_1, d_2, \ldots, d_M)$ if it has $M$ different critical values $c_1, c_2,\ldots, c_M$ and for any $i=1\ldots, M$ there are $d_i$ critical points over $c_i$, counted with multiplicity.
	\end{definition}
	We show that  the condition of critical values degree for a polynomial can be given in terms of algebraic expressions.
		\begin{proof} [Proof of Lemma \ref{subalgevaiertyId}]
			By  definition of the \textit{discriminant} $\Delta$, see  \cite[\S 10.9]{hodgehossein}, we know that $\xi$ is a critical value of  $f(x)$ if and only if $\Delta(f(x)-\xi)=0$. For $f(x)=x^{d}+r_{d-1}x^{d-1}+\ldots +r_0$, we have that $\Delta_\xi(f):=\Delta(f-\xi)$ is a polynomial $\lambda(\xi)=\A_{d-1}\xi^{d-1}+\A_{d-2}\xi^{d-2}+\ldots +\A_0$. It defines the map
			\begin{align*}
			\C^{d}&\xto{\Delta_\xi} \C^{d-1}\\
			(r_{d-1},\ldots, r_0)&\to (\A_{d-2},\ldots, \A_0). 
			\end{align*}
			The polynomial $\lambda(\xi)$ can also be expressed in terms of the critical values $t_1,\ldots, t_{d-1}$ of $f$, as $\lambda(\xi)=(\xi-t_1)(\xi-t_2)\ldots(\xi-t_{d-1})=(\xi-t_{i_1})^{d_1}(\xi-t_{i_2})^{d_2}\ldots (\xi-t_{i_{M}})^{d_M}$, where we have used the definition of critical values degree. Let $\C^{d-1}\xto{\varphi} \C^{d-1}$ be the map given by the Vieta's formula 
			\begin{align*}
			\C^{d-1}&\xto{\varphi} \C^{d-1}\\
			(t_1, t_2, \ldots, t_{d-1})&\xto{\varphi}(\varepsilon_1\sum_{i}t_i, \varepsilon_2\sum_{i\neq j}t_it_j, \ldots, \varepsilon_{d-1}t_1t_2\ldots t_{d-1}),
			\end{align*}
			where $\varepsilon_j=(-1)^j\A_{d-1}$. Hence,  $\varphi$ take the roots of a polynomial and gives the coefficients of the polynomial. Let $V$ be a subvariety in the domain of $\varphi$ given by $M$ equations of the form $t_{i_1}=t_{i_2}\ldots =t_{i_{d_j}}$ with $j=1,\ldots, M$.  The subvariety $V$ has the information of the critical values degree.
			
			The closure of $\varphi(V)$ is a subvariety of $\C^{d-1}$, we denote it by $W$. The pullback of $W$ by the map $\Delta_\xi$ is a subvariety in $\C^{d}$ in terms of the parameters $r_k$ which is the closure of the space of polynomial in $\C[x]_{\leq d}$ with critical values degree $(d_1, d_2, \ldots d_M).$
			\begin{center}
				\begin{tikzpicture}
				\matrix (m) [matrix of math nodes, row sep=0.5em,
				column sep=1.5em]{
					V\subset \C^{d-1} & W\subset \C^{d-1}  \\ 
					\C^{d}		&	 \\};			
					\path[-stealth]
					(m-1-1) edge [bend right=0] node [above] {$\varphi$}  (m-1-2) 
					(m-2-1) edge [bend right=0] node [below] {$\Delta_\xi$}  (m-1-2) 
					;
					\end{tikzpicture}
					\end{center}
					\end{proof}
	In order to compute an explicit expression for $W$ we use the implicitation  algorithm \cite[\S 3.3]{idealsalgorithms}. Let $v_1, v_2,..., v_s$ be the polynomials which describes $V$, thus $v_i=v_i(t_1,\ldots, t_{d-1})$. Let  $I\subset \C[t_1,\ldots,t_{d-1}, x_1,\ldots,x_{d-1}]$ be the ideal 
	$$I=\langle z_1-\sum t_i, z_2-\sum_{i\neq j}t_it_j,...,z_{d-1}- t_1t_2...t_{d-1},v_1,..., v_s\rangle.$$ 
					
	If $G$ is a Groebner basis of $I$ with respect to lexicographic order $t_1>t_2>...>t_{d-1}>z_1> ...>z_{d-1}$, then $G_z=G\cap \C[z]$ is a Groebner basis of the ideal $I_z:=I\cap\C[z]$. Also $W:=\textbf{V}(I_z)$ is the smallest variety in $\C^{d-1}$ containing $\varphi(V)$.
\section{Numerical supplementary items}
\label{numericalsection}
In this section we provide the explanation of the codes used in the proof of Proposition \ref{freevanishingcycles}.  To start, it is necessary to get the MATLAB's functions MonMatrix and \mbox{VanCycleSub}. These are available in \href{https://github.com/danfelmath/Intersection-matrix-for-polynomials-with-1-crit-value.git}
{https://github.com/danfelmath/Intersection-matrix-for-polynomials-with-1-crit-value.git}.\\

The function MonMatrix computes the monodromy matrix for the polynomial              \begin{equation}
\label{generalpolfibration}
f:=y^{e}+x^d
\end{equation}
 in the basis described in \S \ref{Monfordirectsumonecrivalue}. That is, for each by considering a perturbation of $y^e$ and $x^d$, such that the critical values are different. Thus, we have a real curve similar to the Figure \ref{realcriticalvalues}.  Moreover, we can suppose that the critical points induce the  0-Dynkin diagrams 
\begin{equation*}
	\begin{tikzpicture}           
	\matrix (m) [matrix of math nodes, row sep=0.7em,
	column sep=0.7em]{
\sigma_{l_d+1}&\sigma_1 &\sigma^1_{l_d+2} & \sigma_2& \sigma_{l_d+3} & \sigma_3&\cdots \\
\gamma_{l_e+1}&\gamma_1 &\gamma^1_{l_e+2} & \gamma_2& \gamma_{l_e+3} & \gamma_3&\cdots \\
	};			
	\path[-stealth]
	(m-1-1) edge [-, densely dashed] (m-1-2)
	(m-1-2)	edge [-, densely dashed] (m-1-3)
	(m-1-3)	edge [-, densely dashed] (m-1-4)
	(m-1-4)	edge [-, densely dashed] (m-1-5)
	(m-1-5)	edge [-, densely dashed] (m-1-6)
	(m-2-1) edge [-, densely dashed] (m-2-2)
	(m-2-2)	edge [-, densely dashed] (m-2-3)
	(m-2-3)	edge [-, densely dashed] (m-2-4)
	(m-2-4)	edge [-, densely dashed] (m-2-5)
	(m-2-5)	edge [-, densely dashed] (m-2-6);
	\end{tikzpicture}
\end{equation*}
where $l_d=\lfloor\frac{d-1}{2}\rfloor$ and $l_e=\lfloor\frac{e-1}{2}\rfloor$. The cycles $\sigma_i$ with $i=1,\ldots, d-1$ and $\gamma_j$ with $j=1,\ldots, e-1$  are the 0-dimensional vanishing cycles associated to $x^{d}$ and $y^e$, respectively. The last vanishing cycle on the right in this Dynkin diagram is $\gamma_{l_e}$ or $\gamma^i_{e-1}$, depending on whether
$e$ is odd or even, respectively (analogously for the last $\sigma_i$).   
	
Then, we consider the  basis given by the join cycles of the vanishing cycles $\gamma_{j}*\sigma_{i}$ where $j=1,\ldots, e-1$ and $i=1,\ldots, d-1$. Furthermore, we  consider the orderings  $\sigma_{l_e+1}>\gamma_1>\sigma_{l_e+2}>\sigma_2>\sigma_{l_e+3}>\sigma_3>\cdots,$ and $\gamma_{l_d+1}>\gamma_1>\gamma_{l_d+2}>\gamma_2>\gamma_{l_d+3}>\gamma_3>\cdots,$ for any $i$. The ordering for the join cycles is  given by
\begin{equation*}
\gamma_{j}*\sigma_{i}>\gamma_{j'}*\sigma_{i'} \text{, if and only if, }\sigma_{i}>\sigma_{i'}\text{, or }i=i' \text{ and }\gamma_{j}>\gamma_{j'}.
\end{equation*}
 We use the ordered basis $( \gamma_{j}*\sigma_i, \hspace{2mm}>)$, in order to write the intersection and monodromy matrices associated to the fibration given by \ref{generalpolfibration}. For example, let $f(x,y)=x^6+y^4$, therefore the ordered basis is $$\gamma_2*\sigma_3, \gamma_1*\sigma_3, \gamma_3*\sigma_3, \gamma_2*\sigma_1, \gamma_1*\sigma_1,\ldots,\gamma_3*\sigma_5.$$ In \S \ref{Monfordirectsumonecrivalue}  we use the notation $\delta_i^j$ to denote the vanishing cycle in the row $i$ and column $j$ of the Dynkin diagram \ref{dynkindeltaij}.   Thus, $\delta_{1}^j=\gamma_{2}*\sigma_{\rho(j)}, \hspace{4mm}
\delta_{2}^j=\gamma_{1}*\sigma_{\rho(j)}, \hspace{4mm}
\delta_{3}^j=\gamma_{3}*\sigma_{\rho(j)},$
where $\rho$ is the permutation $(1,2,3,4,5)\xto{\rho}(3,1,4,2,5)$. 

In order to compute the intersection matrix by using the function MonMatrix, it is enough to write \textit{Im=MonMatrix(m,p)}, where \textit{m} is a vector whose coordinate corresponds to $m(1)=d$ and $m(2)=e$. The parameter $p$ should be a integer number such that: If $p=0$, then $Im$ is the intersection matrix, else  $Im$ is the monodromy matrix. For the previous example, we get the matrix \ref{matrixM4},  by writing the lines
\begin{footnotesize}
\begin{verbatim}
m=[6,4];
Im=MonMatrix(m, 0)
\end{verbatim}
\end{footnotesize}
The function VanCycleSub computes the subspace spanned by the monodromy action of the fibration given by \ref{generalpolfibration}, acting on each vanishing cycle. In other words, this function compute the Krylov space of the monodromy matrix and each one of the the vectors of the basis previously described. That is by computing the eigenvalues and eigenvector of the monodromy matrix as in the proof of Proposition \ref{freevanishingcycles}. The usage of this function is as follows: \textit{[Dim, Wout,Vout]=VanCycleSub(m)}, where the vector $m$ represented again the degrees $d, e$. 

The output \textit{Dim} is the number of different eigenvalues of the monodromy matrix. Note that the  dimension of the homology group $H_n(f^{-1}(b))$, is $N=(d-1)(e-1)$. If $Dim=N$, then the array \textit{Wout} of size $N\times N\times N$, represents the Krylov space of each vanishing cycle. Thus, the columns of the matrix \textit{Wout(:,:,j)} are a basis of the  subspace generated by the monodromy action  on the vector $e_j=(0\hspace{1mm}\cdots \hspace{1mm}0\hspace{1mm}1\hspace{1mm}0\hspace{1mm}\cdots \hspace{1mm}0)$. The vector $e_j$ corresponds with the $j$-th joint cycle according to the previously defined order.

Finally, \textit{Vout} is a matrix where the $j$-th column  is a list of the vanishing cycles in the Krylov subspace of the vector $e_j$ with the monodromy matrix. Actually, this list is the position given by the order, associated to these vanishing cycles. Note that these vanishing cycles correspond to the rows of \textit{Wout(:,:,j)} with a single $1$ and zeros in the others. Continuing the example,
\begin{footnotesize}
	\begin{verbatim}
	m=[6,4];
	[Dim, Wout,Vout]=VanCycleSub(m)
	\end{verbatim}
\end{footnotesize}
In this case the second column of $Vout$ is the list $(2,5,8,11,14)$, which are the positions associated to the vanishing cycles $\delta^k_2$, with $k=1,\ldots,5$ (see Dynkin diagram \ref{dynkindeltaij}). The fifth column is the list $(5,11)$; it is because the vanishing cycle associated to the position 5th and 11th are $\delta^2_2$ and $\delta^4_2$, respectively, and $gcd(d,2)=2$ (see Proposition \ref{freevanishingcycles}).

\clearpage

\begin{footnotesize}
	\bibliographystyle{abbrv}
   \bibliography{References}

\begin{thebibliography}{10}

\bibitem{ACampo}
N.~A'Campo.
\newblock Le groupe de monodromie du d\'eploiement des singularit\'es isol\'ees
  de courbes planes i.
\newblock {\em Math. Ann.}, pages 1--32, 1975.

\bibitem{singularnold}
V.~I. Arnold, A.~N. Varchenko, and S.~Gusein-Zade.
\newblock {\em Singularities of Differentiable Maps: Volume II Monodromy and
  Asymptotic Integrals}, volume~83.
\newblock Springer Science \& Business Media, 1988.

\bibitem{Netoirreducible}
D.~Cerveau and A.~L. Neto.
\newblock Irreducible components of the space of holomorphic foliations of
  degree two in {$\mathbb{CP}(n), n\geq 3$}.
\newblock {\em Annals of mathematics}, pages 577--612, 1996.

\bibitem{ChMonodromy}
C.~Christopher and P.~Marde{\v{s}}i{\'c}.
\newblock The monodromy problem and the tangential center problem.
\newblock {\em Functional analysis and its applications}, 44(1):22--35, 2010.

\bibitem{idealsalgorithms}
D.~Cox, J.~Little, and D.~OShea.
\newblock {\em Ideals, varieties, and algorithms: an introduction to
  computational algebraic geometry and commutative algebra}.
\newblock Springer Science \& Business Media, 2013.

\bibitem{doranmorgan}
C.~Doran and J.~Morgan.
\newblock Mirror symmetry and integral variations of {Hodge} structure
  underlying one parameter families of {Calabi}-{Yau} threefolds.
\newblock {\em V, AMS/IP Studies in Advanced Mathematics}, 38, 2006.

\bibitem{Dulac}
H.~Dulac.
\newblock {\em D{\'e}termination et int{\'e}gration d'une certaine classe
  d'{\'e}quations diff{\'e}rentielles ayant pour point singulier un centre},
  volume~32.
\newblock Gauthier-Villars, 1908.

\bibitem{Fr}
J.~P. Fran{\c{c}}oise.
\newblock Successive derivatives of a first return map, application to the
  study of quadratic vector fields.
\newblock {\em Ergodic theory and Dynamical systems}, 16(1):87--96, 1996.

\bibitem{gavrilovpetrov}
L.~Gavrilov.
\newblock Petrov modules and zeros of {Abelian} integrals.
\newblock {\em Bulletin des sciences mathematiques}, 122(8):571--584, 1998.

\bibitem{GavrilovMovasati}
L.~Gavrilov and H.~Movasati.
\newblock The infinitesimal 16th {Hilbert} problem in dimension zero.
\newblock {\em Bulletin des sciences mathematiques}, 131, 2007.

\bibitem{Il}
Y.~Ilyashenko.
\newblock The origin of limit cycles under perturbation of the equation
  $dw/dz=-r_z/r_w$, where $r(z,w)$ is a polynomial.
\newblock {\em Matematicheskii Sbornik}, 120(3):360--373, 1969.

\bibitem{Lamotke}
K.~Lamotke.
\newblock {The topology of complex projective varieties after {S}.
  {Lefschetz}}.
\newblock {\em Topology}, 20(1):15--51, 1981.

\bibitem{DanielLopezquintic}
D.~L\'opez~G.
\newblock Homology supported in {Lagrangian} submanifolds in mirror quintic
  threefolds.
\newblock {\em Canadian Mathematical Bulletin}, 2020.

\bibitem{hosseinabelian}
H.~Movasati.
\newblock Abelian integrals in holomorphic foliations.
\newblock {\em Revista Matem{\'a}tica Iberoamericana}, 20(1):183--204, 2004.

\bibitem{hosseincenterlog}
H.~Movasati.
\newblock Center conditions: rigidity of logarithmic differential equations.
\newblock {\em Journal of Differential Equations}, 197(1):197--217, 2004.

\bibitem{hodgehossein}
H.~Movasati.
\newblock {\em A Course in {Hodge} theory, with emphasis on multiple
  integrals}.
\newblock
  http://w3.impa.br/\textasciitilde{}hossein/myarticles/hodgetheory.pdf. To be
  published by IP, Boston, 2017.

\bibitem{Netocomponentesbook}
A.~L. Neto.
\newblock {\em Componentes irredut{\'\i}veis dos espa{\c{c}}os de
  folhea{\c{c}}{\~o}es}.
\newblock Publica{\c{c}}oes Matematicas do IMPA, 2007.

\bibitem{netoFwithcenter}
A.~L. Neto.
\newblock Foliations with a morse center.
\newblock {\em J. Singul}, 9:82--100, 2014.

\bibitem{Ro}
R.~Roussarie.
\newblock {\em Bifurcation of planar vector fields and Hilbert's sixteenth
  problem}, volume 164.
\newblock Birkhauser, 1998.

\bibitem{eigenvaluestri}
W.~Yueh.
\newblock Eigenvalues of several tridiagonal matrices.
\newblock {\em Applied mathematics e-notes}, 5(66-74):210--230, 2005.

\bibitem{yadollahcenter}
Y.~Zare.
\newblock Center conditions: pull back of differential equations.
\newblock {\em Transactions of the American Mathematical Society}, 2017.

\end{thebibliography}
\end{footnotesize}

\bigskip

\sf{\noindent Daniel L\'opez Garcia\\
	Instituto de Matematica Pura e Aplicada (IMPA),  \\ 
	Estrada Dona Castorina 110, Rio de Janeiro, 22460-320, RJ, Brazil.\\
	daflopez@impa.br}

\end{document}